\documentclass[a4paper,reqno,11pt]{amsart}

\usepackage[utf8]{inputenc}
\usepackage[T1]{fontenc}
\usepackage{lmodern} 
\usepackage{graphicx}
\usepackage[english]{babel}
\usepackage{amsmath,amsfonts,amssymb,amscd,amsthm, mathtools}
\usepackage{hyperref}
\hypersetup{colorlinks = true}
\hypersetup{linktocpage}
\usepackage[all]{xy}
\usepackage{multibib}
\usepackage{cleveref}
\usepackage[shortlabels]{enumitem}
\usepackage{comment}
\usepackage{bbm}
\usepackage{color}
\usepackage{dsfont}

\usepackage{geometry}
\newtheorem{theorem}{Theorem}[section]
\newtheorem{proposition}{Proposition}[section]
\newtheorem{lemma}{Lemma}[section]
\newtheorem{corollary}{Corollary}[section]
\newtheorem{remark}{Remark}[section]

\theoremstyle{definition}
\newtheorem{definition}{Definition}[section]

\newcommand{\C}{\mathbb C}
\newcommand{\R}{\mathbb R}

\newcommand{\Z}{\mathbb Z}

\newcommand{\N}{\mathbb N}

\newcommand{\D}{\mathcal D}

\newcommand{\E}{\mathbb E}

\renewcommand{\Re}{\mathrm{Re}}
\renewcommand{\Im}{\mathrm{Im}}

\DeclareMathOperator{\Tr}{Tr}
\newcommand{\an}[1]{\langle #1 \rangle}
\newcommand{\la}{\langle}
\newcommand{\ra}{\rangle}

\newcommand{\commentU}[1]{\textcolor{blue}{(#1)}}

\newcommand{\green}[1]{\textcolor{green}{(#1)}}

\title[Strichartz Estimates and Scattering for the Massless Reduced DF Equation]{Strichartz Estimates and Scattering for the Massless Reduced Dirac--Fock Equation}
\author[Elena Danesi]{Elena Danesi}
\address{Politecnico di Torino, Dipartimento di Scienze Matematiche, Corso Duca degli Abruzzi 24, 10129 Turin, Italy}
\email{\href{mailto:elena.danesi@polito.it}{\nolinkurl{elena.danesi@polito.it}}}
\author[Anne-Sophie de Suzzoni]{Anne-Sophie de Suzzoni}
\address{Université Évry Paris-Saclay, LaMME UMR 8071, 23 Boulevard de France, 91037 Évry Cedex, France \& Institut Universitaire de France (IUF), 75005 Paris, France}
\email{\href{mailto:anne-sophie.desuzzoni@univ-evry.fr}{\nolinkurl{anne-sophie.desuzzoni@univ-evry.fr}}}
\author[Umberto Morellini]{Umberto Morellini}
\address{Università di Pisa, Dipartimento di Matematica, Largo Bruno Pontecorvo 5, 56127 Pisa, Italy}
\email{\href{mailto:umberto.morellini@dm.unipi.it}{\nolinkurl{umberto.morellini@dm.unipi.it}}}
\date{16 October 2026}

\begin{document}

\begin{abstract}
    In this paper, we establish Strichartz estimates for the reduced Dirac--Fock equation for 3D massless spinors, linearized around non-localized equilibria. Our analysis covers both the random field and the density matrix frameworks. As an application, we deduce scattering for the corresponding nonlinear problem with a high-frequency cut-off.
\end{abstract}

\maketitle

\tableofcontents

\section{Introduction}

We consider the time-dependent Dirac--Fock equation, describing the effective dynamics of a system of infinitely many relativistic fermions in the mean-field regime. In the standard setting for density matrices, the \emph{reduced} Dirac--Fock equation in $\R^3$ reads
\begin{equation}\label{eq:realDF}
i\partial_t \gamma = [\D + w*\rho_\gamma, \gamma]\,,
\end{equation}
where $\gamma \colon [0,T]  \to \mathcal L (L^2(\R^3,\C^4))$ is a self-adjoint integral operator, $\rho_\gamma$ denotes the $\C^4$--trace of the diagonal of the integral kernel of~\(\gamma\) and $w$ is an even pairwise interaction potential.

In this article, we shall address the case of particles without mass. For this reason, $\D$ is here the massless free Dirac operator acting on $4-$dimensional spinors and defined as
\[
\D = -i\boldsymbol\alpha\cdot \nabla=-i \sum_{j=1}^3 \alpha_j \partial_j\,.
\]
Above, one has $\boldsymbol\alpha=(\alpha_1,\alpha_2,\alpha_3)$ with
\[
\alpha_j = \begin{pmatrix}
    0 & \sigma_j \\ \sigma_j & 0
\end{pmatrix}\,,\quad j=1,2,3\,,
\]
where
\[
\sigma_1=\begin{pmatrix}
0 & 1\\
1 & 0
\end{pmatrix},\quad
\sigma_2=\begin{pmatrix}
0 & -i\\
i & 0
\end{pmatrix}\quad\text{and}\quad
\sigma_3=\begin{pmatrix}
1 & 0\\
0 & -1
\end{pmatrix}
\]
are the standard Pauli matrices. We recall that $\alpha_j$, for $j=1,2,3$, is a Hermitian matrix, i.e. \(\alpha_j^* = \alpha_j\). The operator~\(\D\) is a Fourier multiplier by
\[
\boldsymbol\alpha \cdot \xi = \sum_{j=1}^3 \alpha_j \xi_j\,.
\]
Alternatively, we also say that its symbol is~\(\boldsymbol\alpha \cdot \xi \).

\Cref{eq:realDF} admits steady states, notably Fourier multipliers which in addition commute with the Dirac operator~\(\D\). In this paper, we are interested in the linear stability of these stationary states. This question is analogous to studying the linear stability of certain equilibria in the equivalent formulation of~\eqref{eq:realDF} for random fields, which reads 
\begin{equation}
    \label{eq:DF}
    i \partial_t X = \D X + \Big (w \ast \E [\vert X \vert^2] \Big ) X\,,
\end{equation}
where $X \colon [0,T] \times \R^3 \times \Omega \to \C^4$ is a spinor random field defined over a probability space~$(\Omega, \mathcal A, \mathbb P)$, where~$\E$ denotes the expectation on $\Omega$ and $|X|^2=\la X,X\ra_{\C^4}$ with $\la X,Y \ra_{\C^4}=\sum_{j=1}^4 \overline{X}_j Y_j$. This analogy is explained in~\Cref{subsec:reduction}.

The equivalence between Equations~\eqref{eq:realDF}~and~\eqref{eq:DF} can be seen as follows: on the one hand, given a solution $X$ of~\eqref{eq:DF}, if $\gamma$ is the integral operator with kernel $\gamma(x,y)=\E[\vert X(x)\ra \la X(y)\vert]$ (with a slight abuse of notation, $\gamma$ denotes both the operator and its integral kernel), then $\gamma$ is a solution to~\eqref{eq:realDF}. On the other hand, for any non-negative self-adjoint operator $\gamma = \sum_{i=1}^N |\varphi_i\rangle\langle\varphi_i|$ of finite rank that solves~\eqref{eq:realDF}, the random field defined by $X = \sum_{i=1}^N g_i \varphi_i$ satisfies~\eqref{eq:DF}, with~$g_1, \dots, g_N$ being independent standard Gaussian variables. By applying the spectral theorem, this correspondence straightforwardly generalizes to operators of infinite rank.

While we do not rely exclusively on the random-field approach in this work, this formulation provides a convenient framework that allows for useful analogies with the nonlinear Dirac equation. Indeed, the probability parameter $\omega \in \Omega$ can be viewed as an index coupling this system of (potentially infinitely many) Hartree-Dirac-type equations, given by
\begin{equation*}
i\partial_t X(\omega) = \mathcal{D} X(\omega) + \left( w * \int_{\Omega} |X(\omega')|^2 d\mathbb{P}(\omega') \right) X(\omega)\,.
\end{equation*}
This representation highlights that~\Cref{eq:DF} is not a deterministic PDE to be solved for random initial data, but rather a system of equations coupled through the expectation
\[
\mathbb{E}[|X|^2] = \int_{\Omega} |X(\omega')|^2 d\mathbb{P}(\omega').
\]
Furthermore, this setting is particularly suitable for describing the equilibria, as will be discussed in the following section.

The Dirac-Fock equation has been extensively studied in the time-independent setting. The existence of critical points of the corresponding energy functional was established in \cite{EstSer-1999-CMP, Pat-2000-AHP}. We refer here to critical points rather than minimizers due to the unboundedness from below of the energy functional, a consequence of the negative spectrum of the Dirac operator. This difficulty was addressed in \cite{EstSer-2001-AHP} by analyzing the nonrelativistic limit, which yielded a well-defined electronic ground state energy. More recently, a new definition of the Dirac-Fock ground state was introduced in \cite{Ser-2023-CMP}, a framework that has since been extended to a relativistic model for crystals \cite{CatMenPatSer-2024-ARMA}. In the time-dependent framework, the theory is far less developed. Global-in-time existence for a generalized Dirac-Fock evolution arising from the Bogoliubov-Dirac-Fock model was proved in \cite{HaiLewSpa-2005-LMP}, and recent works have extended this well-posedness result to relativistic electron-nucleus dynamics \cite{Mor-2025-JDE} and to a two-dimensional massless analog relevant for graphene \cite{BM25}.

Despite its limited physical interest, the case of 3D fermions without mass poses nontrivial mathematical challenges and serves as a natural first step for exploring more realistic relativistic systems, ranging from massive fermions in 3D (such as electrons in heavy atoms) to massless 2D fermions in graphene. Such investigations will be the subject of future works.

\subsection{Equilibria}
We endow~\(\mathcal M_{4}(\C)\) with its canonical Hilbert structure. Let \(f\in L^2\left(\R^3,\mathcal M_{4}(\C)\right)\). We denote~\(\gamma_f\) the Fourier multiplier by~\(f(\xi)f(\xi)^*\) and define
\begin{equation}
\label{eq:Yf}
Y_f = \int_{\R^3} e^{-it(\boldsymbol\alpha \cdot \xi + \widetilde m)} e^{ix \cdot \xi} f(\xi) dW(\xi)\,,
\end{equation}
where 
\begin{itemize}
    \item $\widetilde m = \int_{\R^3} w\, dx \, \int_{\R^3} \Tr \big (f^* (\xi) f (\xi) \big) \, d\xi\,$;
    \item $e^{i \boldsymbol\alpha \cdot \xi}$ denotes the exponential matrix of $i\boldsymbol\alpha \cdot \xi \,$; 
    \item $dW(\xi)$ is a vector of infinitesimal centered Gaussian variables characterized by
\begin{equation}
    \label{formula_dW}
    \E \Big ( \vert dW(\eta) \ra \la dW(\xi) \vert \Big ) = \mathbbm 1_{\C^4}\delta (\eta-   \xi) d\eta d\xi\, ,
\end{equation}
the bra-ket notation being in~\(\C^4\).
\end{itemize}
The random field $Y_f$ is a solution of \eqref{eq:DF}. Moreover, if we assume that~\(f(\xi)f(\xi)^*\) commutes with~\(\boldsymbol\alpha\cdot \xi\) for all~\(\xi \in \R^3\), then the covariance of~\(Y_f(t)\) does not depend on~\(t\) and is equal to~\(\gamma_f\). Therefore, the law of~\(Y_f(t)\) does not depend on time and for this reason we say that it is an \emph{equilibrium} of~\eqref{eq:DF}. Equivalently, under the above assumption,~\(\gamma_f\) is a steady state of~\eqref{eq:realDF}. We give a proof of these statements in~\Cref{subsec:studyEq}.\par

In the rest of the paper, we assume that~\(f(\xi)f(\xi)^*\) commutes with~\(\boldsymbol\alpha\cdot\xi\) for all~\(\xi\in \R^3\), and that~\(\xi \mapsto f(\xi)f(\xi)^*\) is even.

The aim of this work is to study the linear stability of~\(Y_f\) under the flow of~\eqref{eq:DF}.

\subsection{Linearized Equation}
We now fix an equilibrium $Y$, dropping the subscript $f$ to lighten the notation. As it may be seen from its covariance operator, $Y$ is non-localized and therefore is not in any Sobolev space. For this reason, the study of the stability of~\(Y\) requires us to focus on the perturbation of the full solution around $Y$, namely $Z=X-Y$. In other words, we consider the following Cauchy problem associated to~\eqref{eq:DF}, that is
\begin{equation}
    \begin{cases}\label{eq:X-cauchy-problem}
        i\partial_t X=\D X+\left(w\ast\E\left[|X|^2\right]\right) X\,,\\
        X_0 = X\big|_{t=0}=Y\big|_{t=0}+Z_0\,,
    \end{cases}
\end{equation}
where the solution will have the form $X=Y+Z$ and the random field $Z$ satisfies
\begin{equation}
    i \partial_t Z = (\D + \widetilde m) Z + \Big (  w \ast \E \big (2 \Re \la Y,Z \ra_{\C^4} + \vert Z \vert^2\big  )\Big )( Y+Z)\,.
\end{equation}\par

Equivalently, the operator~\(\gamma_f\), as a Fourier multiplier, is not trace-class and we consider an equation on the unknown~\(Q = \gamma-\gamma_f\) given by
\[
i\partial_t Q = [\D, Q] + [w*\rho_Q, \gamma_f] + [w*\rho_Q,Q]\,.
\]\par

We focus on the linearized equations around~\(Y\) or~\(\gamma_f\), that is respectively
\begin{align}\label{eq:eqlinZ}
i \partial_t Z = & (\D + \widetilde m) Z + \Big (  w \ast \E \big (2 \Re \la Y,Z \ra_{\C^4} \big  )\Big )Y \,,\\ \label{eq:eqlinQ}
i\partial_t Q = & [\D, Q] + [w*\rho_Q, \gamma_f]\,.
\end{align}
We denote their flows as~\(S_Z\) and~\(S_Q\), respectively.

\subsection{Main Results}

We prove the following Strichartz estimates.

\begin{theorem}\label{th:main}
Assume~\(\an{\,\cdot\,}^{8} \widehat w \in L^\infty(\R^3) \).
Let~\(p,q\in[2,+\infty]\) and \(s\in\R\) be such that
\begin{equation*}
    p>2,\quad \frac1p + \frac1q = \frac12\quad\text{and}\quad s = \frac32 - \frac3q - \frac1p\,.
\end{equation*}
Then, there exists $\varepsilon_0 >0$ such that, for any~\(f\in C^3\left(\R^3,\mathcal M_{4}(\C)\right)\) satisfying
\[
\sum_{k = 0}^3 \|\an{\,\cdot\,}^8 \nabla^{\otimes k} f\|_{L^\infty(\R^3)} < \varepsilon_0\,,
\]
the flow~\(S_Z\) is continuous from~\(L^2_\omega(\dot H^{s-2}_x\cap H^s_x)\) to~\(L^2_\omega L^p_t L^q_x\) and~\(Q_0 \mapsto \rho_{S_Q(t)(Q_0)}\) is continuous from~\(\mathfrak S^{4/3}(H^{1/2}_x) 
\) to~\(L^2_{t,x} 
\) 
\end{theorem}

We now give a nonlinear application.

Assume that~\(\widehat w(\eta) = \frac{|\eta|^{\frac32}}{\an{\eta}^{10}} \widehat w_0 (\eta)\) with~\(\widehat w_0 \in L^\infty\). Let~\(\Pi\) be the Fourier multiplier by the characteristic function of the ball of center~\(0\) and radius~\(\Lambda >0\). Assume that~\(f\) is compactly supported in this ball such that~\(\Pi Y = Y\) and consider the equation
\begin{equation}\label{eq:cutoff}
i\partial_t X = \mathcal D X  + \left(w* \E[|X|^2]\right) \Pi X\,.
\end{equation}
We have that~\(Y\) is a solution to~\eqref{eq:cutoff}. 

Set~\(S(t) = e^{-it(\D + \widetilde m)}\).

\begin{theorem}\label{th:NL} Under the same hypothesis as~\Cref{th:main}, there exists~\(\varepsilon_1 >0\) such that, for any~\(Z_0 \in L^2_\omega(\dot H^{-\frac12}_x \cap H^{\frac12}_x)\) satisfying $\Vert Z_0 \Vert_{L^2_\omega (\dot H^{-\frac12}_x \cap H^\frac 12_x) }\leq \varepsilon_1$,~\Cref{eq:cutoff} admits a unique solution~\(X \in Y + L^2_{\omega}(C(\R,L^2_x(\R^3))\cap L^4_{t,x})\) with $X\big|_{t=0}=Y+Z_0$.\par

Furthermore, there exists~\(Z_\pm \in L^2_{\omega, x}\) such that
\[
X(t) = Y(t) + S(t) Z_\pm + o_{L^2_{\omega,x}}(1)\,,
\]
as~\(t\) goes to~\(\pm \infty\). 
\end{theorem}

\begin{remark}
    The study of scattering for Hartree--Fock equations originated from the question of dispersive estimates for orthogonal systems, which were first proven in~\cite{FLLS} for the Schrödinger dispersion relation. This work was followed by a series of papers aimed at a more systematic study, notably in~\cite{FS, HStrichartzDirac}, and for other dispersion relations, including the Dirac case~\cite{Hos-2024-JFA}.

    The seminal articles~\cite{LS2,LS1} studied the stability of steady states for the reduced Hartree--Fock equation in dimension~\(2\). They were then extended to higher dimension by~\cite{CHP}. The link between the Hartree--Fock equation and the Vlasov equation was already present in~\cite{LS2,LS1} through the Penrose stability condition and was made completely explicit in this context by~\cite{LS3}. 

    Simultaneously, two parallel lines of research emerged: on the one hand, the use of kinetic techniques, notably in~\cite{NY1,NY2}; on the other hand, the application of the full range of dispersive techniques, either through a random field framework due to the second author~\cite{CDdSM25,CdS20,CdS22,M25}, or by establishing functional-analytic properties for orthogonal systems (see~\cite{BHS,Had-2025-AHL}).

    To the best of our knowledge, the study of the stability of steady states has been limited to the reduced Hartree--Fock equation, or the full Hartree--Fock equation with exchange term; see~\cite{NY2}. Consequently, the present work provides the first analysis of equilibrium stability for the Dirac--Fock equation.
\end{remark}

\begin{remark}
Our proof relies on a perturbative argument quite similar to~\cite{CDdSM25}, which accounts for the high regularity required on the steady state given by~\(f\). Moreover, while the Hartree--Fock equation has a dispersion relation that is the archetype of a convex function, for the massless Dirac equation the dispersion is similar to that of the wave equation, resulting in weaker dispersive properties. We compensate for this lack of favorable algebraic structure in the dispersion relation by assuming decay hypotheses (in Fourier space) on both the steady state given by~\(f\) and the interaction potential~\(w\). 

At this stage, there is room for improvement regarding the assumptions on~\(f\) and~\(w\), particularly concerning their decay in Fourier space. However, what the optimal hypotheses would be remains unclear to us.
\end{remark}

\begin{remark}
The nonlinear result in \Cref{th:NL} concerns a simple model which allows us to directly apply standard strategies: the hypothesis on~\(w\) is somewhat similar to~\cite{CHP}, and the high-frequency cut-off is reminiscent of studies on the Bogoliubov--Dirac--Fock model; see~\cite{HLS, HaiLewSer-2005-JPA, HaiLewSol-2007-CPAM} for the statics and \cite{BM25, HaiLewSpa-2005-LMP, Mor-2025-JDE} for the dynamics. Let us mention that another kind of gauge-invariant high-frequency regularization has also been used in the time-independent case \cite{BorMor-2026-CAOT, GraHaiLewSer-2013-ARMA, GraLewSer-2018-JMPA, Mor-2025-AHP, Mor-2026-arx}.

On the one hand, the additional hypotheses on~\(w\) serves to overcome the difficulty stemming from the loss of derivatives at low frequencies in the Strichartz estimates, this amounts to taking the~\(\dot H^{s-2}_x\) norm of the initial datum. For the Hartree--Fock equation, this issue is usually solved in the literature by exploiting the null structure of the nonlinearity; see~\cite{CdS22,LS2}. For the Dirac--Fock equation, however, two drawbacks arise: first, the linear part exhibits an even greater loss at low frequencies than in the Hartree--Fock case; second, given the involved algebraic structure of the Dirac equation, it is far from clear whether a similar null structure can be retrieved.

On the other hand, the high-frequency cut-off is introduced to address the loss of derivatives at high frequencies, a feature specific to the cubic Dirac equation. First, Strichartz estimates for Dirac always lose derivatives (unlike in the Schrödinger case). Second, the numerology of the cubic Dirac equation forces us to use both direct and dual Strichartz estimates to close a fixed-point argument leading to scattering. Consequently, applying standard scattering techniques requires a mechanism to compensate for this double loss of derivatives; see~\cite{CM26}. Several strategies exist in the literature to bypass this difficulty: one approach relies on endpoint Strichartz estimates~\cite{MNNO}, while another employs the~$V$ spaces introduced by~\cite{KTVspaces} (see~\cite{BH17,CH18}). At this stage, the most efficient way seems to be the use of the~\(V\) spaces but we wish to draw the attention of the reader to the fact that this is quite a technical tool.
\end{remark}

\begin{remark}
As usual, Strichartz and Sobolev estimates can be combined to obtain an extended range of Strichartz estimates. To do so, one would require stronger decay on~\(f\). In the same spirit, one could trade decay of~\(\widehat w\) for decay of~\(f\), going down to the hypothesis~\(\an{\,\cdot\,}^6 \widehat w \in L^\infty\). At this stage, we did not find it particularly relevant to play this game. First, our hypothesis on~\(w\) is not sharp and thus pushing the assumption down to~$\an{\,\cdot\,}^6 \widehat w \in L^\infty$ seems of limited interest. Second, while an extended range of Strichartz inequalities is typically used to solve nonlinear problems, our use of the high-frequency cut-off (as explained above) makes such a range secondary. To address the full nonlinear problem, one would need to tailor the Strichartz estimates to~\(V\) spaces, which constitutes a different issue.
\end{remark}

\subsection{Sketch of the Proof and Further Remarks}

The proof proceeds as follows. First, in~\Cref{subsec:reduction}, we reduce the linear problem to studying the equation
\[
V = V_0 + L_2 (V)\,,
\]
where~\(V : \R \times \R^3 \rightarrow \R\) is the unknown and~\(L_2\) is a linear operator. Here, the definitions of~$V$ and~$V_0$ depend on the chosen framework:
\begin{itemize}
\item In the random field setting, $V = 2\mathrm{Re}\E(\an{Y,Z})$ and $V_0 = 2 \mathrm{Re}\E(\an{Y,S(t)Z_0})$;
\item In the density matrix setting, $V = \rho_Q$ and $V_0 = \rho_{S(t) Q_0 S(-t)}$.
\end{itemize}

We then prove that~\(L_2\) is a Fourier multiplier which is continuous from~\(L^2_{t,x}\) to itself and whose norm is controlled by some norms of~\(f\) and~\(w\). Then, assuming that~\(f\) is small in this norm, one can invert~\(1-L_2\). Note that this strategy is similar to~\cite{LS2} but is made more involved by the Dirac operator. In particular, the proof relies crucially on~\Cref{stat:h-cos-decomposition} which makes use of the oscillatory character of the integral defining~\(L_2\). 

In the case of random fields, we then use that
\[
Z = S(t)Z_0 + L_1(1-L_2)^{-1}V_0\,,
\]
with
\[
L_1 (V) = -i \int_0^t S(t-s) [w*V(s) Y(s)] ds\,.
\]
Finally, we use a~\(TT^*\) argument due to~\cite{CdS22} to estimate~\(L_1\) in~\(L^p_tL^q_x\)--spaces. We use the same argument to estimate~\(V_0\) thanks to~\(Z_0\) and this concludes the proof of~\Cref{th:main}.

In the case of density matrices, we use that
\[
\rho_Q = (1-L_2)^{-1}\rho_{S(t)Q_0S(-t)} .
\]
We then use 
Strichartz estimates for orthonormal systems to get~\(L^{p/2}_tL^{q/2}_x\)--bounds (see Lemma~\ref{lemma:Str_est}) on~\(\rho_{S(t)Q_0S(-t)}\) with~\(p=q=4\).

For the nonlinear part, we use the decomposition of~\Cref{subsec:reduction}. The linear part is written in such a way that one may overcome the singularities in low frequencies using the particular structure of~\(w\). Therefore, it remains to prove bilinear estimates for the nonlinear part.

\begin{remark} Note that the core of the proof, namely~\Cref{stat:h-cos-decomposition} and~\Cref{stat:h-sin-decomposition}, should find a counterpart in higher dimension. In two dimensions, however, it is doubtful that an analog of~\Cref{stat:h-cos-decomposition} holds. Nevertheless, the specific algebra of the Dirac operator in dimension~\(2\) might help resolve this issue.
\end{remark}

\begin{remark} Note that in the case of density matrices, one has
\[
\rho_Q = \rho_{S(t)Q_0S(-t)}  + L_2 (1-L_2)^{-1}\rho_{S(t)Q_0S(-t)} \,.
\]
Therefore,~\(\rho_Q\) decomposes as a sum of a function~\(\rho_{S(t)Q_0S(-t)}\), which belongs to~\(L^{p/2}_t L^{q/2}_x\)--spaces thanks to standard Strichartz estimates, and of~\(L_2 (1-L_2)^{-1}\rho_{S(t)Q_0S(-t)}\). With a slight adjustment of the proof of~\Cref{stat:h-cos-decomposition} and with some more decay on~\(\widehat w\), one could get that~\(\widetilde F_\pm\) (see the notation of~\Cref{stat:h-cos-decomposition}) belongs to~\(L^{(2+p)/2p}_t L^{2q/(q-2)}_\eta\), which would imply that~\(L_2\) is bounded from~\(L^2_{t,x}\) to~\(L^p_t L^q_x\) (for all~\(p,q \geq 2\)). This would ensure that
\[
L_2 (1-L_2)^{-1}\rho_{S(t)Q_0S(-t)} \in L^p_t L^q_x \,.
\]
This might be useful in the study of the nonlinear problem. However, we have that
\[
L^{p_1/2}_t L^{q_1/2}_x = L^{p_2}_t L^{q_2}_x \,,
\]
with~\(p_1,q_1\) in the range of the Strichartz estimates and~\(p_2,q_2\geq 2\), if and only if~\(p_2=q_2 = 2\) and~\(p_1 = q_1 = 4\).
\end{remark}

\begin{remark}
We now compare the dispersion relation of the half-wave equation,~\(\omega_0(\xi) = |\xi|\), and that of the Dirac--Fock equation,~\(\omega_1(\xi) = \alpha\cdot \xi\). Observing the proof of~\Cref{prop:L_2}, one may see that, due to the fact that~\(\omega
_1\) is odd, the oscillatory part of the integral defining~\(L_2\) is given by
\[
[e^{it\alpha\cdot \xi}, e^{-it\alpha \cdot (\xi + \eta)}]\,,
\]
while if the Dirac dispersion was replaced by the half-wave one, since~\(\omega_0\) is even, the oscillatory part would be given by
\[
\{e^{it |\xi|} , e^{-it |\xi + \eta|} \}  = 2e^{it(|\xi| - |\xi+\eta|)}\,.
\]
Now, these two oscillatory factors oscillate in the same way but they have different behaviors at~\(\eta = 0\). In fact, for~\(\xi\) and~\(t\) fixed, we have 
\[
[e^{it\alpha\cdot \xi}, e^{-it\alpha \cdot (\xi + \eta)}] \sim_{\eta\to 0} |\eta|\,
\]
while~\(e^{it(|\xi| - |\xi+\eta|)} \sim 1\). This is crucial to conclude the proof of~\Cref{prop:L2-continuity-of-L2}. 

In the case of the Hartree--Fock equation, the oscillatory part is given by
\[
e^{-it2\xi \cdot \eta}\,,
\]
which oscillates much more: different mechanisms are at play!
\end{remark}

\subsection{Notation}\label{sec:notation}

\begin{itemize}
\item We denote~\(\an{x} = \sqrt{1+|x|^2}\), for any~\(x\in \R^3\).
    \item We adopt the definition of the Fourier transform that acts as an isometry on $L^2(\R^3)$, namely
    \[
    \widehat u (\xi) = \frac1{(2\pi)^{3/2}} \int_{\R^3} u(x) e^{-ix\cdot \xi} dx\,.
    \]
    \item For $p,q \in [1, +\infty]$, we denote by $L^p_t L^q_x$ the space
    \[
    L^p(\R; L^q(\R^3))\,.
    \]
    \item For $\nu\in\R$ and $s>0$, we denote by \(\dot H^\nu_{-s,x}\) the space
    \[
    (1-\Delta)^{s/2}\dot H_x^\nu\,.
    \]
    \item For $p \in [1, +\infty]$, we denote $\mathfrak S^p(\mathcal{H})$ the $p$--Schatten class over the Hilbert space $\mathcal{H}$, defined as
    \[
    \{Q\in\mathcal{L}(\mathcal{H});\;\Tr(|Q|^p)^{\frac{1}{p}}<+\infty\}\,,
    \]
    endowed with the norm
    \[
    \|Q\|_{\mathfrak{S}^p(\mathcal{H})}=\Tr(|Q|^p)^{\frac{1}{p}}\,.
    \]
    Notice that $\mathfrak S^p$ coincides with the set of trace-class operators for $p=1$ and of Hilbert-Schmidt operators for $p=2$.
    \item For $q \in [2, +\infty]$, $s \in \R$, we denote by $\| \cdot \|_{ W^{s,q}_x}$ the norm 
    \[
    \| f \|_{W^{s,q}_x} = \| \an{D}^s f \|_{L^q(\R^3)}\,,
    \]
    with $\an{D} = (1 -\Delta)^\frac 12$. Notice that the principal symbol of $\an{D}$ is given by $\an{\xi}$.
\end{itemize}

\section{Remarks on the Equilibrium and Reduction of the Problem}

\subsection{Study of the Equilibrium}\label{subsec:studyEq}

\begin{proposition}
The random field $Y_f$, defined in \eqref{eq:Yf}, is a solution of \eqref{eq:DF}. 
\end{proposition}

\begin{proof}
Indeed,
\begin{align*}
     & \E \Big ( \la Y_f, Y_f \ra_{\C^4} \Big ) = \E \bigg( \int_{\R^3} \int_{\R^3} \la e^{-it (\boldsymbol\alpha \cdot \xi + \widetilde m)} e^{i x \cdot \xi} f(\xi) dW(\xi), e^{-it (\boldsymbol\alpha \cdot \xi' + \widetilde m)} e^{i x \cdot \xi'} f(\xi') dW(\xi') \ra_{\C^4} \bigg )\\
     &\hspace{1.5cm} = \int_{\R^3} \int_{\R^3} \Tr \bigg( f^*(\xi') e^{-i x \cdot \xi'} e^{it(\boldsymbol\alpha \cdot \xi' + \widetilde m)}  e^{-it(\boldsymbol\alpha \cdot \xi + \widetilde m)}  e^{i x\cdot \xi} f(\xi) \bigg )  \E \Big ( \vert dW(\xi) \ra \la dW(\xi') \vert \Big ) \\
     &\hspace{1.5cm} = \int_{\R^3} \int_{\R^3} \Tr \bigg( f^*(\xi') e^{-i x \cdot \xi'} e^{it(\boldsymbol\alpha \cdot \xi' + \widetilde m)}  e^{-it(\boldsymbol\alpha \cdot \xi + \widetilde m)}  e^{i x\cdot \xi} f(\xi) \bigg ) \delta (\xi - \xi') d\xi d\xi'\\
     &\hspace{1.5cm} = \int_{\R^3} \Tr \big (f^* (\xi) f (\xi) \big)=\|f\|^2_{L^2\left(\R^3,\mathcal M_{4}(\C)\right)}\,.
\end{align*}
Therefore, recalling the definition of $\widetilde m$ given above, we have
\begin{align*}
     i \partial_t Y_f &= i \int_{\R^3} -i (\boldsymbol\alpha \cdot \xi + \widetilde m) e^{-it(\boldsymbol\alpha \cdot \xi + \widetilde m)} e^{ix\cdot \xi} f(\xi) dW(\xi) \\
     &= \Big ( \D + w \ast \E [\vert Y_f \vert^2] \Big ) Y_f\,.
\end{align*}
\end{proof}

\begin{proposition} The following statements are equivalent: \begin{enumerate}
    \item the operator~\(\gamma_f\) commutes with~\(\D\);
    \item for almost all~\(\xi \in \R^3\), the operator~\(\boldsymbol \alpha\cdot \xi\) commutes with~\(f(\xi)f(\xi)^*\);
    \item the operator~\(\gamma_f\) is a steady state of~\eqref{eq:realDF};
    \item for all~\(t\in \R\), the covariance of~\(Y_f\) is~\(\gamma_f\).
\end{enumerate}
\end{proposition}

\begin{proof}
The equivalence between Assertion~(1) and Assertion~(2) is due to the fact that both~\(\gamma_f\) and~\(\D\) are Fourier multipliers, by respectively~\(f(\xi)f(\xi)^*\) and~\(\alpha \cdot \xi\).

The equivalence between Assertion~(1) and Assertion~(3) is due to the fact that~\(\rho_{\gamma_f}\) is the constant
\[
\int \Tr_{\C^4}(f(\xi) f(\xi)^*)\,d\xi\,.
\]
Therefore, we have
\[
[\D + w*\rho_{\gamma_f}, \gamma_f] = [\D,\gamma_f]\,.
\]
We deduce that~\(\gamma_f\) is a steady-state of~\eqref{eq:realDF} if and only if~\([\D,\gamma_f] = 0\).

Finally, we prove the equivalence between Assertion~(2) and Assertion~(4). Let~\(\gamma (t)\) be the covariance of~\(Y(t)\), namely the operator with integral kernel given by
\[
 \E \Big ( \vert Y(t,x) \ra \la Y(t,y) \vert \Big)\, .
\]
Recall that
\begin{equation}\label{eq:mass-Dirac-propagator}
    S(t) = e^{-it(\D + \widetilde m)}
\end{equation}
and set
\[
Y_f(0)= Y_f \big \vert_{t=0} = \int_{\R^3} e^{ix \cdot \xi} f(\xi) dW(\xi)\,.
\]
We easily have $Y_f(t) = S(t) Y_f(0)$. Then, we compute
\begin{align*}
    &\an{w,\E \Big ( \vert S(t) Y_f(0) \ra \la S(t) Y_f(0) \vert \Big ) v}\\
    &\hspace{3cm} = \E \Big ( \int_{\R^3}  \an{S(-t) w ,e^{ix \cdot \xi} f(\xi) dW(\xi)}_{L^2_x}  \int_{\R^3}\la   e^{iy \cdot \xi'}  f(\xi') dW(\xi'), S(-t) v\ra_{L^2_y} \Big ) \\
    &\hspace{3cm} = \E \Big ( \int_{\R^3}  \an{e^{it\boldsymbol\alpha\cdot\xi}\widehat w(\xi) , f(\xi) dW(\xi)}_{\C^4}  \int_{\R^3}\la     f(\xi') dW(\xi'), e^{it\boldsymbol\alpha\cdot\xi'}\widehat v(\xi')\ra_{\C^4} \Big ) \\
    &\hspace{3cm} = \int_{\R^3} \int_{\R^3}  \la f(\xi)^* \widehat w(\xi),\E \Big ( | dW(\xi) \ra \la dW(\xi')| \Big) f(\xi')  \widehat v(\xi') \ra_{\C^4}\\
    &\hspace{3cm} = \int_{\R^3}   \la f(\xi)^* e^{it\boldsymbol\alpha\cdot\xi} \widehat w(\xi), f(\xi) e^{it\boldsymbol\alpha\cdot\xi} \widehat v(\xi) \ra_{\C^4} d\xi \\
    &\hspace{3cm} = \int_{\R^3}  \an{\widehat w(\xi), e^{-it\boldsymbol\alpha\cdot\xi}f(\xi) f(\xi)^* e^{it\boldsymbol\alpha\cdot\xi} \widehat v(\xi)}_{\C^4} d\xi\,,
\end{align*}
meaning that the operator $\gamma(t)$ is a Fourier multiplier with symbol given by
\[
e^{- it\boldsymbol\alpha\cdot\xi } f(\xi)f(\xi)^* e^{it\boldsymbol\alpha\cdot\xi}\,.
\]
Therefore,~\(\gamma(0)\) is~\(\gamma_f\) and~\(\gamma(t)\) does not depend on~\(t\) if and only if the operator~\(f(\xi)f(\xi)^*\) coincide with~\( e^{-it\boldsymbol\alpha\cdot\xi}f(\xi) f(\xi)^* e^{it\boldsymbol\alpha\cdot\xi}\), for almost all $\xi\in\R^3$ and all $t\in\R$, or equivalently, if and only if~\(f(\xi)f(\xi)^*\) commutes with~\(\boldsymbol\alpha\cdot \xi\) almost surely in~\(\xi\).
\end{proof}

\subsection{Reduction of the Problem}\label{subsec:reduction}

Recall that 
we are interested in the linear stability of~\(Y\), in the sense that we wish to prove Strichartz estimates for~\Cref{eq:eqlinZ}.

\begin{proposition}\label{stat:reduced-Z-problem}
Consider the following operators
\begin{align}
\label{eq:defL1} L_1(V)(t) = &-i \int_0^t S(t-s)[w*V(s) Y(s)]\,ds\,, \\
\label{eq:defL2} L_2(V) = & \E \big (2 \Re \la Y,L_1(V) \ra_{\C^4})\,, \\
\label{eq:defL3} L_3(Z_0)(t) = & \E \big (2 \Re \la Y(t),S(t) Z_0\ra_{\C^4}\big )\,,
\end{align}
and assume that~\(1-L_2\) is invertible.

The solution~\(Z\) to~\eqref{eq:eqlinZ} with initial datum~\(Z_0\) satisfies
\[
Z(t) = S(t) Z_0 + L_1(1-L_2)^{-1} L_3(Z_0)(t)\,. 
\]
\end{proposition}

\begin{proof}
Notice that one can write~\(Z\) as the solution to the following fixed point:
\[
Z = S(t)Z_0 - i \int_0^t S(t-s)[w\ast \E \big (2 \Re \la Y(s),Z (s)\ra_{\C^4} \big ) Y(s)]\, ds\,.
\]
Set~\(V(s) =\E \big (2 \Re \la Y(s),Z(s) \ra_{\C^4} \big ) \) and~\(L_1,L_2\) the operators respectively defined in \eqref{eq:defL1} and \eqref{eq:defL2}. 

Then, we get that~\(V\) satisfies the following fixed point
\[
V = L_3(Z_0) + L_2(V)\,,
\]
and that $Z$ fulfills
\[
Z (t)= S(t)Z_0 + L_1(V)(t)\,.
\]
Assuming that~\(1-L_2\) is invertible, we have $V=(1-L_2)^{-1}L_3(Z_0)$ which yields
\[
Z(t) = S(t) Z_0 + L_1(1-L_2)^{-1} L_3(Z_0)(t)\,.
\]
\end{proof}

Consider now the linearization of~\Cref{eq:realDF} around~\(\gamma_f\), that is~\eqref{eq:eqlinQ}, with initial datum~\(Q_0\). To see the equivalence between~\eqref{eq:eqlinZ} and~\eqref{eq:eqlinQ}, notice that if~\(Z\) solves~\eqref{eq:eqlinZ}, then~\(\iota(Z)\), defined as the time-dependent field of operators such that~\(\iota(Z)(t)\) has an integral kernel given by
\begin{equation*}
    \E(| Y(t,x) \rangle \langle Z(t,y)|) + \E(| Z(t,x) \rangle \langle Y(t,y)|)\,,
\end{equation*}
solves~\eqref{eq:eqlinQ}. 

\begin{proposition}\label{stat:reduced-Q-problem}
Let~\(Q\) be the solution to~\eqref{eq:eqlinQ} with initial datum~\(Q_0\) and assume that~\(1-L_2\) is invertible, then, we have
\[
\rho_{Q(t)} = (1-L_2)^{-1} \rho_{S(t) Q_0 S(-t)}\,.
\]
\end{proposition}

\begin{proof}
Notice that, by Duhamel's principle,~\(Q\) satisfies 
\[
Q(t) = S(t) Q_0 S(-t) - i \int_0^t S(t-s)[w*\rho_Q(s) , \gamma_f] S(s-t) \, ds
\]
such that~\(\rho_Q\) is solution to
\[
\rho_Q(t) = \rho_{S(t) Q_0 S(-t)} - i \int_0^t \rho\Big[S(t-s)[w*\rho_Q(s) , \gamma_f] S(s-t) \Big] \, ds\,,
\]
where~\(\rho[Q] = \rho_Q\). Set
\[
\widetilde L_2 (V) = - i \int_0^t \rho\Big[S(t-s)[w*V, \gamma_f] S(s-t) \Big]\,ds\, .
\]
We now prove that~\(\widetilde L_2 = L_2\). We focus first on
\[
\iota(L_1(V))\,.
\]
By \eqref{eq:defL1}, the integral kernel of~\(\iota(L_1(V))(t)\) is
\[
-i \int_0^t \Big[ \E(| Y(t,x) \rangle \langle S(t-s) [w*V(s) Y(s)](y))|) + \E(| S(t-s) [w*V(s) Y(s)](x) \rangle \langle Y(t,y)|) \Big]\, ds\, .
\]
Using that $| Y(t,x) \rangle=S(t-s)|Y(s,x)\rangle$ and that $S$ is a deterministic propagator, we deduce
\[
\iota(L_1(V))(t) = -i \int_0^t S(t-s) A(s) S(s-t)\, ds\,,
\]
where~\(A(s)\) is the operator whose integral kernel is defined as
\[
\E(| Y(s,x) \rangle \langle [w*V(s)(y) Y(s,y)]|) + \E(| [w*V(s)(x) Y(s,x)] \rangle \langle Y(s,y)|)\,,
\]
that is
\[
w*V(s) \gamma_f  + \gamma_f w* V(s) = [w*V(s), \gamma_f]\,. 
\]
Therefore, we have
\[
\iota(L_1(V)) = - i \int_0^t S(t-s)[w*V(s) , \gamma_f] S(s-t) \, ds\,.
\]
Noticing that~\(\rho[\iota(Z)(t)]= 2\mathrm{Re}\E(\an{Y(t),Z(t)} )\), we finally get
\[
\widetilde L_2(V) = \rho[\iota(L_1(V)(t)] = 2\Re\E(\an{Y(t),L_1(V)}) = L_2(V)\,.
\]
\end{proof}

We now focus on the nonlinear model given by~\Cref{eq:cutoff}.

\begin{proposition}\label{stat:defQLtilde}
Let~\(L_0\) be the Fourier multiplier by~\(\eta\mapsto \frac{|\eta|^{\frac34}}{\an{\eta}^5}\) and 
\begin{align*}
\widetilde L_1(V) = & -i \int_0^t S(t-s) \left[w_0*(L_0 V(s)) Y(s)\right]\, ds\,, \\ \nonumber
\widetilde L_2(V) = & L_0 \left( 2\mathrm{Re}\E(\an{Y,\widetilde L_1(V)})\right)\,, \\ \nonumber
\widetilde L_3 (Z_0) = & L_0 \left( 2\mathrm{Re}\E(\an{Y,S(t)Z_0}) \right)\,, \\ \nonumber
\widetilde Q_1(Z,V) = & -i \int_0^t S(t-s) \left[ w_0*(L_0 V(s)) \Pi Z(s)\right]\, ds\,, \\ \nonumber
\widetilde Q_2(Z,V) = & L_0 \left( 2\mathrm{Re}\E(\an{Y,\widetilde Q_1(Z,V)})\right) \, .
\end{align*}
Then, solving the Cauchy problem associated to~\eqref{eq:cutoff} with initial datum~\(X_0=Y_f(0) + Z_0\) reduces to solving the fixed-point equation
\begin{equation}\label{eq:fixedpointCO}
\begin{cases}
    Z = S(t) Z_0 + \widetilde L_1(V) + \widetilde Q_1 (Z,V)\,, \\
    V = \widetilde L_3(Z_0) + \widetilde L_2 (V) + \widetilde Q_2(Z,V) + L_0\E(|Z|^2)\,.
\end{cases}
\end{equation}
\end{proposition}

\begin{proof} Equation~\eqref{eq:cutoff} linearized around~\(Y\) reads
\begin{equation}\label{eq:cutoffZ}
i\partial_t Z = (\D+m)Z + (w*\E(\mathrm{Re}(\an{Y,Z})+ |Z|^2))(Y + \Pi Z)\,.
\end{equation}

We introduce the function 
\[
V = \E \big (2 \Re \la Y,Z \ra_{\C^4} + \vert Z \vert^2 \big)
\]
so that $w_0 \ast L_0 V$ is the mean-field potential induced by the perturbation. Recalling $S(t)$ defined in \eqref{eq:mass-Dirac-propagator}, the Duhamel formula yields
\begin{equation}
\label{eq:Duhamel_Z}
Z = S(t) Z_0 -i \int_0^t S(t-\tau) \big [w_0 \ast L_0 V (\tau) \big ] (Y(\tau)+\Pi Z(\tau))\, d\tau\,,
\end{equation}
which implies
\begin{multline*}
V = 2 L_0 \Re \E \big ( \la Y, S(t) Z_0 \ra_{\C^4} \big ) - 2L_0\Re \E \Big ( \la Y,i \int_0^t S(t-\tau) [w_0 \ast L_0 V] Y\,d\tau\ra_{\C^4}  \Big ) + 
\\
- 2L_0\Re \E \Big ( \la Y, i\int_0^t S(t-\tau) [w_0 \ast L_0 V] \Pi Z\,d\tau \ra_{\C^4} \Big ) + L_0\E \big ( \vert Z \vert ^2 \big)\,.
\end{multline*}\par

This allows us to rewrite the initial Cauchy problem as a fixed-point problem for the pair $(Z,V)$:
\begin{equation}\label{eq:fixed-point}
    \begin{pmatrix}
    Z \\ V
\end{pmatrix}
= C_{Z_0}+\widetilde L \begin{pmatrix}
    Z\\V
\end{pmatrix} + \widetilde Q \begin{pmatrix}
    Z\\V
\end{pmatrix}\,
\end{equation}
with
\[
C_{Z_0} = \begin{pmatrix}
    S(t) Z_0\\
    2L_0 \Re \E \big ( \la Y, S(t) Z_0 \ra_{\C^4} \big )
\end{pmatrix}\,, 
\]
\begin{equation}
    \label{eq:linear}
    \widetilde L \begin{pmatrix}
    Z \\ V
\end{pmatrix} = \begin{pmatrix}
    -i \int_0^t S(t-\tau) \big [w_0 \ast L_0 V (\tau) \big ]Y d\tau \\
    - 2L_0 \Re \E \Big (i \la Y, \int_0^t S(t-\tau) [w_0 \ast L_0 V] Yd\tau\ra_{\C^4}  \Big )
\end{pmatrix}\,,
\end{equation}
and 
\[
\widetilde Q \begin{pmatrix}
    Z \\ V
\end{pmatrix} = \begin{pmatrix}-i \int_0^t S(t-\tau) \big [w_0 \ast L_0 V (\tau) \big ] \Pi Z d\tau\\ 2L_0 \Re \E \Big (i \la Y, \int_0^t S(t-\tau) [w_0 \ast L_0 V] \Pi Zd\tau \ra_{\C^4} \Big ) + L_0 \E \big ( \vert Z \vert ^2 \big)
 \end{pmatrix}\,.
\]
\end{proof}

\begin{remark}
Finally, if~\(1-\widetilde L_2\) is invertible, so is~\(1-\widetilde L\) and the fixed-point equation in \eqref{eq:fixed-point} can be easily rewritten as
\begin{equation}\label{eq:fixed-point-CLQ}
    \begin{pmatrix}
    Z \\ V
\end{pmatrix} = (1 -\widetilde L)^{-1} \Bigg (  C_{Z_0} + \widetilde Q \begin{pmatrix}
    Z \\ V
\end{pmatrix} \Bigg )\,.
\end{equation}
\end{remark}

\section{Exact Computations of the Linear Response}

This section is devoted to preparatory computations for the study of the operators~\(L_1,L_2\) and~\(L_3\) previously defined in \eqref{eq:defL1}--\eqref{eq:defL3}. More precisely, we prove that~\(L_2\) is a Fourier multiplier by a function that we explicitly compute; we prove that one can bound~\(L_1\) by estimating the composition of the adjoint of~\(L_3\) with a Fourier multiplier; we explain how a~\(TT^*\) argument allows to estimate~\(L_1\) and~\(L_3\).\par


\subsection{Exact Computation of \texorpdfstring{\(L_2\)}{L2}}\label{subsec:exactL2}

\begin{proposition}
\label{prop:L_2}
    The operator $L_2$ is a multiplication operator given in Fourier space, for any $t \in \R$ and $\eta \in \R^3$, by
    \begin{multline}
    \label{eq:L_2(V)}
    \widehat{L_2(V)} (t, \eta)\\
    =- i \widehat w (\eta) \int_0^t \widehat V(\tau, \eta) \int \Tr \Big [  \Big(  f(\xi)f(\xi)^* - f(\eta + \xi)f(\eta + \xi)^* \Big ) e^{i(t-\tau) \alpha \cdot \xi} e^{- i (t-\tau) \alpha \cdot (\xi + \eta)} \Big ]\,d\xi\, d \tau\,.
    \end{multline}
\end{proposition}
\begin{proof}
Recalling the definition of $S(t)$ in \eqref{eq:mass-Dirac-propagator} and of $Y$ in \eqref{eq:Yf}, we have 
\begin{align*}
  &\int_0^t S(t-\tau) \big [w \ast V (\tau) \big ]Y\,d\tau  \\
  &\hspace{1cm} = \frac{1}{\left(2\pi\right)^{3/2}}\int_0^t \iint e^{i x \cdot \eta} e^{-i(t-\tau) (\boldsymbol\alpha \cdot \eta + \widetilde m)} e^{-i \tau (\boldsymbol\alpha \cdot \eta + \widetilde m)}\mathcal F_x \big ( (w \ast V(\tau)) e^{i x\cdot \xi} \big )(\eta)\, f(\xi)\, d\eta\, dW(\xi)\, d \tau\\
  &\hspace{1cm} = \frac{1}{\left(2\pi\right)^{3/2}}\int_0^t \iint e^{i x \cdot \eta} e^{-i(t-\tau) (\boldsymbol\alpha \cdot \xi + \widetilde m)} e^{-i \tau (\boldsymbol\alpha \cdot \xi + \widetilde m)} \widehat w(\eta - \xi) \widehat V(\tau, \eta-\xi)\, f(\xi)\,d\eta\,  dW(\xi)\, d \tau\\
  &\hspace{1cm} = \frac{1}{\left(2\pi\right)^{3/2}}\int_0^t \iint  e^{-i(t-\tau) (\boldsymbol\alpha \cdot (\xi + \eta) + \widetilde m)} e^{-i \tau (\alpha \cdot \xi + \widetilde m)} e^{ix\cdot (\eta + \xi)} \widehat w(\eta) \widehat V(\tau, \eta)\, f(\xi)\, d\eta\, dW(\xi)\, d\tau\,.
\end{align*}
In the last equality, we have performed the change of variable~\(\eta \leftarrow \eta - \xi\). Notice that here we use $\widehat g$ to denote the Fourier transform of a function $g$ with respect to the variable $x$. Then, recalling \eqref{formula_dW}, we get
\begin{multline*}
-i \E \Big ( \la Y, \int_0^t S(t-\tau) [w \ast V] Yd\tau\ra_{\C^4}  \Big ) = \E(\la Y, L_1(V) \ra)\\
= -\frac{i}{\left(2\pi\right)^{3/2}} \int_0^t \iiint \Tr_{\C^4} \Big [ \Big ( e^{-it (\boldsymbol\alpha \cdot \xi' + \widetilde m)} e^{i x \cdot \xi'} f(\xi') \Big )^* \Big ( e^{-i(t-\tau) (\boldsymbol\alpha \cdot (\eta + \xi) + \widetilde m)} e^{-i \tau(\boldsymbol\alpha \cdot \xi + \widetilde m)} f(\xi)\Big) \Big ]\\
\times
e^{ix\cdot(\eta + \xi)} \widehat w(\eta) \widehat V (\tau,\eta) d\eta\, \delta(\xi - \xi')\, d\xi'\, d\xi\, d \tau\,,
\end{multline*}
which equals to
\[
-\frac{i}{\left(2\pi\right)^{3/2}} \int_0^t \iint \Tr_{\C^4} \Big [ f(\xi) f(\xi)^* e^{it \boldsymbol\alpha \cdot \xi}e^{-i(t-\tau) \boldsymbol\alpha \cdot (\eta + \xi)} e^{-i \tau\boldsymbol\alpha \cdot \xi} ] e^{i x \cdot \eta} \widehat w(\eta) \widehat V(\tau, \eta)\, d\eta\, d\xi\, d \tau\,.
\]
Using that $[ f(\xi)f^*(\xi), e^{it \boldsymbol\alpha \cdot \xi}] = 0$, we can write 
\begin{align*}
   -i \E \Big ( \la Y, \int_0^t S(t-\tau) [w \ast V] Y\, d\tau\ra_{\C^4}  \Big ) = -\frac{i}{\left(2\pi\right)^{3/2}} \int_0^t \int e^{i x \cdot \eta} \widehat w (\eta) \widehat V(\tau, \eta) M(t-\tau, \eta) \, d\eta\, d\tau\,,
\end{align*}
where
\[
M(t,\eta)\coloneqq \int \Tr_{\C^4} \Big [ f(\xi) f^*(\xi) e^{it \boldsymbol\alpha \cdot \xi} e^{-it \boldsymbol\alpha \cdot (\eta + \xi)}\Big ] \,d\xi\,.
\]
To compute $L_2(V)$, we recall that
\begin{equation*}
    L_2(V)=2\Re\E\left[\la Y,L_1\left(V\right)\ra_{\C^4}\right] = \E\left[\la Y,L_1\left(V\right)\ra_{\C^4}\right]+\overline{\E\left[\la Y,L_1\left(V\right)\ra_{\C^4}\right]}\,.
\end{equation*}
Therefore, using that $\overline{\Tr\left(A\right)}=\Tr\left(A^*\right)$, for any matrix $A$, and that $\widehat{\overline{\varphi}}\left(\eta\right)=\overline{\widehat{\varphi}\left(-\eta\right)}$, for any complex function $\varphi$, we get
\begin{multline*}
   \widehat{L_2(V)}(t,\eta) = -\frac{i}{\left(2\pi\right)^{3/2}} \int_0^t \int e^{i x \cdot \eta} \widehat w (\eta) \widehat V(\tau, \eta) M(t-\tau, \eta)\, d\eta\, d\tau \\
   +\frac{i}{\left(2\pi\right)^{3/2}} \int_0^t \int e^{i x \cdot \eta} \overline{\widehat w (-\eta)} \overline{\widehat V(\tau, -\eta) }\overline{M(t-\tau, -\eta)}\, d\eta\, d\tau\,, 
\end{multline*}
which implies
\[
\widehat{L_2(V)}(t,\eta) = -\frac{i}{\left(2\pi\right)^{3/2}} \int_0^t \int e^{i x \cdot \eta} \widehat w (\eta) \widehat V(\tau, \eta) (M(t-\tau, \eta) - \overline{M(t-\tau, -\eta)})\, d\eta\, d\tau\,,
\]
since~\(w\) and~\(V\) are real-valued. Using the cyclicity of the trace, we compute
\begin{align*}
    \overline{M(t, -\eta)} = \int \Tr \Big [ f(\xi) f^*(\xi) e^{it \boldsymbol\alpha \cdot (\xi - \eta)} e^{-it \boldsymbol\alpha \cdot \xi}\Big ] d\xi
\end{align*}
so that, after the change of variables~\(\xi \leftarrow \xi-\eta\), we get
\begin{align*}
    & i \left(2\pi\right)^{3/2}\widehat{L_2(V)}(t,\eta)\\
    &= \int e^{i x \cdot \eta} \widehat w (\eta) \widehat V(\tau, \eta) M(t-\tau, \eta) \,d\eta\, d\tau + i \int_0^t  \int e^{-i x \cdot \eta} \widehat w (-\eta) \widehat V(\tau, -\eta) \overline{M(t-\tau, \eta)}\, d\eta\, d\tau \\
    & =\int_0^t \iint e^{i x\cdot \eta} \widehat w (\eta) \widehat V(\tau, \eta) \Tr \Big [  \Big(  f(\xi)f(\xi)^* - f(\eta + \xi)f(\eta + \xi)^* \Big ) e^{i(t-\tau) \alpha \cdot \xi} e^{- i (t-\tau) \alpha \cdot (\xi + \eta)} \Big ]\, d\xi\, d \eta\, d \tau,
\end{align*}
which concludes the proof.
\end{proof}
\begin{corollary}
\label{rmk:L_2}
    Recall that~\(g\), defined as $g(\xi)= f(\xi) f^*(\xi)$, is even by assumption. Therefore, the Fourier symbol of $L_2(V)$ can be rewritten as
    \begin{equation}\label{eq:L2-fourier-symbol}
        \widehat{L_2(V)} (t,\eta) = \int_0^t F(t-\tau,\eta) \widehat V(\tau, \eta)\,d\tau\,,
    \end{equation}
    where
    \begin{equation}\label{eq:L2-F-kernel}
        F(t,\eta) = \widehat w(\eta) \int \Tr \left( g^-_\eta(\xi) (\xi\wedge \eta)\cdot A \boldsymbol\alpha\right) \frac{\sin(t|\xi|)}{|\xi|} \frac{\sin(t|\eta+\xi|)}{|\eta+\xi|}\,d\xi\,,
    \end{equation}
    with~\(g^-_\eta(\xi) = g(\xi) - g(\xi + \eta)\) and 
    \begin{equation}\label{eq:A-matrix}
        A = \begin{pmatrix}
        0& \mathbbm 1_{\C^2}\\
        \mathbbm 1_{\C^2} & 0 
        \end{pmatrix}\,.
    \end{equation}
\end{corollary}

We shall use the following formulae.

\begin{lemma}\label{lemma:exp}
For any $\tau\in\R$ and $\xi,\eta\in\R^3$, one has:
    \begin{enumerate}
        \item $e^{i\tau \boldsymbol\alpha\cdot\xi} = \cos (\tau \vert \xi \vert) \mathbbm 1_{\C^4} + i \sin (\tau \vert \xi \vert ) \frac{\boldsymbol\alpha\cdot\xi}{\vert \xi \vert}\,$;
        \item\label{ass:exponential-adjoint}
        $\big ( e^{i\tau \xi \cdot \alpha} \big)^* = e^{-i\tau \xi \cdot \alpha}\,$;
        \item
        $\Big [ e^{i \tau  \boldsymbol\alpha\cdot\xi} , e^{-i\tau \boldsymbol\alpha\cdot(\eta +\xi)} \Big] = 2i \frac{ \sin (\tau \vert \xi \vert )}{\vert \xi \vert}  \frac{ \sin (\tau \vert \eta + \xi \vert )}{\vert \eta + \xi \vert} ( \xi \wedge \eta )\cdot A\boldsymbol\alpha$, with $A$ defined in \eqref{eq:A-matrix};
        \item \begin{multline*}
            \{e^{i\tau \boldsymbol\alpha\cdot\xi},e^{-i\tau\boldsymbol\alpha\cdot(\eta+\xi)}\} = 2\cos\left(\tau\left|\xi\right|\right)\cos\left(\tau\left|\xi+\eta\right|\right)-2i\cos\left(\tau\left|\xi\right|\right)\sin\left(\tau\left|\xi+\eta\right|\right)\frac{\boldsymbol\alpha\cdot\left(\xi+\eta\right)}{\left|\xi+\eta\right|}\\
            +2i\cos\left(\tau\left|\xi+\eta\right|\right)\sin\left(\tau\left|\xi\right|\right)\frac{\boldsymbol\alpha\cdot\xi}{\left|\xi\right|}+2\sin\left(\tau\left|\xi\right|\right)\sin\left(\tau\left|\xi+\eta\right|\right)\frac{\left|\xi\right|^2+\xi\cdot\eta}{\left|\xi\right|\left|\xi+\eta\right|}\,.
        \end{multline*}
    \end{enumerate}
\end{lemma}
\begin{remark}
    We have already used the identity in 2). It follows from the fact that~\(\alpha^*_j = \alpha_j\,,\) for $j=1,2,3$.
\end{remark}
\begin{proof}
Equality in~(1) easily descends from the fact that
\[
(\boldsymbol\alpha\cdot\xi)^2 = \frac12\{\boldsymbol\alpha\cdot\xi, \boldsymbol\alpha\cdot\xi\} = |\xi|^2\,.
\]\par
Let us now prove the commutator formula in~(3). By using~(1), we get
\begin{align*}
     \Big [ e^{i \tau  \boldsymbol\alpha\cdot\xi } , e^{-i\tau \boldsymbol\alpha\cdot (\eta +\xi)} \Big] & =  \frac{ \sin (\tau \vert \xi \vert )}{\vert \xi \vert}  \frac{ \sin (\tau \vert \eta + \xi \vert )}{\vert \eta + \xi \vert} [ \boldsymbol\alpha\cdot\xi, \boldsymbol\alpha\cdot(\eta + \xi)]\,.
\end{align*}
Focusing on the commutator on the right-hand side, 
\begin{align*}
[ \boldsymbol\alpha\cdot\xi, \boldsymbol\alpha\cdot(\eta + \xi)] &=  [ \boldsymbol\alpha\cdot\xi, \boldsymbol\alpha\cdot\eta] = \sum_{j \ne k} \xi_j \eta_k [\alpha_j, \alpha_k] = 2 \sum_{j \ne k} \xi_j \eta_k \alpha_j \alpha_k \\
& = 2 \sum_{j \ne k} \xi_j \eta_k \begin{pmatrix}
    \sigma_j \sigma_k & 0\\
    0 & \sigma_j \sigma_k
\end{pmatrix}
= 2i \sum_{j \ne k} \xi_j \eta_k \sum_m \varepsilon_{jkm} \begin{pmatrix}
    \sigma_m & 0 \\
    0 & \sigma_m
\end{pmatrix}\\
& = 2 i \sum_{j \ne k} \sum_m  \varepsilon_{jkm} \xi_j \eta_k A \alpha_m = 2i ( \xi \wedge \eta ) \cdot A \boldsymbol \alpha \,,
\end{align*}
where $A$ is defined in \eqref{eq:A-matrix} and $\varepsilon_{jkm}$ denotes the Levi-Civita symbol, we get the desired equality.\par

Finally, recalling that $\{\boldsymbol \alpha \cdot \xi,\boldsymbol \alpha \cdot \left(\xi+\eta\right)\}=2\left|\xi\right|^2+2\xi\cdot\eta$, a lengthy but easy calculation yields the anticommutator formula in~(4).
\end{proof}

\begin{proof}[Proof of~\Cref{rmk:L_2}]
    Recalling the notation $g(\xi) = f(\xi) f^*(\xi)$ and that $AB=\frac{1}{2}\left(\left[A,B\right]+\{A,B\}\right)$, we can write
    \begin{multline*}
        \int \Tr \Big [  \Big(  f(\xi)f(\xi)^* - f(\eta + \xi)f(\eta + \xi)^* \Big ) e^{i(t-\tau) \boldsymbol\alpha \cdot \xi} e^{- i (t-\tau) \boldsymbol\alpha \cdot (\xi + \eta)} \Big ] d\xi \\
        = \frac{1}{2} \int \Tr \Big (  \Big( g(\xi) - g(\xi + \eta) \Big ) \Big ([e^{i(t-\tau) \boldsymbol\alpha \cdot \xi}, e^{- i (t-\tau) \alpha \cdot (\xi + \eta)}] + \{e^{i(t-\tau) \boldsymbol\alpha \cdot \xi}, e^{- i (t-\tau) \alpha \cdot (\xi + \eta)}\} \Big ) \,d\xi\,.
    \end{multline*}
    Then, we observe that the function 
    \[
    \Sigma: (\xi,\zeta) \mapsto  \Tr \Big (  \Big( g(\xi) - g(\zeta) \Big ) \{e^{i(t-\tau) \boldsymbol\alpha \cdot \xi}, e^{ i (t-\tau) \boldsymbol\alpha \cdot \zeta}\}\Big)
    \]
    is skew-symmetric yielding (with the change of variable~\(\xi\leftarrow -\xi-\eta\)) 
    \[
    \int \Sigma(\xi,-(\eta+\xi))\, d\xi = \int \Sigma(-(\eta + \xi),\xi)\, d\xi = 0 .
    \]
    Recalling \eqref{eq:L_2(V)} and the expression for the commutator in Lemma~\ref{lemma:exp}, we conclude the proof. 
\end{proof}

\subsection{Setting the \texorpdfstring{\(TT^*\)}{TT*} Argument}

The previous subsection provides a description of the operator~\(L_2\) which is precise enough to estimate it. To estimate~\(L_1\), we use a~\(TT^*\) argument where the~\(TT^*\) operator takes the same form as~\(L_2\). 

\begin{proposition}\label{prop:TTstar}
Let~\(\Theta\) be the operator acting on~\(L^2_{\omega,x}\) such that
\[
\Theta : Z_0 \mapsto \left( t\mapsto i \E(\la Y(t),S(t) Z_0\ra)\right) \,.
\]
We have that: \begin{enumerate}
\item \(L_3 = 2\Re (-i \Theta)\);
\item
\(
\Theta^* : U \mapsto -i \int_{\R} S(-\tau) [U(\tau) Y(\tau)] \,d\tau \, ;
\)
\item if there exist Banach spaces~\(B_1\) and~\(B_2\) and~\(p>2\) such that the operator defined by
\begin{equation*}
    L_1^\infty(V)=S(t) \Theta^* (w\ast V)=\int_\R S(t-\tau)[w\ast V(\tau) Y(\tau)]\,d\tau \,
\end{equation*}
is continuous from~\(L^2(\R,B_1)\) to~\(L^p(\R,B_2)\), then so is~\(L_1\);
\item\label{ass:ThetaThetastar} the operator~\(\Theta \Theta^*\) is given by
\[
\widehat{\Theta\Theta^*(V)}(t,\eta) = \int_{\R} \widehat V(\tau, \eta) M(t-\tau, \eta)\, d\tau \,,
\]
where we recall
\[
M(t,\eta) = \int_{\R^3} \Tr_{\C^4} \big[ f(\xi) f(\xi)^*e^{it\boldsymbol{\alpha} \cdot \xi} e^{-it\boldsymbol{\alpha}\cdot (\eta + \xi)}\big] \, d\xi\, .
\]
\end{enumerate}
\end{proposition}

\begin{proof}
    Assertion~(1) follows from the definitions of~\(\Theta\) and~\(L_3\).

    Assertion~(2) follows from the computation below:
    \begin{align*}
    \an{Z_0,\Theta^*(U)}_{\omega,x} = & \an{\Theta(Z_0),U}_{t,x} \\
    = & \int_{\R} \int_{\R^3} -i \E(\an{S(t)Z_0,Y(t)}_{\C^4}) U(t)\, dx\, dt \\
    = & \int_{\R} \E \left( -i \an{S(t) Z_0, U(t) Y(t)}_x\right)\, dt \\
    = & \int_{\R} \E \left( -i \an{Z_0, S(-t) [U(t) Y(t)]}_x \right)\, dt \\
    = & \an{Z_0, \int_\R S(-t)[U(t) Y(t)] \, dt}_{x,\omega}\,.
    \end{align*}

    Assertion~(3) is due to the fact that
    \[
    L_1^\infty(V) = \int_\R S(t-\tau)[w\ast V(\tau) Y(\tau)] \,d\tau\quad\text{and}\quad \quad L_1(V) = \int_0^t S(t-\tau)[w\ast V(\tau) Y(\tau)] \,d\tau\, . 
    \]
    Therefore, it follows from the Christ--Kiselev lemma.

    By definition, we have 
    \[
    \Theta \Theta^*(U) = \E\left(\la Y(t) , \int_{\R} S(t-\tau) [U(\tau) Y(\tau)]\, d\tau \ra \right)\,.
    \]
    Recall that
    \begin{multline*}
    S(t-\tau)[U(\tau) Y(\tau)] = \iint e^{i\eta \cdot x} e^{-i(t-\tau) (\boldsymbol{\alpha}\cdot \eta +\widetilde m)} e^{-i\tau (\boldsymbol{\alpha}\cdot \xi + \widetilde m)} \mathcal F_x{V(\tau) e^{i x\cdot \xi}}(\eta) f(\xi)\, d\eta\, dW(\xi) \\
    = \iint e^{i\eta \cdot x} e^{-i(t-\tau) (\boldsymbol{\alpha}\cdot \eta +\widetilde m)} e^{-i\tau (\boldsymbol{\alpha}\cdot \xi + \widetilde m)} \widehat V(\tau, \eta-\xi) f(\xi) \,d\eta\, dW(\xi)\,.
    \end{multline*}
    We deduce
    \begin{multline*}
    \E\left(\la Y(t) , S(t-\tau) [U(\tau) Y(\tau)] \ra \right) \\
    = \iint \Tr_{\C^4} \left( f(\xi)^*e^{it(\boldsymbol{\alpha}\cdot \xi +\widetilde m)} e^{i(\eta - \xi) \cdot x} e^{-i(t-\tau) (\boldsymbol{\alpha}\cdot \eta +\widetilde m)} e^{-i\tau (\boldsymbol{\alpha}\cdot \xi + \widetilde m)} \widehat V(\tau, \eta-\xi) f(\xi) \right)\, d\eta\, d\xi\, .
    \end{multline*}
    Since~\(V\) is scalar, we have
    \begin{multline*}
    \E\left(\la Y(t) , S(t-\tau) [U(\tau) Y(\tau)] \ra \right) \\
    = \iint \Tr_{\C^4} \left( f(\xi)^*e^{it\boldsymbol{\alpha}\cdot \xi }  e^{-i(t-\tau) \boldsymbol{\alpha}\cdot \eta } e^{-i\tau \boldsymbol{\alpha}\cdot \xi }  f(\xi) \right)e^{i(\eta - \xi) \cdot x} \widehat V(\tau, \eta-\xi)\, d\eta \, d\xi\, .
    \end{multline*}
    Using the cyclicity of the trace and the commutativity of~\(f(\xi) f(\xi)^*\) with~\(\boldsymbol{\alpha}\cdot \xi\), we get
    \begin{multline*}
    \E\left(\la Y(t) , S(t-\tau) [U(\tau) Y(\tau)] \ra \right)\\
    = \iint \Tr_{\C^4} \left(  e^{-i(t-\tau) \boldsymbol{\alpha}\cdot \eta } e^{-i(\tau - t) \boldsymbol{\alpha}\cdot \xi}  f(\xi) f(\xi)^*\right) e^{i(\eta - \xi) \cdot x} \widehat V(\tau, \eta-\xi)\, d\eta\, d\xi\, .
    \end{multline*}
    A change of variables yields
    \begin{multline*}
    \E\left(\la Y(t) , S(t-\tau) [U(\tau) Y(\tau)]\ra \right) \\
    =\iint \Tr_{\C^4} \left(  e^{-i(t-\tau) \boldsymbol{\alpha}\cdot (\eta + \xi) } e^{-i(\tau - t) \boldsymbol{\alpha}\cdot \xi}  f(\xi) f(\xi)^*\right) e^{i\eta \cdot x} \widehat V(\tau, \eta)\, d\eta\, d\xi\, .
    \end{multline*}
    Finally, integrating in time, we get Assertion~\eqref{ass:ThetaThetastar}.
\end{proof}

%

\section{Estimates of the Linear Response}

\subsection{Preliminary Lemmata}

Let~\(\eta \neq 0\). We set~\(\widehat \eta = \frac{\eta}{|\eta|}\) and let~\(\mathcal R\in O(3)\) such that~\(\widehat \eta = \mathcal Re_3\) where~\(e_3 = (0,0,1)\). For all~\(\xi \in \R^3\), we set
\[
\xi = z\widehat \eta + r\sigma\,,
\]
where~\(z = \xi\cdot \widehat \eta = (\mathcal R^*\xi)\cdot e_3\), where~\(r = \sqrt{|\xi|^2 - z^2}\), and~\(\sigma = \mathcal R^*(\cos \theta,\sin\theta,0)\) such that 
\[
r(\cos\theta,\sin \theta,0) = \mathcal R^*\xi - ze_3\,.
\]
For any integrable map $h$ of~\(\R^3\), we have
\begin{equation}\label{eq:r-z-theta-change}
    \int h(\xi) \,d\xi = \int_{\R_+^* \times \R\times[0,2\pi]} h(z\widehat\eta + r\sigma)\, r\,dr\,dz\,d\theta\, .
\end{equation}
We write~\(I(r,z,\theta) = z\widehat\eta + r\sigma\) and set 
\[
\begin{array}{cccc}
\Phi_\pm : & \R_+^* \times \R & \longrightarrow & \R \\
& (r,z) & \longmapsto & \sqrt{(z+|\eta|)^2 + r^2)} \pm \sqrt{z^2 + r^2}
\end{array}\,.
\]
With a slight abuse of notation, we also denote~\(\Phi_\pm(\xi,\eta) = \Phi_\pm(\sqrt{|\xi|^2 - (\xi\cdot \widehat \eta)^2},\xi\cdot \widehat \eta) =  |\xi + \eta| \pm |\xi|\).

\begin{lemma}\label{lemma:diffeo}
The map
\[
\Psi : \begin{array}{ccc}
\R_+^* \times \R \times \mathrm{S}^1  & \longrightarrow & (|\eta|,\infty)\times(-|\eta|,|\eta|)\times\mathrm{S}^1\\
(r,z,\theta) & \longmapsto & (\Phi_+(r,z), \Phi_-(r,z),\theta)
\end{array} 
\]
is a~\(C^1\)--diffeomorphism. Its inverse is given by
\begin{equation}\label{eq:Psi-inverse}
    \Psi^{-1} (s_+,s_-,\theta) = \left( \sqrt{\left(\frac{s_+-s_-}{2}\right)^2 - \left(\frac{s_+s_--|\eta|^2}{2|\eta|}\right)^2} , \frac{s_+s_--|\eta|^2}{2|\eta|},\theta\right)\,.
\end{equation}
Moreover, we have
\begin{align*}
    \lim_{s_+\rightarrow |\eta|} \Psi^{-1}(s_+,s_-,\theta) =& \left(0,\frac{s_--|\eta|}{2},\theta\right)\,,\\
    \lim_{s_-\rightarrow \pm |\eta|} \Psi^{-1}(s_+,s_-,\theta) = &\left( 0,\frac{\pm s_+ - |\eta|}{2},\theta\right)\,,
\end{align*}
and
\begin{equation}\label{eq:Psi-jacobian}
    \mathrm{jac}(\Psi)(r,z,\theta) = \frac{2r|\eta|}{|\xi|\,|\xi+\eta|}\,.
\end{equation}
\end{lemma}
\begin{proof}
    Set $W\coloneqq(|\eta|,\infty)\times(-|\eta|,|\eta|)\times\mathrm{S}^1$. We start by showing that $\Psi$ is bijective. First, $\mathrm{Im}\Psi\subset W$. Indeed,
    \begin{equation*}
        \Phi_+\left(r,z\right)=|\xi+\eta|+|\xi|\geq |\eta|-|\xi|+|\xi|=|\eta|
    \end{equation*}
    and
    \begin{equation*}
        \Phi_-\left(r,z\right)=|\xi+\eta|-|\xi|\implies|\Phi_-|\leq|\xi+\eta-\xi|=|\eta|\,.
    \end{equation*}\par
    Secondly, $\Psi$ is injective. Indeed, suppose that there exists $r,r'\in\R^*_+$ and $z,z'\in\R$ such that $\Phi_\pm\left(r,z\right)=\Phi_\pm\left(r',z'\right)$. Then,
    \begin{equation*}
        \Phi_+\left(r,z\right)\Phi_-\left(r,z\right)=|\xi+\eta|^2-|\xi|^2=|\eta|^2+2\xi\cdot \eta=|\eta|^2+2|\eta|z\,,
    \end{equation*}
    and, similarly,
    \begin{equation*}
        \Phi_+\left(r',z'\right)\Phi_-\left(r',z'\right)=|\eta|^2+2|\eta|z'\,,
    \end{equation*}
    which imply $z=z'$. Moreover,
    \begin{equation*}
        |\xi|=\frac{\Phi_+\left(r,z\right)-\Phi_-\left(r,z\right)}{2}=\frac{\Phi_+\left(r',z'\right)-\Phi_-\left(r',z'\right)}{2}=|\xi'|\,,
    \end{equation*}
    yielding $\sqrt{r^2+z^2}=\sqrt{r'^2+z^2}$ which implies $r=r'$.\par
    Finally, $\Psi$ is surjective. Indeed, let $\left(s_+,s_-\right)\in(|\eta|,\infty)\times(-|\eta|,|\eta|)$. Set
    \begin{align}\label{eq:z-function-of-s+s-}
        z&=\frac{s_+s_--|\eta|^2}{2|\eta|}\,,\\\label{eq:r-function-of-s+s-}
        r&=\sqrt{\frac{\left(s_+-s_-\right)^2}{4}-\left(\frac{s_+s_--|\eta|^2}{2|\eta|}\right)^2}\,.
    \end{align}
    One can first easily check that $r$ is well-defined since $s_-\in(-|\eta|,|\eta|)$. Then, an easy calculation yields $\Phi_\pm\left(r,z\right)=s_\pm$. As a consequence, we also get that the inverse of $\Psi$ is obviously given by \eqref{eq:Psi-inverse}.\par
    Concerning the Jacobian, we first notice that
    \begin{equation*}
        \mathrm{jac}(\Psi)(r,z,\theta) =\Bigg|\frac{\partial\left(\Phi_+,\Phi_-\right)}{\partial\left(r,z\right)}\Bigg|\,,
    \end{equation*}
    and then compute
    \begin{align*}
        \partial_r\Phi_\pm=\frac{r}{|\xi+\eta|}\pm\frac{r}{|\xi|}\,,\quad\partial_z\Phi_\pm=\frac{|\eta|+z}{|\xi+\eta|}\pm\frac{z}{|\xi|}\,,
    \end{align*}
    which yields \cref{eq:Psi-jacobian}. This concludes the proof.
\end{proof}


%

\begin{definition}
Let $m_1,m_2\in\N$. We define
\[
\mathcal S_{m_1,m_2} = \{ h \in L^\infty(\R^3_\xi\times \R^3_\eta) \;|\; (\an{\xi}^{m_1} +\an{\eta}^{m_1})\partial^k_{\xi} h \in L^\infty(\R^3_\xi\times\R^3_\eta),\,\forall k\in\N^3, |k|\leq m_2\}\,,
\]
endowed with the norm
\[
\|h\|_{m_1,m_2} = \sum_{|k|\leq m_2} \|(\an{\xi}^{m_1}+\an{\eta}^{m_1}) \partial^k_\xi h\|_{L^\infty(\R^6)}\,.
\]
\end{definition}

\begin{lemma}\label{stat:h-cos-decomposition}
Let~\(F_\pm\) be the linear map defined from~\(L^1_\xi L^\infty_\eta\) to~\(L^\infty_{t,\eta}\) as 
\[
F_\pm(h)(t,\eta) = \int_{\R^3} h(\xi,\eta) \cos(t\Phi_\pm(\xi,\eta)) \,d\xi\, .
\]
There exists two linear maps~\(\overline F_\pm ,\widetilde F_\pm\) such that~\(\overline F_\pm\) and~\(\widetilde F_\pm\) are continuous from~\(\mathcal S_{8,2}\) to respectively~\(\an{\eta}^{-1}L^\infty(\R^3)\) and~\(L^1(\R,L^\infty(\R^3))\) and satisfy, for all~\(t\in \R\) and all~\(\eta\in \R^3\),
\[
 F_\pm(h)(t,\eta) = \frac{\sin(|\eta|t)}{|\eta| t} \overline F_\pm(h)(\eta) + |\eta|^{-2}\widetilde F_\pm(h)(t,\eta)\,.
\]
Furthermore, if~\(h(\xi,\eta)\) vanishes whenever~\(\xi\) and~\(\eta\) are colinear, then~\(\overline F_\pm(h) = 0\). 

Moreover, \(\overline F_\pm\) and~\(\widetilde F_\pm\) can be extended to continuous maps from~\(\frac{1}{|\xi|\;|\xi +\eta|} \mathcal S_{6,2}\) to respectively~\(\an{\eta}^{-1}L^\infty(\R^3)\) and~\(L^1(\R,L^\infty(\R^3))\), meaning that
\begin{align*}
    \|\an{\eta} \overline F_\pm(h)\|_{L^\infty_\eta} \lesssim &\|\,|\xi|\,|\xi+\eta|\,h\|_{6,2}\,,\\
    \|\widetilde F_\pm(h)\|_{L^1_tL^\infty_\eta} \lesssim &\|\,|\xi|\,|\xi+\eta|\,h\|_{6,2}\,.
\end{align*}
\end{lemma}

\begin{proof}
First of all, let~\(\mu : \R \rightarrow [0,1]\) be a~\( C^\infty\)--compactly supported function such that~\(\mu|_{[-1,1]} \equiv 1\). We set~\(\overline \mu = 1-\mu\) and
\begin{align*}
    F^{\mu}_\pm(h) (t,\eta) = & \mu(t) F_\pm(h) (t,\eta)\,, \\
    F^{\overline \mu}_\pm(h) (t,\eta) = & \overline\mu(t) F_\pm(h) (t,\eta)\,,
\end{align*}
such that
\[
F_\pm = F_\pm^\mu + F_\pm^{\overline\mu}\,.
\]
Since
\[
\|\an{\eta}^2 \,F_\pm (h)\|_{L^\infty_{t,\eta}} \leq \|h\|_{6,0}\,,
\]
we deduce
\[
\|\an{\eta}^2F_\pm^\mu (h)\|_{L^1_tL^\infty_{\eta}} \leq \|\mu\|_{L^1} \|\,h\|_{6,0}\,.
\]\par

Assume now that~\(\eta \neq 0\). First, recalling \eqref{eq:r-z-theta-change}, we have
\begin{align*}
    F_\pm(h)(t,\eta) &= \int_{\R^3} h(\xi,\eta) \cos(t\Phi_\pm(\xi,\eta))\, d\xi\\
    &= \int_{\R_+^* \times \R\times[0,2\pi]} h\left(I\left(r,z,\theta,\eta\right)\right)\cos\left(t\Phi_\pm\left(I\left(r,z,\theta,\eta\right)\right)\right)\,r\, dr\,dz\, d\theta\,,
\end{align*}
and perform a change of variables through the diffeormorphism~\(\Psi\) in Lemma~\ref{lemma:diffeo} to get
\[
F_\pm^{\overline\mu} (h) (t,\eta) = \overline \mu(t) \int_{(|\eta|,\infty) \times (-|\eta|,|\eta|) \times (0,2\pi)} h(I(\Psi^{-1}(s_+,s_-),\theta),\eta) \frac{|\xi|\,|\xi+\eta|}{2 |\eta|} \cos(s_\pm t)\, ds_+\, ds_-\, d\theta\,.
\]
We then proceed with a first integration by parts with respect to $s_+$, respectively $s_-$, to obtain
\begin{align*}
F_+^{\overline\mu} (h) (t,\eta) = & \overline \mu(t) \int_{ (-|\eta|,|\eta|) \times (0,2\pi)} \Big[ h(I(\Psi^{-1}(s_+,s_-),\theta),\eta) \frac{|\xi|\,|\xi+\eta|}{2 |\eta|} \frac{\sin(s_+ t)}{t}\Big]_{|\eta|}^\infty\, ds_-\,d\theta\\
& - \overline \mu(t) \int_{ (|\eta|,\infty) \times (-|\eta|,|\eta|) \times (0,2\pi)} \partial_{s_+}\Big[h(I(\Psi^{-1}(s_+,s_-),\theta),\eta) \frac{|\xi|\,|\xi+\eta|}{2 |\eta|}\Big] \frac{\sin(s_+ t)}{t}\, ds _+\, ds_-\, d\theta\,,\\
F_-^{\overline\mu} (h) (t,\eta) = & \overline \mu(t) \int_{ (|\eta|,\infty) \times (0,2\pi)} \Big[ h(I(\Psi^{-1}(s_+,s_-),\theta),\eta) \frac{|\xi|\,|\xi+\eta|}{2 |\eta|} \frac{\sin(s_- t)}{t}\Big]_{-|\eta|}^{|\eta|}\,ds_+\,d\theta \\
& - \overline \mu(t) \int_{ (|\eta|,\infty) \times (-|\eta|,|\eta|) \times (0,2\pi)} \partial_{s_-}\Big[h(I(\Psi^{-1}(s_+,s_-),\theta),\eta) \frac{|\xi|\,|\xi+\eta|}{2 |\eta|}\Big] \frac{\sin(s_- t)}{t}\, ds _+\, ds_-\, d\theta\,.
\end{align*}
Since
\begin{equation*}
    |\xi|=\frac{s_+-s_-}{2}\,,\quad |\xi+\eta|=\frac{s_++s_-}{2}\,,
\end{equation*}
thanks to Lemma~\ref{lemma:diffeo} and the polynomial decrease of~\(h\), we have
\begin{align*}
 \Big[ h(I(\Psi^{-1}(s_+,s_-),\theta),\eta) \frac{|\xi|\,|\xi+\eta|}{2 |\eta|} \frac{\sin(s_+ t)}{t}\Big]_{|\eta|}^\infty  = & - h(\frac{s_- - |\eta|}{2} \widehat \eta,\eta) \frac{|\eta|^2 - s_-^2}{8 |\eta|} \frac{\sin(|\eta| t)}{t}\,, \\
 \Big[ h(I(\Psi^{-1}(s_+,s_-),\theta),\eta) \frac{|\xi|\,|\xi+\eta|}{2 |\eta|} \frac{\sin(s_- t)}{t}\Big]_{-|\eta|}^{|\eta|} = & \frac{\sin(|\eta|t)}{t} \frac{s_+^2 - |\eta|^2}{8 |\eta|} \Big( h(\frac{s_+ - |\eta|}{2}\widehat \eta,\eta) + h(\frac{-s_+ - |\eta|}{2}\widehat \eta,\eta) \Big)\,.
\end{align*}
Set
\begin{align*}
\overline F_+(h)(\eta) = & -\int_{(-|\eta|,|\eta|) \times (0,2\pi)} h(\frac{s_- - |\eta|}{2} \widehat \eta,\eta) \frac{|\eta|^2 - s_-^2}{8}\,ds_-\, d\theta\,, \\ 
\overline F_-(h)(\eta) =& \int_{(|\eta|,\infty) \times (0,2\pi)} \frac{s_+^2 - |\eta|^2}{8} \Big( h(\frac{s_+ - |\eta|}{2}\widehat \eta,\eta) + h(\frac{-s_+ - |\eta|}{2}\widehat \eta,\eta) \Big)\,ds_+\, d\theta\,.
\end{align*}
Since
\begin{align*}
    0 \leq |\eta|^2 - s_-^2 \leq & \an{|\eta| - s_-}[\an{|\eta| - s_-} + 2 \an \eta]\,, \\
    0 \leq s_+^2 - |\eta|^2  \leq & \an{s_+ -|\eta| }[\an{s_+- |\eta|} + 2 \an \eta]\,,
\end{align*}
we deduce that
\[
\|\an{\eta} \overline F_\pm (h)\|_{L^\infty_\eta},\,
\|\an{\eta}^2 \overline F_\pm (h)\|_{L^\infty_\eta} \lesssim \|h\|_{6,0}\,.
\]
Additionally, notice that if \(h\) vanishes whenever~\(\xi\) and~\(\eta\) are colinear, then
\[
\overline F_\pm(h) = 0\,.
\]\par

With this definition of~\(\overline F_\pm\), we get that
\begin{align*}
F_\pm^{\overline\mu} (h) (t,\eta) = & \frac{\sin(|\eta|t)}{|\eta| t} \overline F_\pm(h)(\eta) + \mu(t) \frac{\sin(|\eta|t)}{|\eta| t} \overline F_\pm(h)(\eta) \\
& - \overline \mu(t) \int_{ (|\eta|,\infty) \times (-|\eta|,|\eta|) \times (0,2\pi)} \partial_{s_\pm}\Big[h(I(\Psi^{-1}(s_+,s_-),\theta),\eta) \frac{|\xi|\,|\xi+\eta|}{2 |\eta|}\Big] \frac{\sin(s_\pm t)}{t}\, ds_+\, ds_-\, d\theta\,.
\end{align*}

We have, since~\(\frac{\sin(|\eta| t)}{|\eta| t}\) is bounded 
\[
\| \an{\eta}^2\mu(t) \frac{\sin(|\eta|t)}{|\eta| t} \overline F_\pm(h)\|_{L^1_t,L^\infty_\eta} \lesssim \|h\|_{6,0}\,.
\]
Hence, we are left with proving that~\(|\eta|F^1_\pm(h)\) belongs to~\(L^1_t,L^\infty_\eta\), with
\begin{align*}
F^1_\pm(h) (t,\eta) = & - \overline \mu(t) \int_{ (|\eta|,\infty) \times (-|\eta|,|\eta|) \times (0,2\pi)} \partial_{s_\pm}\Big[h(I(\Psi^{-1}(s_+,s_-),\theta),\eta) \frac{|\xi|\,|\xi+\eta|}{2 }\Big] \frac{\sin(s_\pm t)}{t}\, ds _+\, ds_-\, d\theta\,. 
\end{align*}
Set~\(a \in (0,\frac13)\) and split the integral over the domains respectively given by~\(r\leq t^{-a}\) and by~\(r>t^{-a}\). Set
\begin{align}
\label{eq:F-smaller} F^\leq_\pm(h) (t,\eta) = & - \overline \mu(t) \int_{ r\leq t^{-a}} \partial_{s_\pm}\Big[h(I(\Psi^{-1}(s_+,s_-),\theta),\eta) \frac{|\xi|\,|\xi+\eta|}{2 }\Big] \frac{\sin(s_\pm t)}{t}\, ds _+\, ds_-\, d\theta\,, \\
\label{eq:F-bigger} F^>_\pm(h) (t,\eta) = & - \overline \mu(t) \int_{ r> t^{-a}} \partial_{s_\pm}\Big[h(I(\Psi^{-1}(s_+,s_-),\theta),\eta) \frac{|\xi|\,|\xi+\eta|}{2 }\Big] \frac{\sin(s_\pm t)}{t}\, ds _+\, ds_-\, d\theta\,.
\end{align}
Concerning the integral in \eqref{eq:F-smaller}, we perform a reverse change of variable and get
\begin{align*}
F^\leq_\pm(h) (t,\eta) = & - \overline \mu(t) \int_{ r\leq t^{-a}} \partial_{s_\pm}\Big[h(I(\Psi^{-1}(s_+,s_-),\theta),\eta) \frac{|\xi|\,|\xi+\eta|}{2 }\Big] \frac{\sin(s_\pm t)}{t} \frac{|\eta|r}{|\xi+\eta|\,|\xi|}\, dr\, dz\, d\theta\,.
\end{align*}
We then use the Leibniz rule to obtain 
\begin{multline}\label{eq:h-leibniz-aux}
    \partial_{s_\pm}\Big[h(I(\Psi^{-1}(s_+,s_-),\theta),\eta) \frac{|\xi|\,|\xi+\eta|}{2 }\Big]  \frac{|\eta|r}{|\xi+\eta|\,|\xi|}\\
    =  \frac12\partial_{s_\pm} I(\Psi^{-1}(s_+,s_-),\theta) \cdot \nabla_\xi h(\xi,\eta)|\eta| r \pm h(\xi,\eta) \frac{s_\pm}{4} \frac{|\eta|r}{|\xi+\eta|\,|\xi|}\,.
\end{multline}
Focusing on the first addend on the right-hand side, we have 
\[
\partial_{s_\pm} I(\Psi^{-1}(s_+,s_-),\theta)=\frac{s_\mp}{2|\eta|}\widehat{\eta}\pm\underbrace{\frac{s_+-s_-}{4r}}_{|\xi|/2r}\sigma-\frac{zs_\mp}{2r|\eta|}\sigma\,,
\]
which implies (recall that $|\xi|=\sqrt{r^2+z^2}\geq |r|,|z|$ and $s_\pm\leq |\xi|+|\eta|$)
\begin{equation}\label{eq:fstderPsi}
r |\eta| \partial_{s_\pm} I(\Psi^{-1}) (s_+,s_-),\theta)| \lesssim (\an \xi + \an \eta)^2\,.
\end{equation}
Moreover, from \(r\leq |\xi|,|\xi+\eta|\), one gets
\[
\frac{|s_\pm|}{4} \frac{|\eta|r}{|\xi+\eta|\,|\xi|} \leq \frac{|\eta|}{2}\,.
\]
Thus, we deduce
\[
|\eta|\,|F^\leq_\pm(h) (t,\eta) |\lesssim \overline \mu(t) t^{-1-a} \|h\|_{6,1}\,,
\]
which yields
\[
\||\eta| \,F^\leq_\pm(h)  \|_{L^1_t L^\infty_\eta} \lesssim \|h\|_{6,1}\,.
\]\par

We turn to~\(F_\pm^>\) in \eqref{eq:F-bigger}. Here, we perform an additional integration by parts (again with respect to $s_+$ and $s_-$ respectively) to get
\begin{align}
\notag F^>_+(h) (t,\eta) = &  \overline \mu(t) \int_{ (-|\eta|,\eta) \times (0,2\pi)} \Big[ \partial_{s_+}\Big[h(I(\Psi^{-1}(s_+,s_-),\theta),\eta) \frac{|\xi|\,|\xi+\eta|}{2 }\Big] \frac{\cos(s_+ t)}{t^2}\Big]_{T_+(t)}^\infty\, ds_-\, d\theta \\
&\label{eq:F-plus-bigger-2derivative} - \overline \mu(t) \int_{ r> t^{-a}} \partial_{s_+}^2\Big[h(I(\Psi^{-1}(s_+,s_-),\theta),\eta) \frac{|\xi|\,|\xi+\eta|}{2 }\Big] \frac{\cos(s_+ t)}{t^2}\, ds _+\, ds_- d\theta\,, \\
\notag F^>_-(h) (t,\eta) = & \overline \mu(t) \int_{ (\sqrt{|\eta|^2 + 4t^{-2a}},\infty)\times (0,2\pi)} \Big[\partial_{s_-}\Big[h(I(\Psi^{-1}(s_+,s_-),\theta),\eta) \frac{|\xi|\,|\xi+\eta|}{2 }\Big] \frac{\cos(s_- t)}{t^2}\Big]_{-T_-(t)}^{T_-(t)}\, ds _+\, d\theta \\
&\label{eq:F-minus-bigger-2derivative}- \overline \mu(t) \int_{ r> t^{-a}} \partial_{s_-}^2\Big[h(I(\Psi^{-1}(s_+,s_-),\theta),\eta) \frac{|\xi|\,|\xi+\eta|}{2 }\Big] \frac{\cos(s_- t)}{t^2}\, ds _+\, ds_-\, d\theta\,,
\end{align}
where
\begin{align*}
    T_+(t)  =& \left(\frac{4 t^{-2 a} + |\eta|^2 -s_-^2}{1 - s_-^2/|\eta|^2}\right)^{1/2}\,, \\
    T_-(t) = &  \left(\frac{s_+^2 - 4 t^{-2 a} - |\eta|^2 }{s_+^2/|\eta|^2 - 1}\right)^{1/2}\,.
\end{align*}
We obtain the formulae for~\(T_\pm\) by recalling \eqref{eq:r-function-of-s+s-}, and remarking that
\[
r> t^{-a} \quad \Leftrightarrow \quad s_+^2 + s_-^2 - \frac{s_+^2 s_-^2 }{|\eta^2|} > 4 t^{-2 a} + |\eta|^2\,.
\]

Recall that we have seen in \eqref{eq:h-leibniz-aux} that 
\begin{multline*}
 \partial_{s_\pm}\Big[h(I(\Psi^{-1}(s_+,s_-),\theta),\eta) \frac{|\xi|\,|\xi+\eta|}{2 }\Big]    \\
 =   \left( \frac{s_\mp}{2|\eta|}\widehat{\eta}\pm\frac{s_+-s_-}{4r}\sigma-\frac{zs_\mp}{2r|\eta|}\sigma\right) \cdot \nabla_\xi h(\xi,\eta)\frac{s_+^2 - s_-^2}{8} \pm h(\xi,\eta) \frac{s_\pm}{4}\,.
\end{multline*}
Therefore, the boundary terms respectively become
\begin{multline*}
\Big[ \partial_{s_+}\Big[h(I(\Psi^{-1}(s_+,s_-),\theta),\eta) \frac{|\xi|\,|\xi+\eta|}{2 }\Big] \frac{\cos(s_+ t)}{t^2}\Big]_{T_+(t)}^\infty\\
=- \left(\left( \frac{s_-}{2|\eta|}\widehat{\eta}+\frac{T_+(t)-s_-}{4r}\sigma-\frac{zs_-}{2r|\eta|}\sigma\right) \cdot \nabla_\xi h(\xi,\eta)\frac{T_+(t)^2 - s_-^2}{8} + h(\xi,\eta) \frac{T_+(t)}{4} \right) \frac{\cos(T_+(t) t)}{t^2}\,,
\end{multline*}
and
\begin{align*}
&\Big[\partial_{s_-}\Big[h(I(\Psi^{-1}(s_+,s_-),\theta),\eta) \frac{|\xi|\,|\xi+\eta|}{2 }\Big] \frac{\cos(s_- t)}{t^2}\Big]_{-T_-(t)}^{T_-(t)}\\
&=\Bigg(\left(\left( \frac{s_+}{2|\eta|}\widehat{\eta}+\frac{s_+-s_-}{4r}\sigma-\frac{zs_+}{2r|\eta|}\sigma\right) \cdot \nabla_\xi h(\xi,\eta)\frac{s_+^2 - s_-^2}{8} - h(\xi,\eta) \frac{s_-}{4}\right)\Bigg|_{s_-= T_-(t)}\\
&- \left(\left( \frac{s_+}{2|\eta|}\widehat{\eta}+\frac{s_+-s_-}{4r}\sigma-\frac{zs_+}{2r|\eta|}\sigma\right) \cdot \nabla_\xi h(\xi,\eta)\frac{s_+^2 - s_-^2}{8} - h(\xi,\eta) \frac{s_-}{4}\right)\Bigg|_{s_-= -T_-(t)} \Bigg) \frac{\cos(T_-(t)t)}{t^2}\,.
\end{align*}
We infer that
\begin{multline*}
    |\eta|\,\Big|\overline \mu(t) \int_{ (-|\eta|,\eta) \times (0,2\pi)} \Big[ \partial_{s_+}\Big[h(I(\Psi^{-1}(s_+,s_-),\theta),\eta) \frac{|\xi|\,|\xi+\eta|}{2 }\Big] \frac{\cos(s_+ t)}{t^2}\Big]_{T_+(t)}^\infty\, ds_-\, d\theta\, \Big| \\
    \lesssim \|h\|_{5,1} \overline\mu(t) t^{-2 + a}\,,
\end{multline*}
\begin{multline*}
    |\eta|\,\Big|\overline \mu(t) \int_{ (\sqrt{|\eta|^2 + 4t^{-2a}},\infty)\times (0,2\pi)} \Big[\partial_{s_-}\Big[h(I(\Psi^{-1}(s_+,s_-),\theta),\eta) \frac{|\xi|\,|\xi+\eta|}{2 }\Big] \frac{\cos(s_- t)}{t^2}\Big]_{-T_-(t)}^{T_-(t)}\,ds _+\,  d\theta\,\Big|\\
    \lesssim \|h\|_{6,1} |\eta|^{-1} \overline \mu(t) t^{-2 + a}\,.
\end{multline*}
Going back to \eqref{eq:F-plus-bigger-2derivative} and \eqref{eq:F-minus-bigger-2derivative}, by the Leibniz rule, we compute
\begin{align*}
\partial_{s_\pm}^2 \Big[ h\circ I \circ \Psi^{-1} \frac{|\xi|\,|\xi+\eta|}{2}\Big] = &\partial_{s_\pm}^2 I\circ\Psi^{-1} \cdot \nabla h \circ I \circ \Psi^{-1} \frac{|\xi|\,|\xi+\eta|}{2} \\
& + (\partial_{s_\pm} I \circ \Psi^{-1})^{\otimes 2} \cdot \nabla^{\otimes 2}h\circ I\circ \Psi^{-1} \frac{|\xi|\,|\xi+\eta|}{2} \\
& \pm \partial_{s_\pm} I\circ \Psi^{-1} \cdot \nabla h \circ I \circ \Psi^{-1} \frac{s_\pm}{4}\\
& \pm  \frac14 h\circ I\circ \Psi^{-1}\,.
\end{align*}
Set 
\[
F^2_\pm (h) (t,\eta) = \int_{r>t^{-a}} \partial_{s_\pm}^2 \Big[h\circ I \circ \Psi^{-1} \frac{|\xi|\,|\xi+\eta|}{2}\Big] \frac{\cos(ts_\pm)}{t^2}\,ds_+\, ds_-\, d\theta\,,
\]
that we split as
\[
F^2_\pm = F^3_\pm + F^4_\pm\,,
\]
with
\begin{equation*}
    F^3_\pm(h)(t,\eta) = \int_{r>t^{-a}} \Big[\pm \partial_{s_\pm} I\circ \Psi^{-1} \cdot \nabla h \circ I \circ \Psi^{-1} \frac{s_\pm}{4} \pm  \frac14 h\circ I\circ \Psi^{-1} \Big] \frac{\cos(ts_\pm)}{t^2}\,ds_+\, ds_-\, d\theta\,,
\end{equation*}
\begin{multline*}
    F^4_\pm(h) (t,\eta) = \int_{r>t^{-a}}\Big[\partial_{s_\pm}^2 I\circ\Psi^{-1} \cdot \nabla h \circ I \circ \Psi^{-1} \frac{|\xi|\,|\xi+\eta|}{2} \\
    + (\partial_{s_\pm} I \circ \Psi^{-1})^{\otimes 2} \cdot \nabla^{\otimes 2}h\circ I\circ \Psi^{-1} \frac{|\xi|\,|\xi+\eta|}{2}\Big] \frac{\cos(ts_\pm)}{t^2}\,ds_+\, ds_-\, d\theta\,.
\end{multline*}
For~\(F^4\), as we have already seen in~\eqref{eq:fstderPsi},
\[
|\partial_{s_\pm} I\circ \Psi^{-1} | \lesssim \frac{(\an\xi + \an{\eta})^2} {|\eta | r}\,.
\]
We also have 
\[
\partial_{s_\pm}^2I\circ\Psi^{-1} = (\pm A - B) \sigma\,,
\]
with
\begin{align*}
    A&=\frac{4r^2|\eta|-\left(s_+-s_-\right)^2|\eta|\pm 2zs_\mp\left(s_+-s_-\right)}{16r^3|\eta|}\,,\\
    B&=\frac{2s^2_\mp r^2|\eta|\mp s_\mp\left(s_+-s_-\right)z|\eta|^2+2 s^2_\mp z^2|\eta|}{8r^3|\eta|^3}\,.
\end{align*}
which implies
\[
|\partial_{s_\pm}^2 I \circ \Psi^{-1}| \lesssim \frac{(\an{\xi} + \an{\eta})^4}{r^3 |\eta|^2} \leq t^{3a}\frac{(\an{\xi} + \an{\eta})^4}{ |\eta|^2}\,.
\]
Therefore, since~\(s_-\) is integrated over~\((-|\eta|, |\eta|)\), we finally get
\[
\Big| F^4_\pm(t,\eta) \Big| \lesssim \|h\|_{8,2} |\eta|^{-1} t^{-2 + 3a}\,.
\]
On the other hand,
\[
\Big| F^3_\pm(t,\eta) \Big| \lesssim \|h\|_{6,1} |\eta|^{-1} t^{-2 +a}\,,
\]
which concludes the first part of the proof.\par

We now prove that
\begin{align*}
    \|\overline F_\pm(h)\|_{L^\infty} \lesssim &\|\,|\xi|\,|\xi+\eta|\,h\|_{6,2}\,,\\
    \|\widetilde F_\pm(h)\|_{L^1,L^\infty} \lesssim &\|\,|\xi|\,|\xi+\eta|\,h\|_{6,2}\,.
\end{align*}
The proof follows exactly the same previous steps with extra singularities at low frequencies. We repeat it setting 
\[
g(\xi,\eta) = |\xi|\, |\xi+\eta| \, h(\xi,\eta)
\]
and without reintroducing the previous notations.\par

Since
\[
\|\an{\eta}|\eta|\,F_\pm (h)\|_{L^\infty_{t,\eta}} \lesssim \|\an\eta \frac{|\eta|}{|\xi|\,|\xi+\eta|} g\|_{L^1_\xi L^\infty_\eta} \lesssim \big\| \an \eta\left( \frac1{|\xi|} + \frac1{|\xi+\eta|} \right) g \big\|_{L^1_\xi L^\infty_\eta} \lesssim \|g\|_{4,0}\,,
\]
we deduce
\[
\|\an\eta |\eta|F_\pm^\mu (h)\|_{L^1_tL^\infty_{\eta}} \leq \|\mu\|_{L^1} \|\,g\|_{4,0}\,.
\]\par

As before, assume now that~\(\eta \neq 0\). We perform the change of variables described by the diffeormorphism~\(\Psi\) in Lemma~\ref{lemma:diffeo} and get
\[
2|\eta|F_\pm^{\overline\mu} (h) (t,\eta) = \overline \mu(t) \int_{(|\eta|,\infty) \times (-|\eta|,|\eta|) \times (0,2\pi)} g(I(\Psi^{-1}(s_+,s_-),\theta),\eta)  \cos(s_\pm t)\, ds_+\, ds_-\, d\theta\,.
\]
We proceed with a first integration by parts to obtain
\begin{align*}
2 |\eta| F_+^{\overline\mu} (h) (t,\eta) = & \overline \mu(t) \int_{ (-|\eta|,|\eta|) \times (0,2\pi)} \Big[ g(I(\Psi^{-1}(s_+,s_-),\theta),\eta)  \frac{\sin(s_+ t)}{t}\Big]_{|\eta|}^\infty\, ds_-\,d\theta\\
& - \overline \mu(t) \int_{ (|\eta|,\infty) \times (-|\eta|,|\eta|) \times (0,2\pi)} \partial_{s_+}\Big[g(I(\Psi^{-1}(s_+,s_-),\theta),\eta) \Big] \frac{\sin(s_+ t)}{t}\, ds _+\, ds_-\, d\theta\,, \\
2|\eta|F_-^{\overline\mu} (h) (t,\eta) = & \overline \mu(t) \int_{ (|\eta|,\infty) \times (0,2\pi)} \Big[ g(I(\Psi^{-1}(s_+,s_-),\theta),\eta)  \frac{\sin(s_- t)}{t}\Big]_{-|\eta|}^{|\eta|}\, ds_+\, d\theta \\
& - \overline \mu(t) \int_{ (|\eta|,\infty) \times (-|\eta|,|\eta|) \times (0,2\pi)} \partial_{s_-}\Big[g(I(\Psi^{-1}(s_+,s_-),\theta),\eta) \Big] \frac{\sin(s_- t)}{t}\, ds _+\, ds_-\, d\theta\,.
\end{align*}
Thanks to Lemma~\ref{lemma:diffeo} and the polynomial decrease of~\(g\), we have
\begin{align*}
 \Big[ g(I(\Psi^{-1}(s_+,s_-),\theta),\eta)  \frac{\sin(s_+ t)}{t}\Big]_{|\eta|}^\infty  = & - g(\frac{s_- - |\eta|}{2} \widehat \eta,\eta)  \frac{\sin(|\eta| t)}{t}\,, \\
 \Big[ g(I(\Psi^{-1}(s_+,s_-),\theta),\eta) \frac{\sin(s_- t)}{t}\Big]_{-|\eta|}^{|\eta|} = & \frac{\sin(|\eta|t)}{t}  \Big( g(\frac{s_+ - |\eta|}{2}\widehat \eta,\eta) + g(\frac{-s_+ - |\eta|}{2}\widehat \eta,\eta) \Big)\,.
\end{align*}
Recalling that
\begin{align*}
2\overline F_+(h)(\eta) = & -\int_{(-|\eta|,|\eta|) \times (0,2\pi)} g(\frac{s_- - |\eta|}{2} \widehat \eta,\eta)\, ds_-\, d\theta\,, \\ 
2\overline F_-(h)(\eta) =& \int_{(|\eta|,\infty) \times (0,2\pi)}  \Big( g(\frac{s_+ - |\eta|}{2}\widehat \eta,\eta) + g(\frac{-s_+ - |\eta|}{2}\widehat \eta,\eta) \Big)\, ds_+\, d\theta\,,
\end{align*}
we infer 
\[
\|\an{\eta}^2 \overline F_\pm (h)\|_{L^\infty_\eta} \lesssim \|g\|_{4,0}\,,
\]
which implies
\[
\| \an\eta^2 \mu(t) \frac{\sin(|\eta|t)}{|\eta| t} \overline F_\pm(h)\|_{L^1_tL^\infty_\eta} \lesssim \|g\|_{4,0}.
\]
Hence, it remains to prove that~\(|\eta| F^1_\pm(h)\) belongs to~\(L^1_tL^\infty_\eta\), with
\begin{align*}
F^1_\pm(h) (t,\eta) = & -   \overline \mu(t) \int_{ (|\eta|,\infty) \times (-|\eta|,|\eta|) \times (0,2\pi)} \partial_{s_\pm}\Big[g(I(\Psi^{-1}(s_+,s_-),\theta),\eta) \Big] \frac{\sin(s_\pm t)}{t}\, ds _+\, ds_-\, d\theta\,. 
\end{align*}
Set~\(a \in (0,\frac13)\). We split the integral over the domains given respectively by~\(r\leq t^{-a}\) and by~\(r>t^{-a}\). 
Recall
\begin{align}
\label{eq:F-smaller-bis} F^\leq_\pm(h) (t,\eta) = & -  \overline \mu(t) \int_{ r\leq t^{-a}} \partial_{s_\pm}\Big[g(I(\Psi^{-1}(s_+,s_-),\theta),\eta) \Big] \frac{\sin(s_\pm t)}{t}\, ds _+\, ds_-\, d\theta\,, \\
\label{eq:F-bigger-bis}F^>_\pm(h) (t,\eta) = & - \overline \mu(t) \int_{ r> t^{-a}} \partial_{s_\pm}\Big[g(I(\Psi^{-1}(s_+,s_-),\theta),\eta) \Big] \frac{\sin(s_\pm t)}{t}\, ds _+\, ds_-\, d\theta\,.
\end{align}
Concerning the first integral in \eqref{eq:F-smaller-bis}, we perform a reverse change of variable and get
\begin{equation*}
F^\leq_\pm(h) (t,\eta) = - \overline \mu(t) \int_{ r\leq t^{-a}} \partial_{s_\pm}\Big[g(I(\Psi^{-1}(s_+,s_-),\theta),\eta) \Big] \frac{\sin(s_\pm t)}{t} \frac{|\eta|r}{|\xi+\eta|\,|\xi|}\, dr\, dz\, d\theta\, .
\end{equation*}
which implies 
\begin{equation*}
F^\leq_\pm(h) (t,\eta) = - \overline \mu(t) \int_{ r\leq t^{-a}} \partial_{s_\pm} I\circ \Psi^{-1}\cdot \nabla g(I(\Psi^{-1}(s_+,s_-),\theta),\eta) \frac{\sin(s_\pm t)}{t} \frac{|\eta|r}{|\xi+\eta|\,|\xi|}\, dr\, dz\, d\theta\, .
\end{equation*}
Recalling that
\begin{equation}\label{eq:fstderPsi-bis}
    |r |\eta| \partial_{s_\pm} I(\Psi^{-1}) (s_+,s_-),\theta)| \lesssim (\an \xi + \an \eta)^2
\end{equation}
we obtain
\[
|\eta| \, |F^\leq_\pm(h) (t,\eta)| \lesssim t^{-1} \overline \mu(t) \|g\|_{4,1} \int_{ r\leq t^{-a}} \frac{(\an \xi + \an \eta)^{-2} |\eta|}{|\xi+\eta|\,|\xi|}\, dr\, dz\, d\theta\,.
\]
Then, using that~\(\frac{|\eta|}{|\xi|\, |\xi+\eta|} \lesssim \frac{1}{|\xi|} + \frac1{|\xi +\eta|}\) and performing the change of variable~\(\xi \leftarrow - \xi -\eta\) yields
\[
|\eta| \, |F^\leq_\pm(h) (t,\eta)| \lesssim t^{-1} \overline \mu(t) \|g\|_{4,1} \int_{ r\leq t^{-a}} \frac{1 }{\an{\xi}^2|\xi|}\, dr\, dz\,. 
\]
Recalling that~\( |\xi| = \sqrt{z^2 + r^2}\), we finally obtain
\begin{align*}
|\eta| \, |F^\leq_\pm(h) (t,\eta)| & \lesssim t^{-1} \overline \mu(t) \|g\|_{4,1} \int_{ r\leq t^{-a}} \frac{1 }{(1+z^2 + r^2)\sqrt{z^2 + r^2}}\, dr\, dz \\
& \lesssim \frac1{t} \overline \mu(t) \|g\|_{4,1} \int_{ r\leq t^{-a}} \frac{1 }{(1+z^2 )\sqrt{|z| r}}\, dr\, dz \\
&\lesssim t^{-1-a/2} \overline \mu(t) \|g\|_{4,1}\,.
\end{align*}
which implies
\[
\||\eta| F^\leq_\pm(h)  \|_{L^1_t,L^\infty_\eta} \lesssim \|g\|_{4,1}\,.
\]\par

We turn to~\(F_\pm^>\) in \eqref{eq:F-bigger-bis}. Here, we proceed with an additional integration by parts to get
\begin{align}
\notag F^>_+(h) (t,\eta) = &  \overline \mu(t) \int_{ (-|\eta|,\eta) \times (0,2\pi)} \Big[ \partial_{s_+}\Big[g(I(\Psi^{-1}(s_+,s_-),\theta),\eta) \Big] \frac{\cos(s_+ t)}{t^2}\Big]_{T_+(t)}^\infty\, ds_-\, d\theta \\
&\label{eq:F-plus-bigger-2derivative-bis} - \overline \mu(t) \int_{ r> t^{-a}} \partial_{s_+}^2\Big[g(I(\Psi^{-1}(s_+,s_-),\theta),\eta) \Big] \frac{\cos(s_+ t)}{t^2}\, ds _+\, ds_- d\theta\,, \\
\notag F^>_-(h) (t,\eta) = & \overline \mu(t) \int_{ (\sqrt{|\eta|^2 + 4t^{-2a}},\infty)\times (0,2\pi)} \Big[\partial_{s_-}\Big[g(I(\Psi^{-1}(s_+,s_-),\theta),\eta) \Big] \frac{\cos(s_- t)}{t^2}\Big]_{-T_-(t)}^{T_-(t)}\, ds _+\, d\theta \\
&\label{eq:F-minus-bigger-2derivative-bis}- \overline \mu(t) \int_{ r> t^{-a}} \partial_{s_-}^2\Big[g(I(\Psi^{-1}(s_+,s_-),\theta),\eta) \Big] \frac{\cos(s_- t)}{t^2}\, ds _+\, ds_-\, d\theta\, .
\end{align}

Recall that we have seen in \eqref{eq:h-leibniz-aux} that 
\begin{align*}
 \partial_{s_\pm}\Big[g(I(\Psi^{-1}(s_+,s_-),\theta),\eta) \Big]    =  & \left( \frac{s_\mp}{2|\eta|}\widehat{\eta}\pm\frac{s_+-s_-}{4r}\sigma-\frac{zs_\mp}{2r|\eta|}\sigma\right) \cdot \nabla_\xi g(\xi,\eta)\, .
\end{align*}
Therefore, the boundary terms respectively become
\begin{multline*}
\Big[ \partial_{s_+}\Big[g(I(\Psi^{-1}(s_+,s_-),\theta),\eta)\Big] \frac{\cos(s_+ t)}{t^2}\Big]_{T_+(t)}^\infty\\
=- \left(\left( \frac{s_-}{2|\eta|}\widehat{\eta}+\frac{T_+(t)-s_-}{4r}\sigma-\frac{zs_-}{2r|\eta|}\sigma\right) \cdot \nabla_\xi g(\xi,\eta) \right) \frac{\cos(T_+(t) t)}{t^2}\,,
\end{multline*}
and
\begin{multline*}
\Big[\partial_{s_-}\Big[g(I(\Psi^{-1}(s_+,s_-),\theta),\eta) \Big] \frac{\cos(s_- t)}{t^2}\Big]_{-T_-(t)}^{T_-(t)} 
= \\
\Bigg(\left(\left( \frac{s_+}{2|\eta|}\widehat{\eta}+\frac{s_+-s_-}{4r}\sigma-\frac{zs_+}{2r|\eta|}\sigma\right) \cdot \nabla_\xi g(\xi,\eta)\right)\Bigg|_{s_-= T_-(t)}\\
- \left(\left( \frac{s_+}{2|\eta|}\widehat{\eta}+\frac{s_+-s_-}{4r}\sigma-\frac{zs_+}{2r|\eta|}\sigma\right) \cdot \nabla_\xi g(\xi,\eta) \right)\Bigg|_{s_-= -T_-(t)} \Bigg) \frac{\cos(T_-(t) t)}{t^2}\,.
\end{multline*}
We infer that
\begin{equation*}
    \Big||\eta|\,\overline \mu(t) \int_{ (-|\eta|,\eta) \times (0,2\pi)} \Big[ \partial_{s_+}\Big[g(I(\Psi^{-1}(s_+,s_-),\theta),\eta) \Big] \frac{\cos(s_+ t)}{t^2}\Big]_{S_+(t)}^\infty  ds_- d\theta \Big|\lesssim \|g\|_{4,1} \overline\mu(t) t^{-2 + a}\,,   
\end{equation*}
\begin{multline*}
    \Big||\eta|\,\overline \mu(t) \int_{ (\sqrt{|\eta|^2 + 4t^{-2a}},\infty)\times (0,2\pi)} \Big[\partial_{s_-}\Big[g(I(\Psi^{-1}(s_+,s_-),\theta),\eta) \Big] \frac{\cos(s_- t)}{t^2}\Big]_{-S_-(t)}^{S_-(t)} ds _+  d\theta\Big|\\
\lesssim \|g\|_{4,1} \overline \mu(t) t^{-2 + a}\,.
\end{multline*}
Going back to \eqref{eq:F-plus-bigger-2derivative-bis} and \eqref{eq:F-minus-bigger-2derivative-bis}, by the Leibniz rule, we compute
\begin{equation*}
\partial_{s_\pm}^2 \Big[ g\circ I \circ \Psi^{-1} \Big] = \partial_{s_\pm}^2 I\circ\Psi^{-1} \cdot \nabla g \circ I \circ \Psi^{-1}+ (\partial_{s_\pm} I \circ \Psi^{-1})^{\otimes 2} \cdot \nabla^{\otimes 2}g\circ I\circ \Psi^{-1}\,.
\end{equation*}
Setting
\[
F^2_\pm (h) (t,\eta) = \int_{r>t^{-a}} \partial_{s_\pm}^2 \Big[g\circ I \circ \Psi^{-1} \Big] \frac{\cos(ts_\pm)}{t^2}\, ds_+\, ds_-\, d\theta\,,
\]
we have
\begin{multline*}
    F^2_\pm(h) (t,\eta) = \int_{r>t^{-a}}\Big[\partial_{s_\pm}^2 I\circ\Psi^{-1} \cdot \nabla g \circ I \circ \Psi^{-1}  
    + (\partial_{s_\pm} I \circ \Psi^{-1})^{\otimes 2} \cdot \nabla^{\otimes 2}g\circ I\circ \Psi^{-1} \Big] \frac{\cos(ts_\pm)}{t^2}\, ds_+\, ds_-\, d\theta\,.
\end{multline*}
Starting from \eqref{eq:fstderPsi-bis}, we also get
\begin{align*}
|\partial_{s_\pm} I\circ \Psi^{-1} | \lesssim & \frac{(\an\xi + \an\eta)^2)} {|\eta | r} \leq t^a \frac{(\an \xi + \an \eta)^2}{|\eta|}\,,\\
|\partial_{s_\pm}^2 I \circ \Psi^{-1}| \lesssim & \frac{(\an{\xi} + \an{\eta})^4}{r^3 |\eta|^2} \leq t^{3a}\frac{(\an{\xi} + \an{\eta})^4}{ |\eta|^2}\,.
\end{align*}
Therefore, since~\(s_-\) is integrated over~\((-|\eta|, |\eta|)\), we finally get
\[
\Big| F^2_\pm(t,\eta) \Big| \lesssim \|g\|_{6,2} |\eta|^{-1} t^{-2 + 3a}\,,
\]
which concludes the proof.
\end{proof}
\begin{lemma}\label{stat:h-sin-decomposition}
Let~\(G_\pm\) be a linear map defined from~\(L^1_\xi L^\infty_\eta\) to~\(L^\infty_{t,\eta}\) as 
\[
G_\pm(h)(t,\eta) = \int_{\R^3} h(\xi,\eta) \sin(\Phi_\pm(\xi,\eta)) \, d\xi\, .
\]
There exists two linear maps~\(\overline G_\pm,\widetilde G_\pm\) such that~\(\overline G_\pm\) and~\(\widetilde G_\pm\) are continuous from~\(\frac{1}{(|\xi|^{-1}+|\eta +\xi|^{-1})}\mathcal S_{7,2}\) to respectively~\(\an{\eta}^{-1}L^\infty(\R^3)\) and~\(L^1(\R,L^\infty(\R^3))\) and satisfy, for all~\(t\in \R\) and all~\(\eta\in \R^3\),
\[
G_\pm(h)(t,\eta) = \frac{\sin(|\eta|t) - t |\eta| \cos(|\eta|t)}{|\eta|^2t^2} \overline G_\pm(h)(\eta) + |\eta|^{-2} \widetilde G_\pm(h)(t,\eta)\,.
\]
In other words, we have
\begin{align*}
    \|\an{\eta} \overline G_\pm(h)\|_{L^\infty} \lesssim &\min(\|\,|\xi|\,h\|_{m_1,2}, \|\,|\xi+\eta|\,h\|_{7,2})\,,\\
    \|\widetilde G_\pm(h)\|_{L^1,L^\infty} \lesssim &\min(\|\,|\xi|\,h\|_{m_1,2}, \|\,|\xi+\eta|\,h\|_{7,2})\,.
\end{align*}
\end{lemma}

\begin{proof} Set~\(g(\xi,\eta) = |\xi| h(\xi,\eta)\). As in the proof of Lemma~\ref{stat:h-cos-decomposition}, we introduce a cut-off function~\(\mu\) and define
\[
G_\pm^\mu = \mu(t) G_\pm\,,\quad G_\pm^{\overline \mu} = \overline \mu(t) G_\pm\,.
\]
Since
\[
\|\an{\eta}^2\,G_\pm (h)\|_{L^\infty_{t,\eta}} \lesssim \| \frac{\an{\eta}^2}{|\xi|} g\|_{L^1_\xi L^\infty_\eta}  \lesssim \|g\|_{5,0}\,,
\]
we deduce
\[
\|\an{\eta}^2G_\pm^\mu (h)\|_{L^1_tL^\infty_{\eta}} \leq \|\mu\|_{L^1} \|\,g\|_{5,0}\,.
\]\par

Assume that~\(\eta \neq 0\). Performing a change of variables through the diffeormorphism~\(\Psi\) defined in Lemma~\ref{lemma:diffeo} yields
\[
4|\eta|G_\pm^{\overline\mu} (h) (t,\eta) = \overline \mu(t) \int_{(|\eta|,\infty) \times (-|\eta|,|\eta|) \times (0,2\pi)} g(I(\Psi^{-1}(s_+,s_-),\theta),\eta)(s_+ + s_-)  \sin(s_\pm t)\, ds_+\, ds_-\, d\theta\,.
\]
We then proceed with a first integration by parts and get
\begin{align*}
4 |\eta| G_+^{\overline\mu} (h) (t,\eta) = & -\overline \mu(t) \int_{ (-|\eta|,|\eta|) \times (0,2\pi)} \Big[ g(I(\Psi^{-1}(s_+,s_-),\theta),\eta) (s_+ +s_-)  \frac{\cos(s_+ t)}{t}\Big]_{|\eta|}^\infty\, ds_-\,d\theta\\
& + \overline \mu(t) \int_{ (|\eta|,\infty) \times (-|\eta|,|\eta|) \times (0,2\pi)} \partial_{s_+}\Big[g(I(\Psi^{-1}(s_+,s_-),\theta),\eta) (s_+ +s_-) \Big] \frac{\cos(s_+ t)}{t}\, ds _+\, ds_-\, d\theta\,, \\
4|\eta|G_-^{\overline\mu} (h) (t,\eta) = & -\overline \mu(t) \int_{ (|\eta|,\infty) \times (0,2\pi)} \Big[ g(I(\Psi^{-1}(s_+,s_-),\theta),\eta) (s_+ +s_-) \frac{\cos(s_- t)}{t}\Big]_{-|\eta|}^{|\eta|}\, ds_+\, d\theta \\
& + \overline \mu(t) \int_{ (|\eta|,\infty) \times (-|\eta|,|\eta|) \times (0,2\pi)} \partial_{s_-}\Big[g(I(\Psi^{-1}(s_+,s_-),\theta),\eta) (s_+ +s_-)\Big] \frac{\cos(s_- t)}{t}\, ds _+\, ds_-\, d\theta\,.
\end{align*}
Thanks to Lemma~\ref{lemma:diffeo} and the polynomial decrease of~\(g\), we have
\begin{equation*}
 \Big[ g(I(\Psi^{-1}(s_+,s_-),\theta),\eta) (s_+ +s_-) \frac{\cos(s_+ t)}{t}\Big]_{|\eta|}^\infty  = - g(\frac{s_- - |\eta|}{2} \widehat \eta,\eta) (|\eta| + s_-)  \frac{\cos(|\eta| t)}{t}\,,
\end{equation*}
\begin{multline*}
    \Big[ g(I(\Psi^{-1}(s_+,s_-),\theta),\eta) (s_+ +s_-) \frac{\cos(s_- t)}{t}\Big]_{-|\eta|}^{|\eta|} \\
    = \frac{\cos(|\eta|t)}{t}  \Big( g(\frac{s_+ - |\eta|}{2}\widehat \eta,\eta)(s_+ + |\eta|) - g(\frac{-s_+ - |\eta|}{2}\widehat \eta,\eta)  (s_+ - |\eta|) \Big)\,.
\end{multline*}
Setting
\begin{align*}
4\overline G_+(h)(\eta) = & \int_{ (-|\eta|,|\eta|) \times (0,2\pi)} g(\frac{s_- - |\eta|}{2} \widehat \eta,\eta) (|\eta| + s_-)\, ds_-\, d\theta\,, \\ 
4\overline G_-(h)(\eta) =&  \int_{ (|\eta|,\infty) \times (0,2\pi)} \Big( g(\frac{s_+ - |\eta|}{2}\widehat \eta,\eta)(s_+ + |\eta|) - g(\frac{-s_+ - |\eta|}{2}\widehat \eta,\eta)  (s_+ - |\eta|) \Big)\, ds_+\, d\theta\,,
\end{align*}
we deduce 
\[
\|\an{\eta}^2 \overline G_\pm (h)\|_{L^\infty_\eta} \lesssim \|g\|_{5,0}\,.
\]
Additionally, we have
\[
G^{\overline\mu}_\pm(h) = -\overline \mu \frac{\cos(|\eta|t)}{t |\eta|} \overline G_\pm(h)  + |\eta|^{-2} \overline \mu(t) G^1_\pm (h)\,,
\]
with
\[
G^1_\pm(h) = |\eta |\overline \mu(t) \int_{ (|\eta|,\infty) \times (-|\eta|,|\eta|) \times (0,2\pi)} \partial_{s_\pm}\Big[g(I(\Psi^{-1}(s_+,s_-),\theta),\eta) (s_+ +s_-) \Big] \frac{\cos(s_\pm t)}{t}\, ds _+\, ds_-\, d\theta\, .
\]
Now, we observe that
\[
-\frac{\cos(|\eta|t)}{t |\eta|} = \frac{\sin(|\eta| t) - t|\eta|\cos(|\eta|t)}{t^2|\eta|^2}  - \frac{\sin(|\eta| t)}{|\eta|^2 t^2}
\]
and
\[
\|\, |\eta|^2 \frac{\sin(|\eta| t)}{|\eta|^2 t^2} \overline \mu(t) \overline G_\pm (h)\|_{L^1_tL^\infty_\eta} \lesssim \|g\|_{5,0}\,.
\]
Since~\(\frac{\sin(|\eta| t) - t|\eta|\cos(|\eta|t)}{t^2|\eta|^2}\) is bounded, we obtain
\[
\| |\eta|^2 \frac{\sin(|\eta| t) - t|\eta|\cos(|\eta|t)}{t^2|\eta|^2} \mu \overline G_\pm(h)\|_{L^1_t,L^\infty_\eta} \lesssim \|g\|_{5,0}\,.
\]
Therefore, one has
\begin{multline*}
    G^{\overline \mu}_\pm(h) = \frac{\sin(|\eta| t) - t|\eta|\cos(|\eta|t)}{t^2|\eta|^2} \overline G_\pm(h) + \mu \frac{\sin(|\eta| t) - t|\eta|\cos(|\eta|t)}{t^2|\eta|^2} \overline G_\pm(h) \\
    + \overline \mu \frac{\sin(|\eta| t)}{|\eta|^2 t^2} \overline G_\pm(h) + |\eta|^{-2} G^1_\pm(h)
\end{multline*}
and it remains to prove that~\(G^1_\pm(h)\) belongs to~\(L^1_t L^\infty_\eta\).\par

Set~\(a \in (0,\frac13)\) and split the integral over the domains respectively given by~\(r\leq t^{-a}\) and~\(r>t^{-a}\). Define
\begin{align}
\label{eq:G-smaller} G^\leq_\pm(h) (t,\eta) = & |\eta|  \overline \mu(t) \int_{ r\leq t^{-a}} \partial_{s_\pm}\Big[g(I(\Psi^{-1}(s_+,s_-),\theta),\eta) (s_++s_-) \Big] \frac{\cos(s_\pm t)}{t}\, ds _+\, ds_-\, d\theta\,, \\
\label{eq:G-bigger} G^>_\pm(h) (t,\eta) = & |\eta| \overline \mu(t) \int_{ r> t^{-a}} \partial_{s_\pm}\Big[g(I(\Psi^{-1}(s_+,s_-),\theta),\eta)(s_++s_-) \Big] \frac{\cos(s_\pm t)}{t}\, ds _+\, ds_-\, d\theta\,.
\end{align}
Concerning the first integral in \eqref{eq:G-smaller}, we proceed with a reverse change of variable to get
\begin{align*}
G^\leq_\pm(h) (t,\eta) = & \overline \mu(t) \int_{ r\leq t^{-a}} \partial_{s_\pm}\Big[g(I(\Psi^{-1}(s_+,s_-),\theta),\eta) (s_++s_-) \Big] \frac{\cos(s_\pm t)}{t} \frac{|\eta|^2r}{|\xi+\eta|\,|\xi|}\, dr\, dz\, d\theta\, .
\end{align*}
Thus, we obtain 
\begin{multline*}
G^\leq_\pm(h) (t,\eta) =  \overline \mu(t) \int_{ r\leq t^{-a}} \partial_{s_\pm} I\circ \Psi^{-1}\cdot \nabla g(\xi,\eta) \frac{\cos(s_\pm t)}{t} \frac{|\eta|^2r}{|\xi+\eta|}\, dr\, dz\, d\theta \\
+ \overline \mu(t) \int_{ r\leq t^{-a}}  g(\xi,\eta) \frac{\cos(\Phi_\pm t)}{t} \frac{|\eta|^2}{|\xi+\eta|\;|\xi|} r\, dr\, dz\, d\theta\,.
\end{multline*}
For the integral in the second line of the right-hand side, recalling that
\[
\frac{|\eta|^2}{|\xi+\eta|\;|\xi|} \leq |\eta|\Big( \frac1{|\xi|} + \frac{1}{|\xi+\eta|}\Big)\,,
\] 
we have
\begin{align*}
    \int_{ r\leq t^{-a}} |g(\xi,\eta)| \frac{|\eta|^2}{|\xi+\eta|\;|\xi|} d\xi &\lesssim  \|g\|_{2,0} \int_{r\leq t^{-a}} \frac{d\xi}{|\xi| \an{\xi}} \\
    &\lesssim \|g\|_{2,0} \int_{r\leq t^{-a}} \frac{r\, dr\, dz}{\sqrt{|z|r} \sqrt{1+z^2+r^2}}\\
    &\lesssim \|g\|_{2,0} t^{-3a/2}\,.
\end{align*}
For the integral in the first line, recalling~\eqref{eq:fstderPsi},
\[
|r |\eta| \partial_{s_\pm} I(\Psi^{-1}) (s_+,s_-),\theta)| \lesssim (\an \xi + \an \eta)^2\,,
\]
one has
\begin{align*}
\int_{ r\leq t^{-a}} |\partial_{s_\pm} I\circ \Psi^{-1}\cdot \nabla g(\xi,\eta)|&  \frac{|\eta|^2r}{|\xi+\eta|}\, dr\, dz\, d\theta\\
&\lesssim \int_{r\leq t^{-a}} |\nabla g| (\an{\xi} + \an{\eta})^2 \frac{|\eta|}{|\xi +\eta|}\, dr\, dz\, d\theta \\
&\lesssim \|g\|_{4,1} \int_{r\leq t^{-a}} \frac{dr\,dz}{\sqrt{r|z+|\eta||}\sqrt{1+(z+|\eta|)^2 +r^2}}\\
&\lesssim t^{-a/2} \|g\|_{4,1}\,.
\end{align*}
Therefore, we deduce
\[
\|G^\leq_\pm(h) \|_{L^1_t L^2_\eta} \lesssim \|g\|_{4,1}\,.
\]\par

We now go back to~\(G_\pm^>\) in \eqref{eq:G-bigger}. Here, we proceed with a second integration by parts to get
\begin{align*}
G^>_+(h) (t,\eta) = &   |\eta| \overline \mu(t) \int_{ (-|\eta|,\eta) \times (0,2\pi)} \Big[ \partial_{s_+}\Big[g(I(\Psi^{-1}(s_+,s_-),\theta),\eta) (s_++s_-) \Big] \frac{\sin(s_+ t)}{t^2}\Big]_{T_+(t)}^\infty\, ds_-\, d\theta \\
& - |\eta| \overline \mu(t) \int_{ r> t^{-a}} \partial_{s_+}^2\Big[g(I(\Psi^{-1}(s_+,s_-),\theta),\eta) (s_++s_-)\Big] \frac{\sin(s_+ t)}{t^2}\, ds _+\, ds_-\, d\theta\,, \\
G^>_-(h) (t,\eta) = & |\eta| \overline \mu(t) \int_{ (\sqrt{|\eta|^2 + 4t^{-2a}},\infty)\times (0,2\pi)} \Big[\partial_{s_-}\Big[g(I(\Psi^{-1}(s_+,s_-),\theta),\eta)(s_++s_-) \Big] \frac{\sin(s_- t)}{t^2}\Big]_{-T_-(t)}^{T_-(t)}\, ds _+\, d\theta \\
&- |\eta|  \overline \mu(t) \int_{ r> t^{-a}} \partial_{s_-}^2\Big[g(I(\Psi^{-1}(s_+,s_-),\theta),\eta)(s_++s_-) \Big] \frac{\sin(s_- t)}{t^2}\, ds _+\, ds_-\, d\theta\,,
\end{align*}
where we recall
\begin{align*}
    T_+(t)  =& \left(\frac{4 t^{-2 a} + |\eta|^2 -s_-^2}{1 - s_-^2/|\eta|^2}\right)^{1/2}\,, \\
    T_-(t) = &  \left(\frac{s_+^2 - 4 t^{-2 a} - |\eta|^2 }{s_+^2/|\eta|^2 - 1}\right)^{1/2}\,.
\end{align*}
We have 
\begin{multline*}
 \partial_{s_\pm}\Big[g(I(\Psi^{-1}(s_+,s_-),\theta),\eta) (s_++s_-) \Big]    = (s_++s_-)\left( \frac{s_\mp}{2|\eta|}\widehat{\eta}\pm\frac{s_+-s_-}{4r}\sigma-\frac{zs_\mp}{2r|\eta|}\sigma\right) \cdot \nabla_\xi g(\xi,\eta) \\
 + g(\xi,\eta)\,,
\end{multline*}
which makes the boundary terms respectively become
\begin{multline*}
\Big[ \partial_{s_+}\Big[g(I(\Psi^{-1}(s_+,s_-),\theta),\eta) (s_++s_-)\Big] \frac{\sin(s_+ t)}{t^2}\Big]_{T_+(t)}^\infty = \\
- (T_+(t) +s_-)\left(\left( \frac{s_-}{2|\eta|}\widehat{\eta}+\frac{T_+(t) -s_-}{4r}\sigma-\frac{zs_-}{2r|\eta|}\sigma\right) \cdot \nabla_\xi g(\xi,\eta) \right) \frac{\sin(T_+(t) t)}{t^2} + g(\xi,\eta) \frac{T_+(t) t}{t^2}\,,
\end{multline*}
and
\begin{multline*}
\Big[\partial_{s_-}\Big[g(I(\Psi^{-1}(s_+,s_-),\theta),\eta) (s_++s_-) \Big] \frac{\sin(s_- t)}{t^2}\Big]_{-T_-(t)}^{T_-(t)} 
= \\
\Bigg(\left( (s_++s_-)\left( \frac{s_+}{2|\eta|}\widehat{\eta}+\frac{s_+-s_-}{4r}\sigma-\frac{zs_+}{2r|\eta|}\sigma\right) \cdot \nabla_\xi g(\xi,\eta) + g(\xi,\eta)\right)\Bigg|_{s_-= T_-(t)}\\
+ \left((s_++s_-)\left( \frac{s_+}{2|\eta|}\widehat{\eta}+\frac{s_+-s_-}{4r}\sigma-\frac{zs_+}{2r|\eta|}\sigma\right) \cdot \nabla_\xi g(\xi,\eta) + g(\xi,\eta)\right)\Bigg|_{s_-= -T_-(t)} \Bigg) \frac{\sin(T_-(t) t)}{t^2}\,.
\end{multline*}
We deduce that
\begin{multline*}
    \Big||\eta|\,\overline \mu(t) \int_{ (-|\eta|,\eta) \times (0,2\pi)} \Big[ \partial_{s_+}\Big[g(I(\Psi^{-1}(s_+,s_-),\theta),\eta)(s_++s_-) \Big] \frac{\sin(s_+ t)}{t^2}\Big]_{T_+(t)}^\infty\, ds_-\, d\theta \Big|\\
    \lesssim \|g\|_{4,1} \overline\mu(t) t^{-2 + a}\,,
\end{multline*}
\begin{multline*}
    \Big||\eta|\,\overline \mu(t) \int_{ (\sqrt{|\eta|^2 + 4t^{-2a}},\infty)\times (0,2\pi)} \Big[\partial_{s_-}\Big[g(I(\Psi^{-1}(s_+,s_-),\theta),\eta)(s_++s_-) \Big] \frac{\cos(s_- t)}{t^2}\Big]_{-T_-(t)}^{T_-(t)}\, ds _+\, d\theta \Big|\\
    \lesssim \|g\|_{5,1} \overline \mu(t) t^{-2 + a}\,.
\end{multline*}
We also have
\begin{align*}
\partial_{s_\pm}^2 \Big[ g\circ I \circ \Psi^{-1} (s_++s_-)\Big] = &(s_++s_-)\partial_{s_\pm}^2 I\circ\Psi^{-1} \cdot \nabla g \circ I \circ \Psi^{-1} \\
& + (s_++s_-) (\partial_{s_\pm} I \circ \Psi^{-1})^{\otimes 2} \cdot \nabla^{\otimes 2}g\circ I\circ \Psi^{-1}  \\
& + 2\partial_{s_\pm} I \circ \Psi^{-1} \cdot \nabla g \circ I \circ \Psi^{-1}\,.
\end{align*}
ans setting
\[
G^2_\pm (h) (t,\eta) = \int_{r>t^{-a}} \partial_{s_\pm}^2 \Big[g\circ I \circ \Psi^{-1} (s_++s_-)\Big] \frac{\sin(ts_\pm)}{t^2}\, ds_+\, ds_-\, d\theta\,,
\]
we get
\begin{multline*}
    G^2_\pm(h) (t,\eta) = \\
    |\eta|\int_{r>t^{-a}} (s_++s_-)\Big[\partial_{s_\pm}^2 I\circ\Psi^{-1} \cdot \nabla g \circ I \circ \Psi^{-1}  
    + |\eta |(\partial_{s_\pm} I \circ \Psi^{-1})^{\otimes 2} \cdot \nabla^{\otimes 2}g\circ I\circ \Psi^{-1} \Big] \frac{\sin(ts_\pm)}{t^2}\, ds_+\, ds_-\, d\theta \\
    + \int_{r>t^{-a}} \partial_{s_\pm} I\circ\Psi^{-1} \cdot \nabla g\frac{\sin(ts_\pm)}{t^2}\, ds_+\, ds_-\, d\theta\,.
\end{multline*}
As we have already seen,
\begin{align*}
|\partial_{s_\pm} I\circ \Psi^{-1} | \lesssim & \frac{(\an\xi + \an\eta)^2)} {|\eta | r} \leq t^a \frac{(\an \xi + \an \eta)^2}{|\eta|}\,,\\
|\partial_{s_\pm}^2 I \circ \Psi^{-1}| \lesssim & \frac{(\an{\xi} + \an{\eta})^4}{r^3 |\eta|^2} \leq t^{3a}\frac{(\an{\xi} + \an{\eta})^4}{ |\eta|^2}\,.
\end{align*}
Therefore, since~\(s_-\) is integrated over~\((-|\eta|, |\eta|)\), we get
\[
\Big| G^2_\pm(t,\eta) \Big| \lesssim \|g\|_{7,2} |\eta|^{-1} t^{-2 + 3a}\,.
\]
Notice that by the change of variables~\(\xi \leftarrow -\xi - \eta\), we have
\[
G_\pm(h)(t,\eta) = \pm \int_{\R^3} h(-\xi-\eta,\eta) \sin(\Phi_\pm(\xi,\eta))\, d\xi\, .
\]
Setting~\(\widetilde h(\xi,\eta) = h(-\xi - \eta,\eta)\), we get
\[
G_\pm (h) = \pm G_\pm (\widetilde h)\,.
\]
We also have 
\begin{align*}
8\overline G_+(h)(\eta) = & \int_{ (-|\eta|,|\eta|) \times (0,2\pi)} h(\frac{s_- - |\eta|}{2} \widehat \eta,\eta) (|\eta|^2 - s_-^2)\, ds_-\, d\theta\,, \\ 
8\overline G_-(h)(\eta) =&  \int_{ (|\eta|,\infty) \times (0,2\pi)} \Big( h(\frac{s_+ - |\eta|}{2}\widehat \eta,\eta) - h(\frac{-s_+ - |\eta|}{2}\widehat \eta,\eta)   \Big) (s_+^2 - |\eta|^2)\, ds_+\, d\theta\,.
\end{align*}
Noticing that
\begin{align*}
h(\frac{s_- - |\eta|}{2} \widehat \eta,\eta) = & \widetilde h (\frac{-s_- - |\eta|}{2} \widehat \eta,\eta)\,, \\
h(\frac{s_+ - |\eta|}{2}\widehat \eta,\eta) = & \widetilde h(\frac{-s_+ - |\eta|}{2}\widehat \eta,\eta)\,, \\
h(\frac{-s_+ - |\eta|}{2}\widehat \eta,\eta) = & \widetilde h(\frac{s_+ - |\eta|}{2}\widehat \eta,\eta)\,,
\end{align*}
the change of variable~\(s_- \leftarrow -s_-\) in the first integral yields
\[
\overline G_\pm (h) = \pm \overline G(\widetilde h)\,,
\]
which implies
\[
\widetilde G_\pm(h) = \pm \widetilde G(\widetilde h)\,.
\]
Finally, we deduce
\begin{align*}
    \|\overline G_\pm (h)\|_{L^\infty_\eta} & \lesssim \||\xi |\, \widetilde h\|_{7,2} \lesssim \||\xi+\eta|\, h\|_{7,2}\,,\\
    \|\widetilde G(h)\|_{L^1_t L^\infty_\eta} & \lesssim \||\xi |\, \widetilde h\|_{7,2} \lesssim \||\xi+\eta|\, h\|_{7,2}\,,
\end{align*}
which concludes the proof.
\end{proof}

\subsection{Estimates of \texorpdfstring{\(L_2\) and~\(L_1\)}{L2 and L1}}

In this section, we assume that~\(w\) is real-valued.

\begin{proposition}\label{prop:L2-continuity-of-L2}
Let $g$ be as in~\Cref{rmk:L_2}. Assume 
\begin{itemize}
    \item $\an{\,\cdot\,}^{6} \widehat  w\in L^\infty (\R^3)\,$;
    \item ~\(g : \R^3 \rightarrow \mathcal M_4(\C)\) is an even function of class $\mathrm{C}^3$ and such that
    \[
    \an{\,\cdot\,}^{7} \sum_{|k| \le 3} \partial^k_{\xi} g  \in L^\infty(\R^3)\,.
    \]
\end{itemize}
Then, \(L_2\) defined in~\eqref{eq:L2-fourier-symbol} is bounded from~\(L^2(\R \times \R^3)\) to itself and its operator norm is bounded by
\[
\|\an{\,\cdot\,}^{6}\widehat w\|_\infty \|\an{\,\cdot\,}^{7}\sum_{|k|\leq 3}\partial^k_{\xi} g \|_\infty\,.
\]\par
Furthermore, the operator $L_2^\infty$, defined as
\[
\widehat{L^\infty_2(V)} = \int_\R F(t-\tau,\eta) \widehat V(\tau,\eta)\, d\tau\,, 
\]
with $F$ given by \eqref{eq:L2-F-kernel}, is also bounded from~\(L^2(\R \times \R^3)\) to itself.
Moreover, the operator $L_2$ is bounded from $L^2_t \dot H^\nu_x$ to itself, for any $\nu \in \R$.
\end{proposition}

\begin{proof}
Recall that
\begin{equation*}
    F(t,\eta) = \widehat w(\eta) \int \Tr \left( g^-_\eta(\xi) (\xi\wedge \eta)\cdot A \boldsymbol\alpha\right) \frac{\sin(t|\xi|)}{|\xi|} \frac{\sin(t|\eta+\xi|)}{|\eta+\xi|}\,d\xi
\end{equation*}
such that
\[
\widehat{L_2(V)} (t,\eta) = \int_0^t F(t-\tau,\eta) \widehat V(\tau, \eta)d\tau\,.
\]
To get the result, we prove that~\(F \in L^1(\R,L^\infty(\R^d))\). Indeed,
\begin{align*}
    \|  \widehat{L_2(V)}\left(t,\eta\right)\|_{L^2_tL^2_\eta}&\leq \| \raisebox{0.5ex}{$\chi$}_{t\geq 0}\left(\cdot\right)F\left(\cdot,\eta\right)\|_{L^1_tL^\infty_\eta}\|  \widehat{V}\|_{L^2_tL^2_\eta}\,,\\
    \|\widehat{L^\infty_2(V)}\left(t,\eta\right)\|_{L^2_tL^2_\eta}&\leq \|F\left(\cdot,\eta\right)\|_{L^1_tL^\infty_\eta}\|  \widehat{V}\|_{L^2_tL^2_\eta}\,,
\end{align*}  
and, for any $\nu \in \R$,
\begin{align*}
  \|  L_2(V)\|_{L^2_t \dot H^\nu_x}   & = \| | \eta |^\nu \widehat{L_2(V)}\left(t,\eta\right)\|_{L^2_tL^2_\eta} = \| \int_0^t F(t-\tau,\eta) | \eta|^\nu\widehat V(\tau, \eta)d\tau\|_{L^2_tL^2_\eta} \\
 & \le \| \raisebox{0.5ex}{$\chi$}_{t\geq 0}\left(\cdot\right)F\left(\cdot,\eta\right)\|_{L^1_tL^\infty_\eta}\|  V\|_{L^2_t \dot H^\nu_x}\,.
\end{align*}
Standard trigonometric identities allow us to decompose $F$ as
\begin{equation*}
    F(t, \eta) = F_+ (t, \eta) + F_- (t, \eta)\,,
\end{equation*}
where
\begin{equation*}
    F_\pm (t, \eta) = \mp \frac {\widehat w(\eta)}2 \int \Tr \left( g^-_\eta(\xi) (\xi\wedge \eta)\cdot A\boldsymbol\alpha \right) \frac{\cos \big ( t ( |\eta + \xi| \pm |\xi|) \big )}{|\xi||\eta+\xi|}d\xi\,.
\end{equation*}
Then, noticing that $\xi \wedge \eta =0$ if $\xi$ and $\eta$ are collinear, by~\Cref{stat:h-cos-decomposition} we get 
\begin{equation*}
   \|\raisebox{0.5ex}{$\chi$}_{t\geq 0}\left(\cdot\right)F\left(\cdot,\eta\right)\|_{L^1_tL^\infty_\eta} \lesssim \| \widehat w(\eta) |\eta|^{-2}  \Tr \left( g^-_\eta(\xi)(\xi\wedge \eta)\cdot A \boldsymbol\alpha\right)\|_{6,2}\,.
\end{equation*}
Focusing on the $6,2$--norm on the right-hand side, we compute
\[
\partial_\xi^k \left(\widehat w(\eta) |\eta|^{-2}  \Tr \left( g^-_\eta(\xi) (\xi\wedge \eta)\cdot A \boldsymbol\alpha \right) \right)= \widehat w(\eta) |\eta|^{-1} \Tr \left( \partial_\xi^k (g_\eta^- \xi)\wedge \widehat \eta \cdot A \boldsymbol\alpha\right)\,,
\]
where $\widehat\eta=\eta/|\eta|$ denotes the unit vector, and
\[
\partial_\xi^k (g_\eta^- \xi) = (\partial_\xi^k g_\eta^-) \xi + k(\partial_\xi^{k-1} g_\eta^-)\,.
\]
Now, if~\(|\eta| \geq 1\), we get that
\[
(\an{\eta} + \an{\xi})^{6} \widehat w(\eta) |\eta|^{-1} \Tr \left( \partial_\xi^k (g_\eta^- \xi)\wedge \widehat \eta \cdot A \boldsymbol\alpha\right) \lesssim (\an{\eta} + \an{\xi})^{6} \widehat w(\eta)\min(\an{\xi},\an{\xi+\eta})(|\partial^k_\xi g_\eta^-| + |\partial^{k-1}_\xi g_\eta^-|)
\]
and, using that~\(\an{\eta} + \an{\xi} \lesssim \an{\eta} + \an{\xi +\eta}\), we have
\begin{align*}
(\an{\eta} + \an{\xi})^{6} \widehat w(\eta) |\eta|^{-1} \Tr \left( \partial_\xi^k (g_\eta^- \xi)\wedge \widehat \eta \cdot A \boldsymbol\alpha\right) &\lesssim (\an{\eta} + \an{\xi})^{6} \an{\xi} |\widehat w(\eta)|(|\partial^k_\xi g (\xi)| + |\partial^{k-1}_\xi g (\xi)|) \\
&\lesssim \|\an{\eta}^6 \widehat w\|_{L^\infty} \|\an{\xi}^7 (|\partial^k_\xi g|+|\partial^{k-1}_\xi g|)\|_{L^\infty}\,.
\end{align*}
Secondly, if~\(|\eta|< 1\) and~\(|\xi| < 2\), we have
\begin{align*}
(\an{\eta} + \an{\xi})^{6} \widehat w(\eta) |\eta|^{-1} \Tr \left( \partial_\xi^k (g_\eta^- \xi)\wedge \widehat \eta \cdot A \boldsymbol\alpha\right) &\lesssim |\eta|^{-1}\|\widehat w\|_{L^\infty}(|\partial^k_\xi g_\eta^- (\xi)| + |\partial^{k-1}_\xi g_\eta^- (\xi)|)\\
&\lesssim \| \widehat w\|_{L^\infty} \| |\partial^{k+1}_\xi g|+|\partial^{k}_\xi g|\|_{L^\infty}\,.
\end{align*}
Lastly, if~\(|\eta| < 1\) and~\(|\xi| \geq 2\), we have~\(|\xi| \geq 2 |\eta|\) which implies, for all~\(z\in [\xi,\xi+\eta]\),~\(\an z \geq \frac12 \an{\xi}\). Thus, we deduce
\begin{align*}
(\an{\eta} + \an{\xi})^{6} \widehat w(\eta) |\eta|^{-1} \Tr \left( \partial_\xi^k (g_\eta^- \xi)\wedge \widehat \eta \cdot A \boldsymbol\alpha\right)&\lesssim \an{\xi}^7|\eta|^{-1}\|\widehat w\|_{L^\infty}(|\partial^k_\xi g_\eta^- (\xi)| + |\partial^{k-1}_\xi g_\eta^- (\xi)|) \\
&\lesssim \| \widehat w\|_{L^\infty} \an{\xi}^7\sup_{[\xi,\xi+\eta]}(|\partial^{k+1}_\xi g|+|\partial^{k}_\xi g|) \\
&\lesssim  \| \widehat w\|_{L^\infty} \|\an{\xi}^7 (|\partial^{k+1}_\xi g|+|\partial^{k}_\xi g|)\|_{L^\infty}\,.
\end{align*}
Finally, all the previous estimates yield
\begin{equation*}
   \|F\left(\cdot,\eta\right)\|_{L^1_tL^\infty_\eta} \lesssim \| \langle \eta \rangle^{6} \widehat w (\eta) \|_{\infty} \| \langle \xi \rangle^{7} \sum_{|k| \le 3}\partial^k_{\xi} g(\xi) \|_{\infty}\,,
\end{equation*}
which concludes the proof.
\end{proof}
\begin{corollary}\label{coro:boundedL2} If additionally~\(\an{\,\cdot\,}^8\widehat w\in L^\infty(\R^3)\), then both~\(L_2\) and~\(L^\infty_2\) are bounded from~\(L^2_t\dot H^{\nu}\) to~\(L^2_t\dot H^{\nu+2}\) for any $\nu\in\R$.
\end{corollary}
\begin{proof}
The result follows immediately from an application of~\Cref{prop:L2-continuity-of-L2} to the function $w_0$ defined by $\widehat{w}_0(\eta) = |\eta|^2\widehat{w}(\eta)$, while keeping the function $g$ unchanged.
\end{proof}

\begin{proposition}
\label{prop:est_L1}
    Let $g$ be as in~\Cref{rmk:L_2}. Assume that 
    \begin{itemize}
        \item $\an{\,\cdot\,}^{4} \widehat  w\in L^\infty (\R^3)\,$;
        \item ~\(g : \R^3 \rightarrow \mathcal M_4(\C)\) is an even function of class $\mathrm{C}^2$ and such that
        \[
        \an{\,\cdot\,}^{8} \sum_{|k| \le 2} \partial^k_{\xi} g  \in L^\infty(\R^3)\,.
        \]
    \end{itemize}
    Then, the linear operator
    \begin{equation}
        L_1(V) = -i \int_0^t S(t-\tau)  [ w \ast V ]Y d\tau
    \end{equation}
    is continuous from $L^2_t\dot{H}^{s-1}_x$ into $L^2_\omega L^p_t L^q_x$, for any $p>2$ and $q \ge 2$ such that $\frac 1p + \frac 1q = \frac 12$ and $s = \frac 32 - \frac 3q - \frac 1p$,
    and its operator norm is bounded by
    \begin{equation*}
        \|\an{\,\cdot\,}^4\widehat w\|_{L^\infty} \|\an{\,\cdot\,}^8\sum_{k=0}^2 \partial_\xi^k g\|_{L^\infty}\,.
    \end{equation*}
\end{proposition}

\begin{proof}
   With the notation of~\Cref{prop:TTstar}, it is sufficient to prove that~\(L^\infty_1\) is continuous from $L^2_t\dot{H}^{s-1}_x$ to $L^p_t L^q_x L^2_\omega$. Recall that 
    \[
    L^\infty_1 (V) = S(t) \underbrace{\Theta^*(w*V)}_{\eqqcolon\mathcal{L}_1^\infty}\,.
    \]
    By~\Cref{lem:TT*}, $\mathcal L_1^\infty$ is bounded from $L^2_t \dot{H}^{s-1}_x$ into $L^2_\omega \dot{H}^s_x$. Indeed, $\mathcal L_1^\infty = T^*$, where $T$ is given by \eqref{eq:T}. Then, the continuity of $L_1^\infty$ follows from an easy application of Strichartz estimates for the Dirac operator, which can be found in Lemma~\ref{lemma:Str_est}. Indeed,
    \begin{equation*}
        \|L_1^\infty\left(V\right)\|_{L^2_\omega L^p_tL^q_x}=\|S\left(t\right) \mathcal L_1^\infty\left( V\right)\|_{L^2_\omega L^p_tL^q_x}\leq C_{p,q}\|\mathcal L_1^\infty \left(V\right)\|_{L^2_\omega \dot{H}^s_x } \le C \| V \|_{L^2_t \dot H^{s-1}_x}\,.
    \end{equation*}  
\end{proof}

\begin{lemma}
\label{lem:TT*}
   Let $g$ be as in~\Cref{rmk:L_2}. Assume that 
    \begin{itemize}
        \item $\an{\,\cdot\,}^{4} \widehat  w \in L^\infty (\R^3)\,$;
        \item ~\(g : \R^3 \rightarrow \mathcal M_4(\C)\) is an even function of class $\mathrm{C}^2$ and such that
        \begin{equation}\label{eq:hp-8-power-g-decay}
            \an{\,\cdot\,}^{8} \sum_{|k| \le 2} \partial^k_{\xi} g  \in L^\infty(\R^3)\,.
        \end{equation}
    \end{itemize}
Then, for any $\nu \in \R$, the linear operator
    \begin{equation}
    \label{eq:T}
        T\coloneqq \overline w\ast_x \Theta\,,
    \end{equation}
    with~$\Theta$ as in \Cref{prop:TTstar} and \(\overline w(x) = \overline{w(-x)}\),
    is bounded from $L^2_\omega \dot H^{\nu-1}_x$ into $L^2_t \dot H^\nu_x $. Furthermore, the linear operator $T^*$ is bounded from $L^2_t \dot H^\nu_x$ into $L^2 _\omega \dot H^{\nu+1}_x$, for any $\nu\in\R$.
\end{lemma}

\begin{proof}
We immediately remark that
\begin{equation*}
    T^*=\Theta^*\left(w\ast_x\cdot\right)\,,
\end{equation*}
so that, by the $TT^*$ argument, for any $\nu \in \R$, the boundedness of $\mathcal{L}_1^\infty$ from $L^2_t \dot{H}^\nu_x$ into $L^2_\omega \dot{H}^{\nu+1}_x$ and of $T$ from $L^2_\omega \dot H^{\nu-1}_x$ into $L^2_t \dot H^\nu_x $ are equivalent to the continuity of $TT^*$ from $L^2_t\dot{H}^\nu_x$ into $L^2_t\dot{H}^{\nu+2}_x$. According to~\Cref{prop:TTstar} (Assertion~\eqref{ass:ThetaThetastar}), we have
    \begin{align*}
        \widehat{\left(TT^*\right)(V)}(t,\eta) &= |\widehat w (\eta)|^2\int_\R \widehat V(\tau,\eta) M(t-\tau,\eta)\, d\tau\,. 
    \end{align*}
Recalling that
\begin{equation*}
    L^\infty_2\left(V\right) \coloneqq 2\Im\left(\left(\Theta\Theta^*\right)\left(w \ast_x V\right)\right)\left(t,x\right)=2\Re\left(\E \big ( \langle Y(t), -i\int_{\R} S(t-\tau) [w \ast_x V\left(\tau\right) Y\left(\tau\right)] \,d\tau \rangle_{\C^4} \big )\right)\,,
\end{equation*}
we have
\[
\widehat{L^\infty_2(V)} = \int_\R F(t-\tau,\eta) \widehat V(\tau,\eta)\, d\tau\,, 
\]
where $F$ is defined in \eqref{eq:L2-F-kernel}. Thus, the continuity of~\(w\ast L^\infty_2\) from $L^2_t\dot{H}^\nu_x$ into $L^2_t\dot{H}^{\nu+2}_x$ follows from~\Cref{coro:boundedL2} as long as $\langle\,\cdot\,\rangle^8 |\widehat{w}|^2\in L^\infty$. \par

Denoting the decomposition of the operator $TT^*$ as follows,
\begin{equation}
    TT^*(V)= w\ast_x\left(\underbrace{\Re \Big (  \E \big ( \langle Y(t), \int_{\R} S(t-\tau) [w \ast_x V] Y\, d\tau \rangle_{\C^4} \big ) \Big )}_{\eqqcolon R(V)} + \frac i2 L_2^\infty(V)\right)\,,
\end{equation}
it remains to prove the boundedness of $w\ast_x R$ of $\Theta\Theta^*$ from $L^2_t\dot{H}^\nu_x$ into $L^2_t\dot{H}^{\nu+2}_x$.

Remarking that~\(R(V)\) is the real part of~\(\Theta \Theta^*(U)\), with~\(U = w*V\) we get
\[
2\widehat{R(V)}(t,\eta) = \widehat{\Theta \Theta^*(U)}(t,\eta) + \overline{\widehat{\Theta \Theta^*(U)}(t,-\eta)} .
\]
Using~\Cref{prop:TTstar} (Assertion~\eqref{ass:ThetaThetastar}) and the fact that~\(U\) is real-valued, we get
\[
2\widehat{R(V)}(t,\eta) = \int_\R \hat U(\tau, \eta) (M(t-\tau,\eta) + \overline{M(t-\tau,-\eta)})\,d\tau\,,
\]
where we recall
\[
M(t,\eta) = \int_{\R^3} \Tr_{\C^4} \big[ g(\xi) e^{it\boldsymbol{\alpha} \cdot \xi}e^{-it\boldsymbol{\alpha} \cdot (\xi +\eta)}\big] \,d\xi\,.
\]
Now,
\[
M(t,\eta) + \overline{M(t,-\eta)} = \int_{\R^3} \Tr_{\C^4} \big[ g(\xi) e^{it\boldsymbol{\alpha} \cdot \xi}e^{-it\boldsymbol{\alpha} \cdot (\xi +\eta)}\big] \,d\xi+ \int_{\R^3} \Tr_{\C^4} \big[ e^{it\boldsymbol{\alpha}\cdot(\xi -\eta)}e^{-it\boldsymbol{\alpha} \cdot \xi}  g(\xi)^* \big]\,d\xi\,.
\]
Using that~\(g(\xi) = g(\xi)^*\) and applying the change of variables~\(\zeta = \xi-\eta\), we get
\[
M(t,\eta) + \overline{M(t,-\eta)} = \int_{\R^3} \Tr_{\C^4} \big[ g(\xi) e^{it\boldsymbol{\alpha} \cdot \xi}e^{-it\boldsymbol{\alpha} \cdot (\xi +\eta)}\big] \,d\xi+ \int_{\R^3} \Tr_{\C^4} \big[ e^{it\boldsymbol{\alpha}\cdot\zeta}e^{-it\boldsymbol{\alpha} \cdot (\zeta + \eta)}  g(\zeta + \eta)\big]\,d\zeta\,,
\]
which yields, by the cyclicity of the trace,
\[
M(t,\eta) + \overline{M(t,-\eta)} = \int_{\R^3} \Tr_{\C^4} \big[ (g(\xi) + g(\xi+\eta)) e^{it\boldsymbol{\alpha} \cdot \xi}e^{-it\boldsymbol{\alpha} \cdot (\xi +\eta)}\big] \,d\xi\,.
\]

Moreover, recalling that $AB =\frac 12 ( [A,B] + \{ A,B\})$ and that $g(\xi)= f(\xi) f(\xi)^*$ is even, and observing that the function
 \[
    (\xi,\zeta) \mapsto  \Tr \Big (  \Big( g(\xi) - g(\zeta) \Big ) [e^{i(t-\tau) \boldsymbol\alpha \cdot \xi}, e^{ i (t-\tau) \boldsymbol\alpha \cdot \zeta}]\Big)
    \]
    is skew-symmetric (with $g=ff^*$), we have 
\begin{equation}\label{eq:FT-R-aux}
    \widehat{R(V)} (t,\eta) = \frac {i}{4}\widehat w(\eta) \iint \widehat V(\tau ,\eta) \Tr \Big ( \big ( g(\xi)+ g(\xi+\eta) \big ) \{ e^{i(t-\tau) \boldsymbol\alpha \cdot \xi}, e^{-i(t-\tau) \boldsymbol\alpha \cdot (\eta + \xi)}\} \Big )\, d \xi\, d\tau\,.
\end{equation}
By~\Cref{lemma:exp}, the following decomposition holds:
\begin{align*}
    \{e^{i(t-\tau) \boldsymbol\alpha \cdot \xi}, e^{-i(t-\tau) \boldsymbol\alpha \cdot (\eta + \xi)}\}=&2\cos\left(\left(t-\tau\right)|\xi|\right)\cos\left(\left(t-\tau\right)|\xi+\eta|\right)\\
    &-2i\cos\left(\left(t-\tau\right)|\xi|\right)\sin\left(\left(t-\tau\right)|\xi+\eta|\right)\frac{\boldsymbol\alpha\cdot\left(\xi+\eta\right)}{|\xi+\eta|}\\
    &+2i\cos\left(\left(t-\tau\right)|\xi+\eta|\right)\sin\left(\left(t-\tau\right)|\xi|\right)\frac{\boldsymbol\alpha\cdot\xi}{|\xi|}\\
    &+2\sin\left(\left(t-\tau\right)|\xi|\right)\sin\left(\left(t-\tau\right)|\xi+\eta|\right)\frac{|\xi|^2+\xi\cdot\eta}{|\xi||\xi+\eta|}\eqqcolon \iota_1+\iota_2+\iota_3+\iota_4\,,
\end{align*}
which allows us to rewrite \eqref{eq:FT-R-aux} as
\begin{equation*}
    \widehat{\bar w \ast R(V)} (t,\eta) =\int_\mathbb{R} G\left(t-\tau,\eta\right)\widehat{V}\left(\tau,\eta\right)\,d\tau\,,
\end{equation*}
where
\begin{equation*}
    G\left(t,\eta\right)=\frac{i}{4}|\widehat{w}\left(\eta\right)|^2\int \Tr\left(g^+_\eta\left(\xi\right)\left(\sum_{j=1}^4\iota_j\left(t,\xi,\eta\right)\right)\right)d\xi\eqqcolon \sum_{j=1}^4G_j\left(t,\eta\right)\,,
\end{equation*}
with $g^+_\eta\left(\xi\right)=f\left(\xi\right)f^*\left(\xi\right)+f\left(\xi+\eta\right)f^*\left(\xi+\eta\right)$. Now, set
\begin{align}\label{eq:G-1-pm}
    G_{1,\pm}\left(t,\eta\right)&=\frac{i}{4}|\widehat{w}\left(\eta\right)|^2\int \Tr\left(g^+_\eta\left(\xi\right)\right)\cos\left(t\left(|\xi|\pm|\xi+\eta|\right)\right)\,d\xi\,,\\ \label{eq:G-2-pm}
    G_{2,\pm}\left(t,\eta\right)&= \frac{1}{4}|\widehat{w}\left(\eta\right)|^2\int \Tr\left(g^+_\eta\left(\xi\right)\frac{\boldsymbol\alpha\cdot\left(\xi+\eta\right)}{|\xi+\eta|}\right)\sin\left( t\left(|\xi|\pm|\xi+\eta|\right)\right)\,d\xi\,,\\ \label{eq:G-3-pm}
    G_{3,\pm}\left(t,\eta\right)&=\mp\frac{1}{4}|\widehat{w}\left(\eta\right)|^2\int \Tr\left(g^+_\eta\left(\xi\right)\frac{\boldsymbol\alpha\cdot\xi}{|\xi|}\right)\sin\left(t\left(|\xi|\pm|\xi+\eta|\right)\right)\,d\xi\,,\\ \label{eq:G-4-pm}
    G_{4,\pm}\left(t,\eta\right)&=\mp\frac{i}{4}|\widehat{w}\left(\eta\right)|^2\int \Tr\left(g^+_\eta\left(\xi\right)\right)\frac{|\xi|^2+\xi\cdot\eta}{|\xi||\xi+\eta|} \cos\left(t \left(|\xi|\pm|\xi+\eta|\right)\right)\,d\xi\,,
\end{align}
so that, by means of standard trigonometric identities,
\begin{equation*}
    G_j\left(t,\eta\right)=G_{j,+}\left(t,\eta\right)+G_{j,-}\left(t,\eta\right),\quad j=1,2,3,4\,.
\end{equation*}
Therefore, the continuity of $w\ast_xR$ from $L^2_t\dot{H}^\nu$ into $L^2_t\dot{H}^{\nu+2}_x$ reduces to the study of the same continuity for $\Gamma_{j,\sigma}$, defined by
\begin{equation*}
    \widehat{\Gamma_{j,\sigma}\left(V\right)}\left(t,\eta\right)=\int_\mathbb{R} G_{j,\sigma}\left(t-\tau,\eta\right)\widehat{V}\left(\tau,\eta\right)\,d\tau\,,\quad j=1,2,3,4\,,\quad \sigma\in\{+,-\}\,.
\end{equation*}
Now,
\begin{equation}\label{eq:aux0}
    \|\Gamma_{j,\sigma}\left(V\right)\|_{L^2_t\dot{H}^{\nu+2}_\eta}=\|\,|\eta|^{\nu+2}\,\widehat{\Gamma_{j,\sigma}\left(V\right)}\|_{L^2_tL^2_\eta}=\|\left(|\eta|^2\, G_{j,\sigma}\right)\ast_t\left(|\eta|^\nu\,\widehat{V}\right)\|_{L^2_tL^2_\eta}\,.
\end{equation}
For $j=1,4$ and $\sigma\in\{+,-\}$, one applies Lemma~\ref{stat:h-cos-decomposition} and gets
\begin{align}\label{eq:aux1}
    \|\Gamma_{j,\sigma}\left(V\right)\|_{L^2_t\dot{H}^{\nu+2}_\eta}&\leq \|\left(|\eta|^2\,\frac{\sin\left(|\eta|t\right)}{|\eta|t}\overline{G}_{j,\sigma}\left(h\right)\left(\eta\right)\right)\ast_t\left(|\eta|^\nu\,\widehat{V}\right)\|_{L^2_tL^2_\eta}\\ \label{eq:aux2}
    &\quad+\|\left(\widetilde{G}_{j,\sigma}\left(h\right)\left(t,\eta\right)\right)\ast_t\left(|\eta|^\nu\,\widehat{V}\right)\|_{L^2_tL^2_\eta}\,.
\end{align}
By Young's convolution inequality, the norm in \eqref{eq:aux2} can be easily estimated as follows:
\begin{align*}
    \|\left(\widetilde{G}_{j,\sigma}\left(h\right)\left(t,\eta\right)\right)\ast_t\left(|\eta|^\nu\,\widehat{V}\right)\|_{L^2_tL^2_\eta}&\leq \|\widetilde{G}_{j,\sigma}\left(h\right)\left(t,\eta\right)\|_{L^1_tL^\infty_\eta}\|\,|\eta|^\nu\,\widehat{V}\|_{L^2_tL^2_\eta}\\
    &=\|\widetilde{G}_{j,\sigma}\left(h\right)\left(t,\eta\right)\|_{L^1_tL^\infty_\eta}\|V\|_{L^2_t\dot{H}^\nu_x}\,.
\end{align*}
Focusing on $j=1$ and $\sigma\in\{+,-\}$, the continuity of $\widetilde{G}_{j,\sigma}$ stated in~\Cref{stat:h-cos-decomposition} implies
\begin{equation*}
    \|\left(\widetilde{G}_{j,\sigma}\left(h\right)\left(t,\eta\right)\right)\ast_t\left(|\eta|^\nu\,\widehat{V}\right)\|_{L^2_tL^2_\eta}\leq \||\widehat{w}\left(\eta\right)|^2h\|_{8,2}\|V\|_{L^2_t\dot{H}^\nu_x}\,,
\end{equation*}
where $h\left(\xi,\eta\right)=\Tr\left(g^+_\eta\left(\xi\right)\right)$ by \eqref{eq:G-1-pm}, and the $8,2$--norm is bounded as long as $\langle\,\cdot\,\rangle^8 |\widehat{w}|^2\in L^\infty$ and
\begin{equation}\label{eq:8-power-decay-f}
    \langle\,\cdot\,\rangle^8\sum_{|k| \le 2} \partial^k (ff^*)  \in L^\infty(\R^3)\,,
\end{equation}
which is equivalent to \eqref{eq:hp-8-power-g-decay}. Instead, for $j=4$ and $\sigma\in\{+,-\}$, once again thanks to~\Cref{stat:h-cos-decomposition}, one has
\begin{equation*}
    \|\left(\widetilde{G}_{j,\sigma}\left(h\right)\left(t,\eta\right)\right)\ast_t\left(|\eta|^\nu\,\widehat{V}\right)\|_{L^2_tL^2_\eta}\leq \||\widehat{w}\left(\eta\right)|^2h\,|\xi|\,|\xi+\eta|\,\|_{6,2}\|V\|_{L^2_t\dot{H}^\nu_x}\,,
\end{equation*}
where $h\left(\xi,\eta\right)=\Tr\left(g^+_\eta\left(\xi\right)\right)\frac{|\xi|^2+\xi\cdot\eta}{|\xi||\xi+\eta|}$ by \eqref{eq:G-4-pm}, and the $6,2$--norm is bounded as long as $\langle\,\cdot\,\rangle^7 |\widehat{w}|^2\in L^\infty$ and $f$ satisfies \eqref{eq:8-power-decay-f}. Going back to \eqref{eq:aux1}, by noticing that
\begin{equation*}
    \frac{\sin\left(|\eta|t\right)}{|\eta|t}=\mathcal{F}_\omega\left(\frac{1}{2|\eta|}\raisebox{0.5ex}{$\chi$}_{[-|\eta|,|\eta|]}\left(\omega\right)\right)\left(t\right)\,,
\end{equation*}
the following estimate for the $L^2_tL^2_\eta$--norm holds
\begin{align*}
    &\|\left(|\eta|^2\,\frac{\sin\left(|\eta|t\right)}{|\eta|t}\overline{G}_{j,\sigma}\left(h\right)\left(\eta\right)\right)\ast_t\left(|\eta|^\nu\,\widehat{V}\right)\|_{L^2_tL^2_\eta}\\
    &\hspace{3cm}=\|\,|\eta|^2\,\mathcal{F}_t\left(\frac{\sin\left(|\eta|t\right)}{|\eta|t}\right)\left(\omega\right)\overline{G}_{j,\sigma}\left(h\right)\left(\eta\right)\,|\eta|^\nu\,\mathcal{F}_t\left(\widehat{V}\right)\left(\omega\right)\|_{L^2_\omega L^2_\eta}\\
    &\hspace{3cm}\leq\|\,|\eta|^{\nu+1}\,\raisebox{0.5ex}{$\chi$}_{[-|\eta|,|\eta|]}\left(\omega\right)\overline{G}_{j,\sigma}\left(h\right)\left(\eta\right)\widehat{V}\left(\omega,\eta\right)\|_{L^2_\omega L^2_\eta}\\
    &\hspace{3cm}\lesssim \|\overline{G}_{j,\sigma}\left(\,|\eta|\,h\right)\left(\eta\right)\|_{L^\infty_\eta}\|\,|\eta|^\nu\,\widehat{V}\left(\omega,\eta\right)\|_{L^2_\omega L^2_\eta}\,.
\end{align*}
Focusing on $j=1$ and $\sigma\in\{+,-\}$, the continuity of $\overline{G}_{j,\sigma}$ stated in~\Cref{stat:h-cos-decomposition} yields
\begin{equation*}
    \|\left(|\eta|^2\,\frac{\sin\left(|\eta|t\right)}{|\eta|t}\overline{G}_{j,\sigma}\left(h\right)\left(\eta\right)\right)\ast_t\left(|\eta|^\nu\,\widehat{V}\right)\|_{L^2_tL^2_\eta}\lesssim \|\,|\widehat{w}\left(\eta\right)|^2h\|_{8,2}\|V\|_{L^2_t \dot{H}^\nu_x}\,,
\end{equation*}
where $h\left(\xi,\eta\right)=\Tr\left(g^+_\eta\left(\xi\right)\right)$ by \eqref{eq:G-1-pm}, and the $8,2$--norm is bounded as long as $\langle\,\cdot\,\rangle^4 |\widehat{w}|^2\in L^\infty$ and $f$ satisfies \eqref{eq:8-power-decay-f}. Instead, for $j=4$ and $\sigma\in\{+,-\}$, once again thanks to~\Cref{stat:h-cos-decomposition}, one has
\begin{equation*}
    \|\left(|\eta|^2\,\frac{\sin\left(|\eta|t\right)}{|\eta|t}\overline{G}_{j,\sigma}\left(h\right)\left(\eta\right)\right)\ast_t\left(|\eta|^\nu\,\widehat{V}\right)\|_{L^2_tL^2_\eta}\lesssim \|\,|\widehat{w}\left(\eta\right)|^2h\,|\xi|\,|\xi+\eta|\,\|_{6,2}\|V\|_{L^2_t\dot{H}^\nu_x}\,,
\end{equation*}
where $h\left(\xi,\eta\right)=\Tr\left(g^+_\eta\left(\xi\right)\right)\frac{|\xi|^2+\xi\cdot\eta}{|\xi||\xi+\eta|}$ by \eqref{eq:G-4-pm}, and the $6,2$--norm is bounded as long as $\langle\,\cdot\,\rangle^{8} |\widehat{w}|^2\in L^\infty$ and $f$ satisfies \eqref{eq:8-power-decay-f}.\par
Going back to \eqref{eq:aux0}, for $j=2,3$ and $\sigma\in\{+,-\}$, one applies~\Cref{stat:h-sin-decomposition} and gets
\begin{align}\label{eq:aux3}
    \|\Gamma_{j,\sigma}\left(V\right)\|_{L^2_t\dot{H}^{\nu+2}_\eta}&\leq \|\left(|\eta|^2\,\left(\frac{\sin(|\eta|t) - t |\eta| \cos(|\eta|t)}{|\eta|t^2} \right)\overline{G}_{j,\sigma}\left(h\right)\left(\eta\right)\right)\ast_t\left(|\eta|^\nu\,\widehat{V}\right)\|_{L^2_tL^2_\eta}\\ \label{eq:aux4}
    &\quad+\|\left(\widetilde{G}_{j,\sigma}\left(h\right)\left(t,\eta\right)\right)\ast\left(|\eta|^\nu\,\widehat{V}\right)\|_{L^2_tL^2_\eta}\,.
\end{align}
The norm in \eqref{eq:aux4} can be estimated similarly to the one in \eqref{eq:aux2}, recalling that $h\left(\xi,\eta\right)=\Tr\left(g^+_\eta\left(\xi\right)\frac{\boldsymbol\alpha\cdot\left(\xi+\eta\right)}{|\xi+\eta|}\right)$ for $j=2$ and $h\left(\xi,\eta\right)=\Tr\left(g^+_\eta\left(\xi\right)\frac{\boldsymbol\alpha\cdot\xi}{|\xi|}\right)$ for $j=3$. In both cases, the $7,2$--norm is bounded as long as $\langle\,\cdot\,\rangle^8 |\widehat{w}|^2\in L^\infty$ and $f$ satisfies \eqref{eq:8-power-decay-f}. 
The norm in \eqref{eq:aux3} can be estimated similarly to the one in \eqref{eq:aux1} by noticing that
\begin{equation*}
    \frac{\sin(|\eta|t) - t |\eta| \cos(|\eta|t)}{|\eta|t^2}=\mathcal{F}_\omega\left(\frac{i\omega}{2|\eta|}\raisebox{0.5ex}{$\chi$}_{[-|\eta|,|\eta|]}\left(\omega\right)\right)\left(t\right)\,,
\end{equation*}
and recalling that $h\left(\xi,\eta\right)=\Tr\left(g^+_\eta\left(\xi\right)\frac{\boldsymbol\alpha\cdot\left(\xi+\eta\right)}{|\xi+\eta|}\right)$ for $j=2$ and $h\left(\xi,\eta\right)=\Tr\left(g^+_\eta\left(\xi\right)\frac{\boldsymbol\alpha\cdot\xi}{|\xi|}\right)$ for $j=3$. In both cases, the $7,2$--norm is bounded as long as $\langle\,\cdot\,\rangle^{8} |\widehat{w}|^2\in L^\infty$ and $f$ satisfies \eqref{eq:8-power-decay-f}. 
\end{proof}

For any~\(k\in \Z\), let~\(W_k\) be the Fourier multiplier by~\(\xi\mapsto\an{\xi}^k\).

\begin{corollary}\label{cor:estimateswithSobolev} Assume that~\(g, p,q\) and \(s\) satisfy the hypothesis of~\Cref{prop:est_L1}. 

\begin{enumerate}
\item For any~\(\nu \in \R\), the operator~\(\Theta^* W_{-4}\) is bounded from~\(L^2_t\dot H^\nu_x\) into~\(L^2_\omega\dot H^{\nu +1}_x\) and the operator~\(W_{-4}\Theta\) is bounded from~\(L^2_\omega\dot H^{\nu -1}_x\) into~\(L^2_t \dot H^\nu_x\) ;
\item if~\(w\) is such that~\(\hat w\in L^\infty\), we have that~\(L_1 W_{-4}\) is bounded from~\(L^2_t \dot H^{s-1}_x\) into~\(L^2_\omega L^p_t L^q_x\);
\item if~\(w\) is such that~\(\hat w\in L^\infty\), we have that~\(L_1\) is bounded from~\(L^2_t (\dot H_x^{s-1} \cap H_x^4)\) into~\(L^2_\omega L^p_t L^q_x\);
\item the operator~\(L_3\) is bounded from~\(L^2_\omega \dot H_x^{\nu - 1}\) into~\(L^2_t \dot H^\nu_{-4,x}\), where we recall that~\(\dot H^\nu_{-4,x}\) stands for~\((1-\Delta)^2\dot H_x^\nu\) (see~\Cref{sec:notation}).
\end{enumerate}
\end{corollary}

\begin{proof}
The operator~\(\Theta^* W^{-4}\) coincides with the operator~\(T^*\) defined in \Cref{lem:TT*}, with~\(w\) given by the inverse Fourier transform of~\(\an{\xi}^{-4}\). Therefore,~\(w\) satisfies the hypothesis of Lemma~\ref{lem:TT*} which  implies (1).

By the Christ-Kiselev lemma, it suffices to prove~(2) for~\(L_1^\infty W_{-4} = S(t)\Theta^* W_{-4} (w\ast_x\cdot)\). By standard Strichartz estimates, the operator~\(S(t)\Theta^* W_{-4}\) is bounded from~\(L^2_t \dot H^{s-1}_x\) into~\(L^2_\omega L^p_t L^q_x\). Since~\(\hat w\) is bounded, the operator~\((w\ast_x\cdot)\) is bounded from~\(L^2_t \dot H^{s-1}_x\) into itself.

We have~\(L_1 = L_1 W_{-4} W_4\) so that the previous continuity estimate implies
\[
\|L_1(V)\|_{L^2_\omega L^p_t L^q_x} \lesssim \|W_4 V\|_{L^2_t \dot H^{s-1}_x}\,.
\]
Since~\(s\leq 1\),
\[
\|W_4 V\|_{L^2_t \dot H^{s-1}_x} \lesssim \|V\|_{L^2_t (\dot H_x^{s-1} \cap H_x^4)}\,,
\]
which yields (3).

We have that~\(L_3 = 2\Re \Theta\) and~\(\Theta = W_4 W_{-4} \Theta\). Since~\(W_4 = (1-\Delta)^2\), we easily get (4).
\end{proof}

\subsection{Strichartz Estimates}

Under the hypothesis of~\Cref{th:main}, we have that~\(1-L_2\) is invertible. Therefore, we have
\[
S_Z(Z_0) (t) =S(t) Z_0 + L_1(1-L_2)^{-1}L_3(Z_0)\quad\text{and}\quad \rho[S_Q(Q_0)(t)] = (1-L_2)^{-1} \rho[S(t) Q_0 S(-t)]. 
\]
We recall the standard (and dual) Strichartz estimates for spinors and for orthonormal systems.

\begin{lemma}[Standard Strichartz estimates \cite{DAnFan-2008-CPDE, GinVel-1995-JFA, Hos-2024-JFA}]\label{lemma:Str_est}
    Let $S(t)$ be as in \eqref{eq:mass-Dirac-propagator} and let $p,q \in [2, +\infty]$ and $s \in \R$ such that 
    \begin{equation*}
        \frac 1p + \frac 1q = \frac12\,, \quad q < + \infty\,, \quad s= \frac 32 - \frac 3q - \frac 1p\,.
    \end{equation*}
    Then, for any $u\in L^2_\omega \dot H^s_x$, there exist a constant $C=C(p,q,s) >0$ such that
    \begin{equation*}
        \| S(t) u \|_{L^2_\omega L^p_t L^q_x} \le C \|u \|_{L^2_\omega \dot H^s_x}\,,
    \end{equation*}
    and, for any $F \in L^2_\omega L^{p'}_t \dot H^{2s}_{q',x}$, a constant $C'= C'(p,q,s)>0$ such that
    \begin{equation*}
        \| \int_0^t S(t-s) F(s) ds \|_{L^2_\omega L^p_t L^q_x} \le C' \| F \|_{L^2_\omega L^{p'}_t \dot H^{2s}_{q',x}}\,,
    \end{equation*}
    where $p'$ and $ q'$ are the conjugate exponents of, respectively, $p$ and $q$.

    Furthermore, there exists a constant~\(C=C(p,q,s)>0\) such that, for any~\(Q_0\) satisfying $\an{\nabla}^{s} Q_0 \an{\nabla}^s\in\mathfrak S^\beta$,
    \[
    \| \rho[S(t) Q_0 S(-t)] \|_{L^{p/2}_t L^{q/2}_x} \leq C \|\an{\nabla}^{s} Q_0 \an{\nabla}^s\|_{\mathfrak S^\beta}\,,
    \]
    where, once again, $\rho[Q]=\rho_Q$. Here, the exponent $\beta$ is given by
    \begin{equation*}
        \beta= \begin{cases}
            \frac{2q}{q+2}\,,&\text{if }q<6\,,\\
            \text{any }\beta\in [1,\frac{p}{2}]\,,\quad&\text{otherwise}\,.
        \end{cases}
    \end{equation*}
\end{lemma}

 The proof of~\Cref{th:main} is a direct consequence of the boundedness of the operators $L_1,L_2, L_3$ and  $S(t)$ proved in~\Cref{prop:L2-continuity-of-L2} and in Lemmas~\ref{lem:boundedness}~and~\ref{lemma:Str_est}.

We recall that, for any~\(k\in \Z\), the operator~\(W_k\) is the Fourier multiplier by~\(\xi\mapsto\an{\xi}^k\).

\begin{lemma}
\label{lem:boundedness} Let~\(p,q,s,f,w\) be such that the hypothesis of~\Cref{th:main} are fulfilled. The operators $L_1, L_2$ and $L_3$ are well-defined and continuous according to the following mappings:
    \begin{align}
        &\label{eq:L1-strichartz-continuity} L_1 W_4 : L^2_t\dot H^{s-1}_x \rightarrow L^2_\omega L^p_tL^q_x \,, \\
        &\label{eq:L3-strichartz-continuity} L_3 : L^2_{\omega} \dot H^s_x \rightarrow L^p_tL^q_x\,, \\
        &\label{eq:L3-continuity} W_{-4} L_3 : L^2_{\omega}\dot H^{\nu-1}_x \rightarrow L^2_{t}\dot H^{\nu}_x\,, \quad \forall \nu \in \R\,.
    \end{align}
\end{lemma}
\begin{proof} The operator~\(L_1\) depends on~\(w\): we indicate this dependence as~\(L_1^w\).

    \eqref{eq:L1-strichartz-continuity}. We have that~\(L_1^w W_4 = L_1^{w_1}\) with~\(w_1 = \an{\nabla}^4 w\). We have that~\(\an{\xi}^4 \hat w_1(\xi) = \an{\xi}^8 \hat w (\xi) \in L^\infty\). We deduce that~\(L_1^w W_4\) is bounded thanks to~\Cref{prop:est_L1}.

    
    \eqref{eq:L3-strichartz-continuity}. Arguing as before and recalling the Strichartz estimates in~\Cref{lemma:Str_est}, we get 
    \[
    \| L_3(Z_0) \|_{L^p_tL^q_x} \lesssim  \| Y \|_{ L^\infty_{t,x}L^2_\omega} \| Z_0 \|_{L^2_\omega \dot H^s_x}\,. 
    \]
    
    \eqref{eq:L3-continuity}. By~\Cref{lem:TT*}, we have that~\(\Theta^* W_{-4}\) is bounded from~\(L^2_t\dot H^\nu_x\) to~\(L^2_\omega\dot H^{\nu-1}_x\). Therefore,~\(W_{-4} \Theta\) and thus~\(W_{-4} L_3\) is bounded from~\(L^2_\omega\dot H^{\nu-1}_x\) to~\(L^2_t\dot H^\nu_x\).
\end{proof}

\begin{proof}[Proof of~\Cref{th:main}]
We start with the proof in the random field setting. Since~\((1-L_2)^{-1}\) is a Fourier multiplier, by \Cref{stat:reduced-Z-problem} we get
\[
S_Z(Z_0)(t) =S(t) Z_0 + L_1 W_4 (1-L_2)^{-1} W_{-4} L_3(Z_0)(t)\,.
\]
On the one hand, \Cref{lem:boundedness} and \Cref{prop:L2-continuity-of-L2} yields, for any $Z_0\in L^2_\omega \dot H^{s-2}_x$,
\begin{equation*}
    Z_1\coloneqq L_1 W_4 (1-L_2)^{-1} W_{-4} L_3(Z_0)\in L^2_\omega L^p_t L^q_x\,, 
\end{equation*}
with
\begin{equation*}
    \|Z_1\|_{L^2_\omega L^p_t L^q_x}\lesssim \|Z_0\|_{\dot H^{s-2}_x}\,.
\end{equation*}
On the other hand, \Cref{lemma:Str_est} implies, for any $Z_0\in L^2_\omega \dot H^s_x$,
\begin{equation*}
    \|S(t)Z_0\|_{L^2_\omega L^p_t L^q_x}\lesssim\|Z_0\|_{L^2_\omega \dot H^s_x}\,.
\end{equation*}
This concludes the proof of the result in the random field framework.

We now turn our attention to the density matrix setting. 

Notice that, thanks to Proposition \eqref{stat:reduced-Q-problem} we have
\[
\rho [Q(t)] = \rho_1 + L_2(1-L_2)^{-1} \rho_1,
\]
where $\rho_1 := \rho_{S(t)Q_0 S(-t)}$. Then, by Proposition \eqref{prop:L2-continuity-of-L2} and Lemma \eqref{lemma:Str_est} we deduce
\[
\| (1-L_2)^{-1} \rho_1 \|_{L^2_{t,x}} \lesssim \|\an{\nabla}^{1/2} Q_0 \an{\nabla}^{1/2}\|_{\mathfrak S^{4/3}}\, ,
\]
which establishes the result in the density matrix formulation and completes the proof of \Cref{th:main}.
\end{proof}

\section{Proof of the Nonlinear Result}

Recall the notation of~\Cref{subsec:reduction} and in particular that of~\Cref{stat:defQLtilde}. We explicitly write the dependence on~\(w\) of the operators~\(L_1\) and~\(L_2\) as~\(L_1^w,L_2^w\).

Since 
\[
L_2^{w_0} :  W \mapsto 2 \mathrm{Re}\E \left( \an{Y,-i\int_{0}^t S(t-s) [(w_0 * W )Y] \,ds}\right)
\]
is a Fourier multiplier, we have that~\(\widetilde L_2\) coincides with~\(L_2^w\). Under the smallness assumption on the norm of $f$, the operator $1-\widetilde L_2$ is invertible on $L_{t,x}^2$. Therefore, we can rewrite \eqref{eq:fixed-point-CLQ} as
\begin{equation}\label{eq:fixedpointCO2}
\begin{cases}
    Z = S(t) Z_0 + \widetilde L_1 (1- \widetilde L_2)^{-1}(\widetilde L_3(Z_0) + \widetilde Q_2 (Z,V) + L_0\E[|Z|^2]) + \widetilde Q_1 (Z,V)\, ,\\
    V = (1- \widetilde L_2)^{-1}\widetilde L_3(Z_0) + (1- \widetilde L_2)^{-1}\widetilde Q_2(Z,V) + (1- \widetilde L_2)^{-1}L_0\E[|Z|^2]\,.
\end{cases}
\end{equation}

Now, the proof of Theorem~\ref{th:NL} reduces to proving that
\begin{align*}
    \|S(t)Z_0\|_{\Xi} \lesssim & \|Z_0\|_{\Xi_0}\,, \\
    \|\widetilde L_1 (1- \widetilde L_2)^{-1}\widetilde L_3(Z_0)\|_\Xi \lesssim & \|Z_0\|_{\Xi_0}\,, \\
    \|\widetilde L_1 (1- \widetilde L_2)^{-1}\widetilde Q_2 (Z,V)\|_\Xi \lesssim & \|Z\|_\Xi \|V\|_{L^2_{t,x}}\,, \\
    \|\widetilde L_1(1-\widetilde L_2)^{-1}L_0\E[\an{Z,Z'}]\|_{\Xi}\lesssim & \|Z\|_{\Xi} \|Z'\|_{\Xi}\\
    \|\widetilde Q_1 (Z,V)\|_\Xi \lesssim & \|Z\|_\Xi \|V\|_{L^2_{t,x}}\,, \\
    \|(1- \widetilde L_2)^{-1}\widetilde L_3(Z_0)\|_{L^2_{t,x}} \lesssim & \|Z_0\|_{\Xi_0}\,, \\
    \|(1- \widetilde L_2)^{-1}\widetilde Q_2(Z,V)\|_{L^2_{t,x}} \lesssim & \|Z\|_\Xi \|V\|_{L^2_{t,x}}\,, \\
    \|(1- \widetilde L_2)^{-1}L_0 \E[\an{Z,Z'}] \|_{L^2_{t,x}} \lesssim & \|Z\|_\Xi \|Z'\|_\Xi\,,
\end{align*}
where~\(\Xi = L^2_\omega(C(\R, L^2(\R^3)) \cap L^4_{t,x})\) and~\(\Xi_0 = L^2_\omega(\dot H^{-\frac12}_x \cap H^{\frac12}_x)\). These continuity estimates easily follow from the following lemmas.

\begin{lemma}\label{lem:contOp} The following operators are continuous:
\begin{align*}
\widetilde L_1 (1- \widetilde L_2)^{-1} : L^2_t\dot H^{-1/4}_x \rightarrow & \Xi\,, \\
\widetilde L_3 : L^2_\omega \dot H^{-1/2}_x\rightarrow & L^2_t \dot H^{-1/4}_x\,, \\
\widetilde L_3 : L^2_\omega \dot H^{-1/4}_x\rightarrow & L^2_{t,x}\,.
\end{align*}
\end{lemma}

\begin{proof} Notice that
\[
\widetilde L_1 = L_1^{w_0} \circ L_0\,,
\]
with
\[
L_1^{w_0}(V) = -i \int_0^t S(t-s) (w_0* V(s)) Y(s) \,ds\,.
\]
Since~\( L_1 ^{w_0}\) is bounded from~\(L^2_t(\dot H^{-1}_x\cap H^{4}_x)\) to~\(\Xi\) (by Corollary~\ref{cor:estimateswithSobolev}) and~\(L_0\) is bounded from~\(L^2_t \dot H^{-\frac14}_x\) to~\(L^2_t(\dot H^{-1}_x\cap H^{4}_x)\) (by definition of~\(L_0\)), we deduce that~\(\widetilde L_1\) is continuous from~\(L^2_t \dot H^{-\frac14}_x\) to~\(\Xi\). By \Cref{prop:L2-continuity-of-L2}, operator~\((1- \widetilde L_2)^{-1}\) is a Fourier multiplier bounded from~\(L^2_{t,x}\) to itself and hence it is also bounded from~\(L^2_t \dot H^{-\frac14}_x\) to itself. Thus,
\[
\widetilde L_1 (1-\widetilde L_2)^{-1} : L^2_t \dot H^{-\frac14}_x\rightarrow \Xi
\]
is a continuous operator.

Analogously, we also have that~\(\widetilde L_3 = L_0 \circ L_3\). Since~\(L_3\) is a bounded operator from~\(L^2_\omega \dot H^{\nu-1}_x\) to~\(L^2(\R, \dot H^{\nu}_{-4,x})\) (by Corollary~\ref{cor:estimateswithSobolev}) and~\(L_0\) is bounded from~\(\dot H^{\nu}_{-4,x}\) to~\(\dot H^{\nu-\frac34}_x\) (by definition of~\(L_0\)), we get that
\[
\widetilde L_3 : L^2_\omega \dot H^{\nu-1}_x \rightarrow L^2_t \dot H^{\nu-\frac34}_x 
\]
is a continuous operator. Taking~\(\nu = \frac12\) and~\(\nu = \frac34\), we obtain the desired results.
\end{proof}

Along with the Strichartz estimates in \Cref{lemma:Str_est} and the fact that the bilinear map, defined as
\begin{align*}
L^2_\omega L^4_{t,x} \times L^2_\omega L^4_{t,x} &\longrightarrow L^2_{t,x} \\
(Z,Z') &\longmapsto \E(\an{Z,Z'}))\,,
\end{align*}
is continuous, the previous lemma reduces the proof of~\Cref{th:NL} to proving the following lemma.

\begin{lemma}\label{lem:contbilinear} We have
\begin{align}\label{eq:Q1}
\|\widetilde Q_1 (Z,V)\|_\Xi \lesssim & \|Z\|_\Xi \|V\|_{L^2_{t,x}}\,, \\ \label{eq:Q21}
\|\widetilde Q_2 (Z,V)\|_{L^2_{t,x}} \lesssim & \|Z\|_\Xi \|V\|_{L^2_{t,x}}\,, \\ \label{eq:Q22}
\|\widetilde Q_2 (Z,V)\|_{L^2_t \dot H^{-\frac14}_x} \lesssim & \|Z\|_\Xi \|V\|_{L^2_{t,x}}\,.
\end{align}
\end{lemma}

\begin{proof} We start with~\eqref{eq:Q1}. By using the dual Strichartz estimates~in \Cref{lemma:Str_est}, we get
\[
\|\widetilde Q_1 (Z,V)\|_\Xi \lesssim \| w_0*(L_0 V) \Pi Z\|_{L^2_\omega L^{\frac43}_t W^{1,\frac43}_x}\,.
\]
Then, we use the Leibniz rule to get
\[
\|\widetilde Q_1 (Z,V)\|_\Xi \lesssim \| w_0*(L_0 V)\|_{L^2_t H^1_x} \|\Pi Z\|_{L^2_\omega L^{4}_t W^{1,4}_x}\,.
\]
Since~\(V\mapsto w_0*(L_0 V)\) is bounded from~\(L^2_x\) to~\(H^1_x\) and~\(Z\mapsto \Pi Z\) is bounded from~\(L^4_x\) to~\(W^{1,4}_x\), we obtain
\[
\|\widetilde Q_1 (Z,V)\|_\Xi \lesssim \| V\|_{L^2_{t,x}} \| Z\|_{L^2_\omega L^{4}_{t,x}}\,,
\]
which proves \eqref{eq:Q1}.

We now focus on~\eqref{eq:Q21} and~\eqref{eq:Q22}. Take~\(W \in L^2_{t,x} + L^2_t\dot H^{\frac14}_x\) to be real-valued. We have 
\[
\an{W, \widetilde Q_2(Z,V)}_{t,x} = 2 \mathrm{Re}\an{ \int_{s}^\infty S(s-t) Y(t) L_0 W(t)\, dt , w_0*(L_0 V)\Pi Z }_{s,x,\omega}\,.
\]
Similarly to~\(L_1^{w_0}\), the operator $L_1'$ defined by
\[
L_1' : W \mapsto \int_{s}^\infty S(s-t) Y(t) W(t)\, dt
\]
is continuous from~\(L^2(\R,\dot H^{-\frac12}_x \cap H^{4}_x)\) into~\(L^2_\omega L^4_{t,x}\). We deduce that~\(L_1'\circ L_0\) is continuous from~\(L^2_{t,x} + L^2_t \dot H^{\frac14}_x\) to~\(L^2_\omega L^4_{t,x}\). Thus,
\[
\an{W, \widetilde Q_2(Z,V)}_{t,x} \lesssim \|W\|_{L^2_{t,x} + L^2_t \dot H^{\frac14}_x} \|w_0*(L_0 V)\Pi Z \|_{L^2_\omega L^{\frac43}_{t,x}}\,,
\]
which allows us to conclude the proof as done for~\(\widetilde Q_1\).
\end{proof}

\begin{proof}[Proof of~\Cref{th:NL}]
Recall that under the smallness assumption on the norm of $f$, the operator $1-\widetilde L_2$ is invertible on $L_{t,x}^2$. Thus, the fixed-point equation \eqref{eq:fixed-point-CLQ} can be rewritten as \eqref{eq:fixedpointCO2}. We first solve it and then prove the scattering result.\par
Defining the mapping $\Phi_{Z_0}:\Xi \times L^2_{t,x} \longrightarrow \Xi \times L^2_{t,x}$ such that
\begin{equation*}
    \begin{pmatrix}
    Z \\ V
\end{pmatrix} \longmapsto \begin{pmatrix}
    S(t) Z_0 + \widetilde L_1 (1- \widetilde L_2)^{-1}(\widetilde L_3(Z_0) + \widetilde Q_2 (Z,V)+L_0\E[|Z|^2]) + \widetilde Q_1 (Z,V) \\ (1- \widetilde L_2)^{-1}\widetilde L_3(Z_0) + (1- \widetilde L_2)^{-1}\widetilde Q_2(Z,V) + (1- \widetilde L_2)^{-1}L_0\E(|Z|^2)
\end{pmatrix}\,,
\end{equation*}
the fixed-point equation \eqref{eq:fixedpointCO2} can be rewritten as
\begin{equation*}
    \begin{pmatrix}
    Z \\ V
\end{pmatrix}=\Phi_{Z_0}\begin{pmatrix}
    Z \\ V
\end{pmatrix}\,.
\end{equation*}
For a suitable constant $R>0$, one can easily prove that the mapping $\Phi_{Z_0}$ is a contraction on $B_{\Xi\times L^2_{t,x}}\left(0,R\|Z_0\|_{\Xi_0}\right)$ by means of Lemmas~\ref{lem:contOp}~and~\ref{lem:contbilinear}, provided that $Z_0\in\Xi_0$ has a sufficiently small norm $\|Z_0\|_{\Xi_0}$. Finally, the Banach fixed-point theorem yields the existence and uniqueness of a solution to~\eqref{eq:fixedpointCO2} in $B_{\Xi\times L^2_{t,x}}\left(0,R\|Z_0\|_{\Xi_0}\right)$ and therefore of a solution~\(X \in Y_f + L^2_{\omega}(C(\R,L^2_x(\R^3))\cap L^4_{t,x})\)  to \eqref{eq:cutoff} such that $X\big|_{t=0}=Y_f+Z_0$.\par

To prove the scattering result, we observe that, recalling \eqref{eq:Duhamel_Z}, we can write
\begin{multline*}
    Z(t) = S(t) \bigg ( Z_0 - i \int_0^\infty S(-\tau) \big [w_0 \ast L_0 V (\tau) \big ] (Y(\tau)+\Pi Z(\tau))\, d\tau\, \\
     + i \int_t^\infty S(-\tau) \big [w_0 \ast L_0 V (\tau) \big ] (Y(\tau)+\Pi Z(\tau))\, d\tau\, \bigg )\,.
\end{multline*}
By Lemma \eqref{lem:contOp} we get that $\int_0^\infty S(-\tau) \big [w_0 \ast L_0 V (\tau) \big ] Y(\tau)\, d\tau\, \in L^2_{\omega,x}$ and that 
\[
\bigg \Vert  \int_t^\infty S(-\tau) \big [w_0 \ast L_0 V (\tau) \big ] Y(\tau)\, d\tau\, \bigg \Vert_{L^2_{\omega,x}} \to 0 \,,
\]
as $t \to + \infty$. Moreover, Lemma \eqref{lem:contbilinear} implies $\int_0^\infty S(-\tau) \big [w_0 \ast L_0 V (\tau) \big ] \Pi Z(\tau))\, d\tau\, \in L^2_{\omega,x}$ and 
\[
\bigg \Vert  \int_t^\infty S(-\tau) \big [w_0 \ast L_0 V (\tau) \big ] \Pi Z(\tau))\, d\tau\, \bigg \Vert_{L^2_{\omega,x}} \to 0\,,
\]
as $t \to + \infty$. This concludes the proof of \Cref{th:NL}.

\end{proof}

\section*{Acknowledgements}
E.~D. acknowledges that this study was carried out within the INdAM – GNAMPA Project CUP E53C25002010001. A.-S.~dS is supported by Soutien Recherche et Attractivité – PR et DR 2026 from Université Paris-Saclay and by ANR No-Limit ANR-25-CE40-1380. U.~M. acknowledges financial support from MUR – Italian Ministry of University and Research and Next Generation EU within PRIN 2022AKRC5P `Interacting Quantum Systems: Topological Phenomena and Effective Theories'. U.~M. also acknowledges financial support from the European Union through the European Research Council’s Starting Grant \textsc{FermiMath}, grant agreement n. 101040991. Views and opinions expressed are those of the authors and do not necessarily reflect those of the European Union or the European Research Council Executive Agency. Neither the European Union nor the granting authority can be held responsible for them. U.~M. has also been partially supported by \emph{Gruppo Nazionale per la Fisica Matematica} GNFM – INdAM.

\bibliographystyle{siam} 
\small{\bibliography{bibliography}}

\end{document}